\documentclass[11pt, oneside]{amsart}

\usepackage{amsmath,amssymb,verbatim,amsthm,bbm, mathrsfs,amscd,pb-diagram}
\usepackage{aliascnt} 
\usepackage{appendix}
\DeclareMathAlphabet{\mathpzc}{OT1}{pzc}{m}{it}
\usepackage[all]{xy}
\usepackage{multirow,longtable}
\usepackage{float} 
\usepackage{hyperref}
\usepackage{cleveref}
\usepackage{marginnote}
\usepackage{dirtytalk}
\usepackage{framed}

\usepackage{csquotes}

\usepackage{tikz}

\usepackage{multicol}

\usepackage{lscape} 

\usepackage[disable]{todonotes}

\makeatletter
\providecommand\@dotsep{5}
\renewcommand{\listoftodos}[1][\@todonotes@todolistname]{%
  \@starttoc{tdo}{#1}}
\makeatother

\usepackage{float}
\restylefloat{table}

\crefname{table}{table}{tables}
\crefname{listing}{Program-code}{Program-codes}  
\Crefname{listing}{Program-code}{Program-codes}
\crefname{subsection}{subsection}{subsections}

\theoremstyle{plain}
\newtheorem{Conj}{Conjecture}

\newtheorem{Thm}{Theorem}[section]
\newtheorem{Cor}[Thm]{Corollary}
\newtheorem{Prop}[Thm]{Proposition}
\newtheorem{Lem}[Thm]{Lemma}

\theoremstyle{definition}
\newtheorem{Remark}[Thm]{Remark}
\newtheorem{Def}[Thm]{Definition}

\numberwithin{equation}{section}
\newcommand{\Places}{\mathcal{P}} 
\DeclareMathOperator{\zfun}{\xi} 
\DeclareMathOperator{\zint}{\mathcal{Z}} 
\DeclareMathOperator{\Lfun}{\mathcal{L}} 
\newcommand{\st}{\operatorname{\mathfrak{st}}} 
\newcommand{\coset}[1]{\left[ #1 \right]}  
\newcommand{\gen}[1]{\left\langle #1 \right\rangle}  
\newcommand{\FNorm}[1]{\left\vert #1 \right\vert} 
\newcommand{\Nm}{\operatorname{Nm}}

\newcommand{\Ind}{\operatorname{Ind}}

\newcommand{\Gal}{\operatorname{Gal}}
\newcommand{\Aut}{\operatorname{Aut}}

\newcommand{\sgn}{\operatorname{sgn}}

\newcommand{\Res}{\operatorname{Res}}

\newcommand{\C}{\mathbb C}
\newcommand{\A}{\mathbb{A}}
\newcommand{\Q}{\mathbb{Q}}

\newcommand{\Z}{\mathbb{Z}}
\newcommand{\R}{\mathbb{R}}
\newcommand{\N}{\mathbb{N}}

\newcommand{\mO}{\mathcal{O}}

\newcommand {\integral}[1]{\int\limits_{{#1}\bk{F}\backslash {#1}\bk{\A}}}
\newcommand{\bk}[1]{\left(#1\right)} 
\newcommand{\bm}{\begin{multline*}}
\newcommand{\tu}{\end  {multline*}}

\DeclareMathOperator{\Id}{\mathbbm{1}} 
\newcommand{\Ga}{\mathbb{G}_a} 
\newcommand{\Gm}{\mathbb{G}_m} 
\newcommand{\Eisen}{\mathcal{E}}
\DeclareMathOperator{\unif}{\varpi} 
\newcommand{\modf}[1]{\mathcal{\delta}_{#1}} 
\renewcommand{\check}[1]{#1 ^{\vee}} 
\DeclareMathOperator{\Real}{\mathfrak{Re}} 
\newcommand{\piece}[1]{\left\{\begin{matrix} #1 \end{matrix}\right.} 
\newcommand{\set}[1]{\left\{ #1 \right\}} 
\newcommand{\mvert}{\mathrel{}\middle\vert\mathrel{}} 
\newcommand{\res}[1]{\Bigg\vert_{#1}}
\newcommand{\suml}{\sum\limits}
\newcommand{\prodl}{\, \prod\limits}
\newcommand{\intl}{\int\limits}
\newcommand{\ldual}[1]{{^L}#1}
\newcommand{\rmod}{/}
\newcommand{\lmod}{\backslash}

\newcommand{\placestimes}{\displaystyle\operatorname*{\otimes}_{\nu\in\Places}}

\makeatletter
\def\imod#1{\allowbreak\mkern10mu({\operator@font mod}\,\,#1)}
\makeatother

\makeatletter
\renewcommand\section{\@startsection{section}{1}{\z@}%
                                  {-3.5ex \@plus -1ex \@minus-.2ex}%
                                  {2.3ex \@plus.2ex}%
                                  {\center\normalfont\large\bfseries}}
\makeatother

\makeatletter
\renewcommand\subsection{\@startsection{subsection}{2}{\z@}%
	{-3.5ex \@plus -1ex \@minus-.2ex}%
	{2.3ex \@plus.2ex}%
	{\normalfont\large\bfseries}}
\makeatother

\makeatletter
\renewcommand\subsubsection{\@startsection{subsubsection}{3}{\z@}%
	{-3.5ex \@plus -1ex \@minus-.2ex}%
	{2.3ex \@plus.2ex}%
	{\normalfont\large\bfseries}}
\makeatother

\makeatletter
\newtheorem*{rep@theorem}{\rep@title} \newcommand{\newreptheorem}[2]{%
\newenvironment{rep#1}[1]{%
\def\rep@title{\bf #2 \ref{##1} }%
\begin{rep@theorem} }%
{\end{rep@theorem} } }
\makeatother
\newreptheorem{theorem}{Theorem}
\newreptheorem{lemma}{Lemma}

\protected\def\ignorethis#1\endignorethis{}
\let\endignorethis\relax
\def\TOCstop{\addtocontents{toc}{\ignorethis}}
\def\TOCstart{\addtocontents{toc}{\endignorethis}}

\usepackage{colortbl}
\usepackage{zref-savepos}
\usepackage{pict2e}

\newcounter{NoTableEntry}
\renewcommand*{\theNoTableEntry}{NTE-\the\value{NoTableEntry}}

\makeatletter
\newcommand*{\notableentry}{%
  \kern-\tabcolsep
  \stepcounter{NoTableEntry}%
  \vadjust pre{\zsavepos{\theNoTableEntry t}}
  \vadjust{\zsavepos{\theNoTableEntry b}}
  \zsavepos{\theNoTableEntry l}
  \raisebox{%
    \dimexpr\zposy{\theNoTableEntry b}sp
    -\zposy{\theNoTableEntry l}sp\relax
  }[0pt][0pt]{%
    \setlength{\unitlength}{1pt}%
    \edef\w{%
      \strip@pt\dimexpr\zposx{\theNoTableEntry r}sp%
      -\zposx{\theNoTableEntry l}sp\relax
    }%
    \edef\h{%
      \strip@pt\dimexpr\zposy{\theNoTableEntry t}sp%
      -\zposy{\theNoTableEntry b}sp\relax
    }%
    \ifdim\w pt=0pt 
    \else
      \begin{picture}(0,0)%
        \edef\x{%
          \noexpand\put(0,0){\noexpand\line(\w,\h){\w}}%
          \noexpand\put(0,\h){\noexpand\line(\w,-\h){\w}}%
        }\x
      \end{picture}%
    \fi
  }%
  \hspace{0pt plus 1filll}%
  \zsavepos{\theNoTableEntry r}
  \kern-\tabcolsep
}

\makeatletter
\providecommand*{\cupdot}{%
	\mathbin{%
		\mathpalette\@cupdot{}%
	}%
}
\newcommand*{\@cupdot}[2]{%
	\ooalign{%
		$\m@th#1\cup$\cr
		\sbox0{$#1\cup$}%
		\dimen@=\ht0 %
		\sbox0{$\m@th#1\cdot$}%
		\advance\dimen@ by -\ht0 %
		\dimen@=.5\dimen@
		\hidewidth\raise\dimen@\box0\hidewidth
	}%
}

\providecommand*{\bigcupdot}{%
	\mathop{%
		\vphantom{\bigcup}%
		\mathpalette\@bigcupdot{}%
	}%
}
\newcommand*{\@bigcupdot}[2]{%
	\ooalign{%
		$\m@th#1\bigcup$\cr
		\sbox0{$#1\bigcup$}%
		\dimen@=\ht0 %
		\advance\dimen@ by -\dp0 %
		\sbox0{\scalebox{2}{$\m@th#1\cdot$}}%
		\advance\dimen@ by -\ht0 %
		\dimen@=.5\dimen@
		\hidewidth\raise\dimen@\box0\hidewidth
	}%
}
\makeatother

\allowdisplaybreaks

\title[Degenerate Eisenstein Series for Quasi-Split Forms of $Spin_8$]{The Degenerate Eisenstein Series Attached to the Heisenberg Parabolic Subgroups of Quasi-Split Forms of $Spin_8$}
\author{Avner Segal${^{1,2}}$}
\address{${^1}$ School of Mathematics, Ben Gurion University of the Negev, POB 653, Be'er Sheva 84105, Israel}
\address{${^2}$ School of Mathematical Sciences, Tel Aviv University, Tel Aviv 69978, Israel}
\email{avners@math.bgu.ac.il, avners@post.tau.ac.il}

\begin{document}

\begin{abstract}
In \cite{MR3284482} and \cite{SegalRSGeneral} a family of Rankin-Selberg integrals were shown to represent the twisted standard $\mathcal{L}$-function $\Lfun\bk{s,\pi,\chi,\st}$ of a cuspidal representation $ \pi$ of the exceptional group of type $G_2$.
These integral representations bind the analytic behavior of this $\mathcal{L}$-function with that of a family of degenerate Eisenstein series for quasi-split forms of $Spin_8$ associated to an induction from a character on the Heisenberg parabolic subgroup.

This paper is divided into two parts.
In part 1 we study the poles of this degenerate Eisenstein series in the right half plane $\Real\bk{s}>0$.
These solves a conjecture made by J. Hundley and D. Ginzburg in \cite{MR3359720}.
In part 2 we use the results of part 1 to give a criterion for $\pi$ to be a {\bf CAP} representation with respect to the Borel subgroup of $G_2$ in terms of the analytic behavior of $\Lfun\bk{s,\pi,\chi,\st}$ at $s=\frac{3}{2}$.
\end{abstract}

\maketitle

\begin{center}
\end{center}

\tableofcontents

\section{Introduction}
In \cite{MR3284482} and \cite{SegalRSGeneral} a family of integrals representing for the standard twisted $\Lfun$-function $\Lfun\bk{s,\pi,\chi,\st}$ of a cuspidal representation $\pi$ of $G_2$ was considered.
These integral representations bind the analytic properties of $\Lfun\bk{s,\pi,\chi,\st}$ with that of a family of degenerate Eisenstein series $\Eisen_E\bk{\chi, f_s, s, g}$ attached to a degenerate principal series induced from a character of the Heisenberg parabolic subgroup of a quasi-split form $H_E$ of $Spin_8$.
In this paper, we study the possible poles of $\Eisen_E\bk{\chi, f_s, s, g}$ and draw a few corollaries connecting the analytic properties of $\Lfun\bk{s,\pi,\chi,\st}$ and properties of $\pi$.

More precisely, let $F$ be a number field and $\A_F$ be its ring of adeles.
The isomorphism classes of quasi-split forms of $Spin_8$ are parametrized by \'etale cubic algebras over $F$.
Such an algebra $E$ has one of the following types.
\begin{enumerate}
\item $E=F\times F\times F$, this is called the split case.
\item $E=F\times K$, where $K$ is a quadratic field extension of $F$.
\item $E$ is a cubic field extension of $F$, either Galois or non-Galois.
\end{enumerate}
If $E$ is an \'etale cubic algebra over $F$ which is not a non-Galois field extension, we call it a \emph{Galois \'etale cubic algebra over $F$} and denote by $\chi_E$ the Hecke character of $\A_F^\times$ associated to $E$ by global class field theory.
In particular $\chi_{F\times F\times F}=\Id$ and $\chi_{F\times K}=\chi_K$.

For an \'etale cubic algebra $E$ over $F$ there exists a quasi split form of $Spin_8$ denoted by $H_E$.
We fix the Heisenberg parabolic subgroup $P_E=M_E\cdot U_E$ with Levi subgroup $M_E$ and unipotent radical $U_E$.
Since
\[
M_E\cong \bk{\Res_{E\rmod F} GL_2}^0 = \set{g\in \Res_{E\rmod F} GL_2 \mvert \det\bk{g}\in\Gm } ,
\]
a determinant $\operatorname{det}_{M_E}$ of $M_E$ is defined.
For a Hecke character $\chi$ of $F^\times\lmod \A_F^\times$ we form the unnormalized parabolic induction
\begin{equation}
I_{P_E}\bk{\chi,s} = \Ind_{P_E\bk{\A_F}}^{H_E\bk{\A_F}} \bk{\chi\otimes \FNorm{\cdot}^{s+\frac{5}{2}}} \circ \operatorname{det}_{M_E} .
\end{equation}
For a standard section $f_s\in I_{P_E}\bk{\chi,s}$ we define the following Eisenstein series
\[
\Eisen_E\bk{\chi, f_s, s, g} = \suml_{\gamma\in P_E\bk{F}\lmod H_E\bk{F}} f_s\bk{\gamma g} .
\]
This series converges for $\Real\bk{s}\gg 0$ and admits a meromorphic continuation to the whole complex plane.
We say that $\Eisen_E\bk{\chi, \cdot , s, \cdot}$ admits a pole of order $n$ at $s_0$ if
\begin{equation}
\sup \set{\operatorname{ord}_{s=s_0} \Eisen_E\bk{\chi, f_s, s, g} \mvert f_s\in I_{P_E}\bk{\chi,s},\ g\in H_E\bk{\A_F}} = n ,
\end{equation}
where the order $\operatorname{ord}_{s=s_0}h\bk{s}$ of a pole of a complex function $h\bk{s}$ at $s_0$ is the unique integer $n$ such that
\[
\lim\limits_{s\to s_0}\bk{s-s_0}^n h\bk{s}\in\C^\times .
\]
In part 1 of this paper we prove the following theorem

\begin{reptheorem}{Thm: Poles of Eisenstein series}
\emph{
The order of the poles of $\Eisen_E\bk{\chi, \cdot , s, \cdot }$ for $\Real\bk{s} > 0$ are bounded by the following numbers:}
\begin{table}[H]
	\begin{tabular}{|c|c|c|c|c|c|c|c|}
		\hline
		& \multicolumn{3}{c|}{$s=\frac{1}{2}$} & \multicolumn{2}{c|}{$s=\frac{3}{2}$} & $s=\frac{5}{2}$ \\ \hline
		& $\chi=\Id$ & $\chi=\chi_E$ & $\chi\neq \Id, \chi_E$ quad. & $\chi=\Id$ & $\chi=\chi_E$ &  $\chi=\Id$ \\ 
		\hline
		$E=F\times F\times F$ & \multicolumn{2}{c|}{3} & 1 & \multicolumn{2}{c|}{2} & 1 \\ \hline
		$E=F\times K$ & 2 & 2 & 1 & 1 & 1  & 1 \\ \hline
		$E$ Galois field extension & 1 & 1 & 1 & 0 & 1 & 1 \\ \hline
		$E$ non-Galois field extension & 1 & \notableentry & 1 & 0 & \notableentry & 1 \\ \hline
	\end{tabular}
\end{table}

\emph{
Assuming that for $\nu=\infty$ the degenerate principal series representation $I_{P_E}\bk{\Id_\nu,\frac{1}{2}}$ is generated by the spherical vector, the bounds can be improved as follows:
}
\begin{table}[H]
	\begin{tabular}{|c|c|c|c|c|c|c|c|}
		\hline
		& $s=\frac{1}{2}$ & \multicolumn{2}{c|}{$s=\frac{3}{2}$} & $s=\frac{5}{2}$ \\ 
		\hline
		& $\chi$ quad. & $\chi=\Id$ & $\chi=\chi_E$ &  $\chi=\Id$ \\ \hline
		$E=F\times F\times F$  & 1 & \multicolumn{2}{c|}{2} & 1 \\ \hline
		$E=F\times K$ & 1  & 1  & 1 & 1 \\ \hline
		$E$ Galois field extension  & 1 & 0 & 1 & 1 \\ \hline
		$E$ non-Galois field extension  & 1 & 0 & \notableentry & 1 \\ \hline
	\end{tabular}
	\caption{Bounds on the Order of Poles of $\Eisen_E\bk{\chi, f_s, s, g}$}
	\label{Poles of Eisenstein Series}
\end{table}
\emph{
And the orders in the last table are attained by some section.
In particular, when $\chi$ is everywhere unramified a pole of the above-mentioned order is attained by the spherical vector.
For any other triple $\bk{E,\chi,s_0}$, not appearing in the table, with $\Real\bk{s_0}\geq 0$ the degenerate Eisenstein series $\Eisen_E\bk{\chi, f_s, s, g}$ is holomorphic at $s_0$.
}
\end{reptheorem}

Furthermore, we find out which of the associated residual representations are square-integrable.
The residual representation at $s=\frac{5}{2}$ with $\chi=\Id$ is the trivial representation.
The residual representation at $s=\frac{3}{2}$ is computed in \cite{MR1918673} for $\chi=\chi_E$ and in \cite{RallisSchiffmannPaper} for $E=F\times K$ and $\chi=\Id$.
The study of the residual representation at $s=\frac{1}{2}$ for various $\chi$ is a work in progress.
In the case where $E\rmod F$ is a cubic field extension, the study of the (non-degenerate) residual spectrum is carried out in \cite{LaoResidualSpectrum}.


Part 1 of this paper is dedicated to the proof of \Cref{Thm: Poles of Eisenstein series}.
In order to determine the orders of the poles of $\Eisen_E\bk{\chi,f_s,s,g}$ we compute its constant term along the Borel subgroup.
This constant term is a sum of various intertwining operators.
The poles of these intertwining operators are the possible poles of $\Eisen_E\bk{\chi,\cdot,s,\cdot}$.
We proceed to check the possible cancellation of the poles of these intertwining operators.
We note that while usually poles of such intertwining operators are canceled in pairs, it happens that the cancellation of poles here is in triples or quintuples of intertwining operators.


Part 2 of this paper is devoted to applications of \Cref{Thm: Poles of Eisenstein series} to the study of  cuspidal representations of $G_2$.
For a cuspidal representation $\pi=\placestimes \pi_\nu$ of $G_2$  and a Hecke character $\chi=\placestimes \chi_\nu$ as above, we fix a finite set of places $S$ of $F$ such that $\pi_\nu$ and $\chi_\nu$ are unramified for any $\nu\notin S$.
For $\nu\notin S$ we let $t_{\pi_\nu}\in G_2\bk{\C}$ denote the Satake parameter of $\pi_\nu$.
We also fix $\st$ to denote the standard 7-dimensional representation of $G_2\bk{\C}$.
We define the standard twisted partial $\Lfun$-function of $\pi$ to be
\begin{equation}
\Lfun^S\bk{s,\pi,\chi,\st} = \prodl_{\nu\notin S}\frac{1}{\operatorname{det}\bk{1-\chi\bk{\unif_\nu} \st\bk{t_{\pi_\nu}} q_\nu^{-s }}},
\end{equation}
where $\unif_\nu$ is a uniformizer of $F_\nu$ and $q_\nu$ is the cardinality of the residue field of $F_\nu$.

For an irreducible cuspidal form $\varphi\in \pi$ and a standard section $f_s\in I_{P_E}\bk{\chi,s}$ it is proven in \cite{SegalRSGeneral} that
\begin{equation}
\label{eq:IntegralRepresentation}
\intl_{G_2\bk{F}\lmod G_2\bk{\A_F}} \Eisen_E\bk{\chi, f_s, s, g} \varphi\bk{g} dg = \Lfun^S\bk{s,\pi,\chi,\st} d_{E,S}\bk{s,f_S,\varphi_S} .
\end{equation}
Where $d_{E,S}$ is a given meromorphic function.
Furthermore, for any $\pi$ there exists an \'etale cubic algebra $E$ over $F$ such that the integral on the left hand side of \Cref{eq:IntegralRepresentation} is non-zero.
In this case, $d_{E,S}\not\equiv 0$ and for any $s_0\in \C$ one can choose $f_S$, $\varphi_S$ such that $d_{E,S}\bk{s,f_S,\varphi_S}$ is analytic and non-vanishing in a neighborhood of $s_0$.
This proves the meromorphic continuation of $\Lfun\bk{s,\pi,\chi,\st}$ and moreover we have
\begin{equation}
\label{eq: inequality on orders of poles}
ord_{s=s_0}\bk{\Lfun^S\bk{s,\pi,\chi,\st}} \leq
ord_{s=s_0}\bk{\Eisen_E\bk{\chi, \cdot , s, \cdot}} .
\end{equation}

The integral in \Cref{eq:IntegralRepresentation} can be used in order to characterize the image of functorial lifts in terms of the analytic behavior of $\Lfun^S\bk{s,\pi,\chi,\st}$.
Here we apply \Cref{Thm: Poles of Eisenstein series} and \Cref{eq: inequality on orders of poles} to classify \textbf{CAP} representations with respect to the Borel subgroup $B$ of $G_2$.

We recall that a cuspidal representation $\pi$ of $G_2$ is called \textbf{CAP} (\emph{cuspidal associated to parabolic}) with respect to $B$ if there exists an automorphic character $\tau$ of the torus $T$ such that $\pi$ is nearly equivalent to a subquotient of $\Ind_{B\bk{\A_F}}^{G_2\bk{\A_F}} \tau$.

We also recall that non-degenerate characters of the Heisenberg parabolic $P=M\cdot U$ of $G_2$ are parametrized by \'etale cubic algebras over $F$ as explained in \cite{SegalRSGeneral}.
Given a non-degenerate character $\Psi:U\bk{F}\lmod U\bk{\A}\to\C^\times$ we say that a cuspidal representation $\pi$ of $G_2$ supports the $\Psi$-Fourier coefficient if
\[
\exists \varphi\in\pi: \ \intl_{U\bk{F}\lmod U\bk{\A}} \varphi\bk{ug} \overline{\Psi\bk{u}}\ du\not\equiv 0 .
\]
We denote by $\mathcal{WF}\bk{\pi}$ the set of all non-degenerate \'etale cubic algebras $E$ over $F$ such that $\pi$ supports the corresponding Fourier coefficient along $U$.
We call the set $\mathcal{WF}\bk{\pi}$ \emph{the wave front of $\pi$ along $U$}; by \cite[Theorem 3.1]{MR2181091} $\mathcal{WF}\bk{\pi}$ is non-empty.

Given an \'etale cubic algebra $E$ over $F$ we let $S_E=\Aut_F\bk{E}$ and recall the dual reductive pair
\[
G_2\times S_E \hookrightarrow H_E\rtimes S_E .
\]
We denote the corresponding $\theta$-lift from $G_2$ to $S_E$ by $\theta_{S_E}$.
In \Cref{Sec: CAP representation} we prove
\begin{reptheorem}{Thm: CAP representations}
Let $\pi$ be a cuspidal representation of $G\bk{\A}$ supporting a Fourier coefficient along $U$ corresponding to an \'etale cubic algebra $E$ over $F$ which is not a non-Galois field extension.
The following are equivalent:
\begin{enumerate}
\item $\pi$ is a \textbf{CAP} representation with respect to $B$.
\item The partial $\Lfun$-function $\Lfun^S\bk{s,\pi,\chi_E,\st}$ has a pole, of order $2$ if $E=F\times F\times F$ or $1$ otherwise, at $s=2$.
\item 
The $\theta$-lift $\theta_{S_E}\bk{\pi}$ of $\pi$ to $S_E=\Aut_F\bk{E}$ is non-zero.
In particular $\pi$ is nearly equivalent to the $\theta$-lift $\theta_{S_E}\bk{\Id}$, where $\Id$ here is the automorphic trivial representation of $S_E\bk{\A_F}$.
\end{enumerate}
\end{reptheorem}


{\bf Acknowledgments.} 
First and foremost I would like to thank Nadya Gurevich for many helpful discussion during the course of my work on my Ph.D. and especially this paper.

	I would also like to thank Wee Teck Gan and Gordan Savin for a few very helpful discussions, for spotting a few errors in an early version of this manuscript and for providing me with an early version of \cite{Gan-Savin-D4} which provided me with a lot of intuition regarding the location of poles of $\Eisen_E\bk{\chi,f_s,s,g}$.

This work consists of parts from the Ph.D. thesis of the author.
The author was partially supported by grants 1691/10 and 259/14 from the Israel Science Foundation.

{\LARGE \part{The Degenerate Eisenstein Series}}
\label{Part 1} \mbox{}
\section{Definitions and Notations}
Let $F$ be a number field and let $\Places$ be its set of places.
For any $\nu\in\Places$ we denote by $F_\nu$ the completion of $F$ at $\nu$.
If $\nu\nmid\infty$ we denote by $\mO_\nu$ the ring of integers of $F_\nu$, by $\unif_\nu$ a uniformizer of $F_\nu$ and by $q_\nu$ the cardinality of the residue field of $F_\nu$.
We also denote by $\A=\A_F$ the ring of adeles of $F$.
Also, through out this paper we denote the trivial character of $\A^\times$ by $\Id$ and the trivial character of $F_\nu$ by $\Id_\nu$ or $\Id$ if there is no source of confusion.

We also note that, in this paper, parabolic induction $\Ind_P^G$ for a parabolic subgroup $P$ of a group $G$ is unnormalized.

\subsection{Quasi-Split Forms of $Spin_8$}
\label{Subsec: Quasi-Split Forms of D4}

Recall, from \cite[Section 3]{MR546587}, the following parametrization of quasi-split forms of a split simply-connected algebraic group $H$ defined over $F$: 
\[
\set{\text{Quasi-split forms of $H$ over $F$}} \longleftrightarrow \set{\varphi:\Gal\bk{\overline{F}\rmod F}\to\Aut\bk{Dyn\bk{H}}},
\]
where $Dyn\bk{H}$ is the Dynkin diagram of $H$.

The Dynkin diagram of type $D_4$ is given as follows
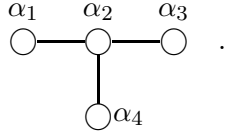
\[
\xygraph{
	!{<0cm,0cm>;<0cm,1cm>:<1cm,0cm>::}
	!{(0,-1)}*{\bigcirc}="1"
	!{(0.4,-1)}*{\alpha_1}="label1"
	!{(0,0)}*{\bigcirc}="2"
	!{(0.4,0)}*{\alpha_2}="label2"
	!{(0,1)}*{\bigcirc}="3"
	!{(0.4,1)}*{\alpha_3}="label3"
	!{(-1,0)}*{\bigcirc}="4"
	!{(-1,0.4)}*{\alpha_4}="label4"
	"1"-"2" "2"-"3" "2"-"4"
} \quad .
\]
We restrict ourselves to the case $H=Spin_8$, the split simply-connected group of type $D_4$.
The quasi-split forms of $H$ were described in \cite{MR2268487}.
Since $\Aut\bk{Dyn\bk{Spin_8}}\cong S_3$ we have
\[
\set{\text{Quasi-split forms of $Spin_8$ over $F$}}
\longleftrightarrow
\set{\varphi:\Gal\bk{\overline{F}\rmod F}\to S_3}.
\longleftrightarrow
\set{\begin{matrix} \text{Isomorphism classes of} \\ \text{\'etale cubic algebras over $F$} \end{matrix}}.
\]
For any cubic algebra $E$ let $S_E=Aut_F(E)$, which is a twisted form of $S_3$.
An action of $S_E$ on the algebraic group $Spin_8$ determines  
a simply-connected quasi-split form $H_E=Spin_8^E$ of the split group $Spin_8$ over $F$. 
We fix a Chevalley-Steinberg system of \'epinglage \cite[Sections 4.1.3-4.1.4]{MR756316}
\[
\set{T_E, B_E, x_\gamma:\Ga\rightarrow\bk{H_E}_\gamma,\gamma\in \Phi_{D_4}} \ ,
\]
where $T_E\subset B_E$ is a maximal torus contained in a Borel subgroup (both defined over $F$) and $\Phi_{D_4}$ are the roots of $H_E\otimes \overline{F}\cong Spin_8\bk{\overline{F}}$.
For any $\gamma$ in the reduced root system of $H_E$ we denote by $F_\gamma$ the field of definition of $\gamma$.

We now recall, from \cite[eq. (1.8)]{MR1637485}, that an \'etale cubic algebra over $F$ is one of the following
\begin{enumerate}
\item $F\times F\times F$.
\item $F\times K$, where $K$ is a quadratic field extension of $F$.
\item $E$, where $E$ is a cubic Galois field extension of $F$.
\item $E$, where $E$ is a cubic non-Galois field extension of $F$.
\end{enumerate}

We call the first three \textbf{Galois \'etale cubic algebras over $F$}.
We also refer to $F\times F\times F$ as the \textbf{split cubic algebra over $F$}.

For a Galois \'etale cubic algebra $E$ we attach an automorphic character $\chi_E$ of $F^\times\lmod\A_F^\times$ as follows:
\begin{enumerate}
	\item If $E=F\times F\times F$ then $\chi_E=\Id$.
	\item If $E=F\times K$ when $K$ is a field then $\chi_E=\chi_K$, where $\chi_K$ is the quadratic automorphic character attached to $K$ by global class field theory.
	\item If $E$ is a field then $\chi_E$ is the cubic automorphic character attached to $E$ by global class field theory.
	Note that $\chi_E^2$ satisfy the same properties, and indeed along this paper all statements regarding $\chi_E$ are also true for $\chi_E^2$.
\end{enumerate}

We now give a more detailed description $H_E$ for the different kinds of \'etale cubic algebras over $F$ in terms of the action of $\Gal\bk{\overline{F}\rmod F}$ on $H_E\bk{\overline{F}}$.

\begin{enumerate}
	\item\underline{$E=F\times F\times F$:}
	In this case $H_E$ is the split reductive simply-connected group of type $D_4$ over $F$.
	It corresponds to the trivial action of $\Gal\bk{\overline{F}\rmod F}$.
	In this case we denote $\Gamma_E=\set{1}$.
	Also, in this case
	\[
	F_{\alpha_1} = F_{\alpha_2} = F_{\alpha_3} = F_{\alpha_4} = F .
	\]
	
	\item\underline{$E=F\times K$:}
	This is the case where $E=F\times K$ with $K$ a quadratic (and hence Galois) extension of $F$.
	It is enough to define an action of $\Gamma_E=\Gal\bk{K\rmod F}=\gen{\sigma}$ on $Spin_8\bk{K}$.
	This action is determined by
	\begin{align*}
	\sigma\bk{x_{\alpha_1}\bk{k}}&=x_{\alpha_1}\bk{\sigma\bk{k}} \\
	\sigma\bk{x_{\alpha_2}\bk{k}}&=x_{\alpha_2}\bk{\sigma\bk{k}} \\
	\sigma\bk{x_{\alpha_3}\bk{k}}&=x_{\alpha_4}\bk{\sigma\bk{k}} \\
	\sigma\bk{x_{\alpha_4}\bk{k}}&=x_{\alpha_3}\bk{\sigma\bk{k}} .
	\end{align*}
	In this case
	\[
	F_{\alpha_1} = F_{\alpha_2} = F, \quad F_{\alpha_3} = F_{\alpha_4} = K .
	\]
	
	Here, we single out the root $\alpha_1$ from $\alpha_3$ and $\alpha_4$.
	
	\item\underline{$E$ is a cubic Galois field extension:}
	It is enough to define an action of $\Gamma_E=\Gal\bk{E\rmod F}=\gen{\sigma\mvert \sigma^3=1}$ on $Spin_8\bk{E}$.
	This action is determined by
	\begin{align*}
	\sigma\bk{x_{\alpha_2}\bk{e}}&=x_{\alpha_2}\bk{\sigma\bk{e}} \\
	\sigma\bk{x_{\alpha_1}\bk{e}}&=x_{\alpha_3}\bk{\sigma\bk{e}} \\
	\sigma\bk{x_{\alpha_3}\bk{e}}&=x_{\alpha_4}\bk{\sigma\bk{e}} \\
	\sigma\bk{x_{\alpha_4}\bk{e}}&=x_{\alpha_1}\bk{\sigma\bk{e}} .
	\end{align*}	
	In this case
	\[
	F_{\alpha_2} = F, \quad F_{\alpha_1} = F_{\alpha_3} = F_{\alpha_4} = E .
	\]
	
	\item \underline{$E$ is a cubic non-Galois field extension:}
	Here we assume that $E$ is a cubic non-Galois extension of $F$. In order to define $H_E\bk{F}$ we first consider the Galois closure $L$ of $E$ over $F$, this is a Sextic Galois extension with $\Gal\bk{L\rmod F}=\gen{\sigma,\ \tau\mvert \sigma^3=1,\ \tau^2=1,\ \bk{\sigma\tau}^2=1}$.
	Note that $L$ is also a Galois extension of $E$.
	We consider the following tower of extensions
	
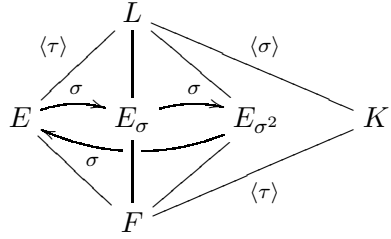
\[
	\xymatrix{
	&L \ar@{-}[ld]_{\gen{\tau}} \ar@{-}[d] \ar@{-}[rd] \ar@{-}[rrd]^{\gen{\sigma}} \\
	E \ar@{-}[rd] \ar@/^/[r]^{\sigma}& E_{\sigma} \ar@{-}[d] \ar@/^/[r]^{\sigma} & E_{\sigma^2} \ar@{-}[ld] \ar@/^1pc/[ll]^-(0.7){\sigma}|!{[l];[dl]}\hole & K \ar@{-}[lld]^{\gen{\tau}} \\
	& F }
	\]
	Where $K=L^{\gen{\sigma}}$ and $E$, $E_\sigma=L^{\gen{\sigma\tau\sigma^2}}$ and $E_{\sigma^2}=L^{\gen{\sigma^2\tau\sigma}}$ are the $\sigma$-conjugates of $E$ in $L$.
	
	
	The action of $\Gamma_E=\Gal\bk{L\rmod F}$ on $Spin_8\bk{L}$ is determined by
	\begin{align*}
	\sigma\bk{x_{\alpha_2}\bk{l}}&=x_{\alpha_2}\bk{\sigma\bk{l}}, \quad \tau\bk{x_{\alpha_2}\bk{l}}=x_{\alpha_2}\bk{\tau\bk{l}} \\
	\sigma\bk{x_{\alpha_1}\bk{l}}&=x_{\alpha_3}\bk{\sigma\bk{l}}, \quad \tau\bk{x_{\alpha_1}\bk{l}}=x_{\alpha_3}\bk{\tau\bk{l}} \\
	\sigma\bk{x_{\alpha_3}\bk{l}}&=x_{\alpha_4}\bk{\sigma\bk{l}}, \quad \tau\bk{x_{\alpha_3}\bk{l}}=x_{\alpha_1}\bk{\tau\bk{l}} \\
	\sigma\bk{x_{\alpha_4}\bk{l}}&=x_{\alpha_1}\bk{\sigma\bk{l}}, \quad \tau\bk{x_{\alpha_4}\bk{l}}=x_{\alpha_4}\bk{\tau\bk{l}} .
	\end{align*}
	Here we singled out $\alpha_4$ from $\alpha_1$ and $\alpha_3$ this is akin to distinguishing $\tau$ from $\sigma\tau\sigma^2$ and $\sigma^2\tau\sigma$.
	In this case
	\[
	F_{\alpha_2} = F, \quad F_{\alpha_1} = E, \quad F_{\alpha_3} = E^\sigma,\quad F_{\alpha_4} = E^{\sigma^2} .
	\]
	
\end{enumerate}

For any root $\alpha$ in the Chevalley-Steinberg system of $H_E$ described above, we denote by $L_\alpha$ the field of definition of $\alpha$.
We denote the cardinality of the residue field of $L_\alpha$ by $q_\alpha$.

Let $P_E$ be the (standard) Heisenberg parabolic subgroup of $H_E$ with Levi decomposition $P_E=M_E\cdot U_E$ such that
\[
M_E\cong \bk{\Res_{E\rmod F}GL_2}^0 = \set{g\in \Res_{E\rmod F}GL_2\mvert \operatorname{det}\bk{g}\in \Gm}
\]
is generated by the simple roots $\alpha_1$, $\alpha_3$ and $\alpha_4$.
In particular, the determinant character $\operatorname{det}_{M_E}$ associated to the Levi subgroup $M_E$ is well defined over $F$.
Restricted to the torus, $det_{M_E}\res{T_E}$ equals the highest root in $\Phi_{H_E}$.
For a more detailed account on the structure of $H_E$ please consider \cite[Section 2]{SegalRSGeneral}.

\subsection{The Degenerate Eisenstein Series}

Fix a finite order Hecke character $\chi:F^\times\lmod\A^\times\to \C^\times$.
We consider the following induced representation
\begin{equation}
I_{P_E}\bk{\chi,s} = \Ind_{P_E\bk{\A}}^{H_E\bk{\A}} \bk{\chi\circ{\operatorname{det}}_{M_E}} \otimes \FNorm{{\operatorname{det}}_{M_E}}^{s+\frac{5}{2}},
\end{equation}
where, as mentioned above, the induction on the right hand side is unnormalized.
We note that $\FNorm{{\operatorname{det}}_{M_E}}^{\frac{5}{2}}$ is the normalization factor \footnote{For the modulus character of $P_E$ it holds that $\modf{P_E}\res{M_E}=\FNorm{\operatorname{det}_{M_E}}^5$.} of the parabolic induction and hence the induced representation on the left hand side is normalized.

For any holomorphic section $f_s\in I_{P_E}\bk{\chi,s}$ we define the following degenerate Eisenstein series
\begin{equation}
\Eisen_E\bk{\chi, f_s, s, g} = \suml_{\gamma\in P_E\bk{F}\lmod H_E\bk{F}} f_s\bk{\gamma g} .
\end{equation}

This series converges for $\Real\bk{s}\gg0$ and admits a meromorphic continuation to the whole complex plane.
For any $s_0\in\C$ we write the Laurent expansion of $\Eisen_E\bk{\chi, f_s, s, g}$:
\[
\Eisen_E\bk{\chi, f_s, s, g} = \suml_{k=-\infty}^\infty \bk{s-s_0}^k \coset{\Lambda_k\bk{\chi,s_0}f_s}\bk{g},
\]
where for each $k$ the coefficient $\Lambda_k\bk{\chi,s_0}$ is an intertwining map
\[
\Lambda_k\bk{\chi,s_0}:I_{P_E}\bk{\chi,s_0} \to Im\bk{\Lambda_{k-1}\bk{\chi,s_0}} \mathcal{A}\bk{H_E},
\]
where $\mathcal{A}\bk{H_E}$ is the space of automorphic forms of $H_E\bk{\A}$.
In particular, if the order of $\Eisen_E\bk{\chi, \cdot, s, \cdot}$ at $s_0$ is $n$ then $\Lambda_{-n}\bk{\chi,s_0}$ is an intertwining map from $I_{P_E}\bk{\chi,s_0}$ to $\mathcal{A}\bk{H_E}$.

Part 1 of this paper is devoted to the study of the analytic properties of $\Eisen_E\bk{\chi, f_s, s, g}$ in the right half-plane $\Real\bk{s}>0$.

\section{Background Theory on Eisenstein Series and Intertwining Operators}

In this section we recall some general information regarding the theory of Eisenstein series.
Most of the results quoted in this section can be found in \cite{MR1361168}.

\subsection{Intertwining Operators and the Constant Term}

We start by noting that
\begin{equation}
I_{P_E}\bk{\chi,s} \hookrightarrow I_{B_E}\bk{\chi_s} = \Ind_{B_E\bk{\A}}^{H_E\bk{\A}} \modf{B_E}^{\frac{1}{2}}\chi_s,
\end{equation}
where
\[
\chi_s=\modf{B_E}^{-\frac{1}{2}} \otimes \bk{\chi\circ{\operatorname{det}}_{M_E}} \otimes \FNorm{{\operatorname{det}}_{M_E}}^{s+\frac{5}{2}} .
\]
Note that, as above, the induction on the right hand side is unnormalized while the induced representation on the left hand side is normalized.

For any $w\in W$ we define the intertwining operator
\[
M\bk{w,\chi_s}: I_{B_E}\bk{\chi_s} \to I_{B_E}\bk{w^{-1}\cdot \chi_s}
\]
by
\begin{equation}
\label{Eq: definition of degenerate intertwining operator}
M\bk{w,\chi_s}f_s\bk{g} = \intl_{N_E\bk{\A}\cap w^{-1}N_E\bk{\A}w\lmod N_E\bk{\A}} f_s\bk{wng} dn .
\end{equation}
This integral converges for $\Real\bk{s}\gg 0$ and admits a meromorphic continuation to $\C$.
When there is no source of confusion we denote $M\bk{w,\chi_s}$ by $M\bk{w}$ or $M_w$.
We also denote $w_{i_1,...,i_k}$ or $w\coset{i_1,...,i_k}$ for $w_{\alpha_{i_1}}\cdots w_{\alpha_{i_k}}$, where $w_{\alpha_{i}}\in W_{H_E}$ denotes the simple reflection associated with the simple root $\alpha_{i}$.

\begin{Remark}
	Note that the definition here is slightly different from the definition given in \cite{MR1361168}.
	As a consequence, if $w=w'w''$ then a cocycle equation is satisfied:
	\begin{equation}
	\label{Eq: Functional equation of intertwining operators}
	M\bk{w,\chi_s} = M\bk{w'',w'^{-1}\cdot\chi_s}\circ M\bk{w',\chi_s} .
	\end{equation}
\end{Remark}

The constant term of $\Eisen_E\bk{\chi, f_s, s, g}$ along $N_E$ is defined to be
\begin{equation}
\Eisen_E\bk{\chi, f_s, s, t}_{B_E} = \intl_{N_E\bk{F}\lmod N_E\bk{\A}} \Eisen_E\bk{\chi, f_s, s, ut} du \quad \forall t\in T_E\bk{\A}.
\end{equation}
By a standard computation, as in \cite{MR1469105}, we obtain
\begin{equation}
\label{Eq: Constant term}
\Eisen_E\bk{\chi, f_s, s, t}_{B_E} = \suml_{w\in W\bk{P_E,H_E}} \bk{M_w\bk{s}f_s}\res{T_E}\bk{t} \quad \forall t\in T_E\bk{\A},
\end{equation}
where $W\bk{P_E,H_E}=\set{w\in W_{H_E}\mvert w\bk{\alpha_2}>0}$ is a set of distinguished representatives for $P_E\lmod H_E\rmod B_E \cong W_{P_E}\lmod W_{H_E}$, given by the shortest representative of each coset.

\begin{Thm}
	\label{Thm: poles of Eisenstein series are the poles of the constant term}
	The degenerate Eisenstein series $\Eisen_E\bk{\chi, f_s, s, g}$ admits a pole of order $n$ at $\bk{\chi,s_0}$ if and only if its constant term $\Eisen_E\bk{\chi, f_s, s, g}_{B_E}$ admits a pole of order $n$ at $\bk{\chi,s_0}$.
\end{Thm}
Indeed, in \Cref{Sec: Poles of the Eisenstein series} we study the poles of $\Eisen_E\bk{\chi, f_s, s, g}$ via the poles of $\Eisen_E\bk{\chi, f_s, s, g}_{B_E}$, using \Cref{Eq: Constant term}.

\subsection{Rank-one Intertwining Operators and Local Factors}
\label{Subsec: Rank-one}
In many instances, the study of Eisenstein series and intertwining operators relies on reduction to the rank-one case via the functional equation, \Cref{Eq: Functional equation of intertwining operators}.
In this section, we recall some useful facts about the rank-one case and the reduction to it.

We fix a number field extension $L\rmod F$ and let $D_L$ be the discriminant of $L\rmod \Q$. 
Let $\zeta_L\bk{s}$ be the completed $\zeta$-function of $L$.
Following \cite{MR1174424}, we define
\[
\zfun_L\bk{s} = \FNorm{D_L}^{\frac{s}{2}}\zeta_L\bk{s} .
\]
The normalized function $\zfun_L$ then satisfies the functional equation
\begin{equation}
\label{Functional Equation - Zeta Function}
\zfun_L\bk{s}=\zfun_L\bk{1-s} .
\end{equation}

Let $\mathcal{B}=\mathcal{T}\cdot\mathcal{N}$ be the Borel subgroup of $SL_2$ with torus $\mathcal{T}$ and unipotent radical $\mathcal{N}$.
Also let $w_0=\begin{pmatrix}0&-1\\1&0\end{pmatrix}$ be the generator of the Weyl group of $SL_2$.
We recall some facts about the intertwining operator $M_{w_0}$ defined on representations of $SL_2$.
Fix a Hecke character $\sigma = \placestimes \sigma_\nu$ of $\mathcal{T}\bk{\A}$, it can be considered as a representation of $\mathcal{B}\bk{\A}$.
For a section $f_s\in \Ind_{\mathcal{B}\bk{\A}}^{SL_2\bk{\A}}\sigma \modf{\mathcal{B}}^{s+\frac{1}{2}}$ we let
\begin{equation}
M\bk{w_0,\sigma,s}f_s\bk{g} = \intl_{\mathcal{N}\bk{\A}\cap w_0^{-1}\mathcal{N}\bk{\A}w_0\lmod \mathcal{N}\bk{\A}} f_s\bk{w_0ng} dn = \intl_{\mathcal{N}\bk{\A}} f_s\bk{w_0ng} dn.
\end{equation}
This integral converges for $\Real\bk{s}\gg0$ and admits meromorphic continuation to the whole complex plane.
The intertwining operator $M_{w_0}$ is factorizable in the sense that if $f_s=\otimes f_{s,\nu}$ then
$M\bk{w_0,\sigma,s}f_s=\placestimes M\bk{w_0,\sigma_\nu,s}f_{s,\nu}$, where for $\Real\bk{s}\gg 0$
\begin{equation}
M\bk{w_0,\sigma_\nu,s}f_{s,\nu}\bk{g} = \intl_{\mathcal{N}\bk{F_\nu}} f_{s,\nu}\bk{w_0ng} dn.
\end{equation}
This integral admits a meromorphic continuation to $\C$.
For a spherical section $f_{s,\nu}^0$ of $\Ind_{\mathcal{B}\bk{F_\nu}}^{SL_2\bk{F_\nu}}\sigma_\nu \modf{\mathcal{B}}^{s+\frac{1}{2}}$ it holds that
\begin{equation}
\label{Eq: Rank one Gindikin Karpelevich}
M\bk{w_0,\sigma_\nu,s} f_{s,\nu}^0 = \frac{\Lfun_{F_\nu}\bk{2s,\sigma_\nu}}{\Lfun_{F_\nu}\bk{2s+1,\sigma_\nu}} f_{-s,\nu}^0 ,
\end{equation}
where:
\begin{itemize}
	\item
	For $\nu\nmid\infty$
	\[
	\Lfun_{F_\nu}\bk{s,\sigma_\nu} = \frac{1}{1-\sigma_\nu\bk{\unif_\nu}q_\nu^{-s}} ,
	\]
	for a uniformizer $\unif_\nu$ of $L_\nu$ and $q_\nu$ the cardinality of the residue field of $F_\nu$.
	This function is a non-vanishing meromorphic function on $\C$ with simple poles at $s=\frac{\log\bk{\sigma_\nu\bk{\unif_\nu}}+2\pi i n}{\log\bk{q_\nu}}$ for all $n\in\Z$.
	
	\item
	The only finite order characters $\sigma_\nu$ of $\R^\times$ are either the trivial one or the sign character.
	Let
	\[
	\epsilon_\nu = \left\{\begin{matrix} 0,& \sigma_\nu=\Id \\ 1,& \sigma_\nu=\sgn \end{matrix} \right.
	\]
	and
	\[
	\Lfun_{\R}\bk{s,\sigma_\nu} = \pi^{-\frac{s+\epsilon_\nu}{2}}\Gamma\bk{\frac{s+\epsilon_\nu}{2}}.
	\]
	\item
	The only finite order character $\sigma_\nu$ of $\C^\times$ is the trivial one.
	For $n\in\Z$ let
	\[
	\sigma_{n,\nu}\bk{z} = \bk{\frac{z}{\FNorm{z}}}^n .
	\]
	Note that any continuous complex character of $\C^\times$ is of the form $\sigma_n\bk{z}\FNorm{z}^s$ for some $n\in\Z$ and $s\in\C$.
	Let
	\[
	\Lfun_{\C}\bk{s,\sigma_{n,\nu}} = 2\bk{2\pi}^{-\bk{s+\frac{\FNorm{n}}{2}}}\Gamma\bk{s+\frac{\FNorm{n}}{2}} .
	\]
\end{itemize}

Recall that $\Gamma\bk{z}$ is a non-vanishing meromorphic function on $\C$ whose only poles are simple, appearing at the points $z=-n$ for $n\geq 0$.

We fix an additive character $\psi=\placestimes\psi_{\nu}:F\lmod\A\to\C^\times$.
For simplicity, we assume that $\psi_\nu$ has conductor $0$ at all finite places.We also fix a global measure $dx=\prodl_{\nu\in\Places}dx_\nu$ such that
\[
\intl_{\mO_\nu} dx_\nu = 1 \quad \forall\nu\in\Places_\infty .
\]
Let $\epsilon_{L_\nu}\bk{s,\sigma_\nu,\psi_\nu}$ be the local $\epsilon$-factor as defined in  \cite[Corollary 3.7]{MR1990377}.
We recall a few facts regarding $\epsilon_{L_\nu}$:
\begin{itemize}
	\item $\epsilon_{L_{\nu}}\bk{s,\sigma_\nu,\psi_\nu}$ is entire in $s$.
	\item For any finite $\nu$ such that $\sigma_\nu$ is unramified, it holds that $\epsilon_{L_{\nu}}\bk{s,\sigma_\nu,\psi_\nu}=1$.
	\item For any $\nu$ it holds that $\epsilon_{L_{\nu}}\bk{s,\sigma_\nu,\psi_\nu}\epsilon_{L_{\nu}}\bk{1-s,\widetilde{\sigma_\nu},\psi_\nu}=1$.
\end{itemize}
\begin{Remark}
	\label{Eq: Global functional equation of Hecke L-function and epsilon-factor}
	We recall the global functional equation
	\begin{equation}
		\Lfun_{L}\bk{s,\sigma} = \epsilon_L\bk{s,\sigma} \Lfun_L\bk{1-s,\overline{\sigma}},
	\end{equation}
	where
	\[
	\Lfun_{L}\bk{s,\sigma} = \prodl_{\nu\in\Places} \Lfun_{L_\nu}\bk{s,\sigma_\nu},\quad
	\epsilon_{L}\bk{s,\sigma} = \prodl_{\nu\in\Places} \epsilon_{L_\nu}\bk{s,\sigma_\nu,\psi_\nu} .
	\]
	Note that if $\sigma$ is unitary then $\sigma^{-1}=\overline{\sigma}$.
	Also note that fixing a finite subset $S\subset\Places$ such that all data is unramified outside of $S$, it holds that
	\[
	\epsilon_{L}\bk{s,\sigma} = \prodl_{\nu\in S} \epsilon_{L_\nu}\bk{s,\sigma_\nu,\psi_\nu} .
	\]
\end{Remark}

\begin{Remark}
	\label{Rem: Local functional equation and gamma-factors}
	We also recall the local functional equation \cite[eq. 3.26]{MR1990377}
	\begin{equation}
	\label{Eq: Local functional equation}
	\Lfun_{L_\nu}\bk{s,\sigma_\nu} = \frac{\epsilon_{L_\nu}\bk{s,\sigma_\nu,\psi_\nu}}{\gamma_{L_\nu}\bk{s,\sigma_\nu,\psi_\nu}} \Lfun_{L_\nu}\bk{1-s,\sigma_\nu^{-1}} ,
	\end{equation}
	where $\gamma_{L_\nu}\bk{s,\sigma_\nu,\psi_\nu}$ is the local $\gamma$-factor as defined in \cite{MR1990377}.
	In particular
	\[
	\prodl_{\nu\in\Places} \gamma_{L_\nu}\bk{s,\sigma_\nu,\psi_\nu} = 1 .
	\]
\end{Remark}
Studying the analytic behavior of $M\bk{w_0,\sigma_\nu,s}$, we have the following lemma (\cite{MR517138} for $\nu\nmid\infty$ and \cite{MR563369} for $\nu\vert\infty$):
\begin{Lem}
	\label{Lem: Holomorphicity of Intertwining operator over L-factor}
	For any $\sigma_\nu:F_\nu^\times\to\C^\times$ the operator $\frac{1}{\Lfun_L\bk{2s,\sigma_\nu}} M\bk{w_0,\sigma_\nu,s}$ is entire and non-vanishing.
\end{Lem}

The \emph{normalized intertwining operator} is defined to be
\[
N\bk{w_0,\sigma_\nu,s} = \frac{\Lfun_{F_\nu}\bk{2s+1,\sigma_\nu}}{\Lfun_{F_\nu}\bk{2s,\sigma_\nu} \epsilon_{F_\nu}\bk{2s,\sigma_\nu,\psi_\nu}} M\bk{w_0,\sigma_\nu,s}
\]
It follows from \Cref{Eq: Rank one Gindikin Karpelevich} that 
\begin{equation}
N\bk{w_0,\sigma_\nu,s} f_{s,\nu}^0 = f_{-s,\nu}^0
\end{equation}

For $\nu\nmid\infty$ it holds that (from the above and \cite[Section 11]{MR2908042}):
\begin{itemize}
	\item The operator $M_{\nu}\bk{w_0,\sigma_\nu,s}$ is entire for $\sigma_\nu$ ramified.
	The normalized intertwining operator $N\bk{w_0,\sigma_\nu,s}$ admits poles at points where $q_\nu^{2s+1}=\sigma_\nu\bk{\unif_\nu}$, in particular $\Real\bk{s}=-\frac{1}{2}$.
	\item If $\sigma_\nu$ is unramified then $M_{\nu}\bk{w_0,\sigma_\nu,s}$ is meromorphic with a simple poles at $\frac{\log\bk{\sigma_\nu\bk{\unif_\nu}}+2\pi i n}{\log\bk{q_\nu}}$ for all $n\in\Z$.
	The normalized intertwining operators $N\bk{w_0,\sigma_\nu,s}$ admits a unique simple pole at $s=-\frac{1}{2}+\frac{\log\bk{\sigma_\nu\bk{\unif_\nu}}+2\pi i n}{\log\bk{q_\nu}}$ for all $n\in\Z$.
	\item Furthermore, when $\sigma_\nu=\Id$ then $M_{\nu}\bk{w_0,\sigma_\nu,s}$ is not injective at $s=\frac{1}{2}$ and $s=-\frac{1}{2}$.
	The normalized intertwining operator $N\bk{w_0,\sigma_\nu,s}$ is not injective at $s=\frac{1}{2}$ and its residue is not injective at $s=-\frac{1}{2}$.
	\item
	In particular, for $\sigma_\nu=\Id$ we have
	\begin{align*}
	& 0\longrightarrow 1 \longrightarrow \Ind_{\mathcal{B}\bk{F_\nu}}^{SL_2\bk{F_\nu}} 1 \overset{M_{w_0}}{\longrightarrow} St \longrightarrow 0 \\
	& 0\longrightarrow St \longrightarrow \Ind_{\mathcal{B}\bk{F_\nu}}^{SL_2\bk{F_\nu}} \modf{\mathcal{B}}^{1} \overset{M_{w_0}}{\longrightarrow} 1 \longrightarrow 0 .
	\end{align*}
	
	\item Whenever $\sigma_\nu\neq \Id$, the induced representation $\Ind_{\mathcal{B}\bk{F_\nu}}^{SL_2\bk{F_\nu}}\sigma \modf{\mathcal{B}}^{s+\frac{1}{2}}$ is reducible if and only if $\sigma_\nu^2=\Id$ and $s=0$.
	In this case, $\Ind_{\mathcal{B}\bk{F_\nu}}^{SL_2\bk{F_\nu}}\sigma \modf{\mathcal{B}}^{s+\frac{1}{2}}=\pi^{\bk{1}}_\nu \oplus \pi^{\bk{-1}}_\nu$ where $\pi^{\bk{1}}_\nu$ and $\pi^{\bk{-1}}_\nu$ are irreducible and if $\sigma_\nu$ is unramified then $\pi^{\bk{1}}_\nu$ is also unramified.
	On the other hand, $\pi^{\bk{-1}}$ is an irreducible representation unramified with respect to the compact subgroup $d\cdot SL_2\bk{\mO_\nu}\cdot d^{-1}$, where 
	\[d=\begin{pmatrix}1&0\\0&\unif_\nu\end{pmatrix} .\]
	Furthermore, $M_{\nu}\bk{w_0,\sigma_\nu,0}$ is bijective and acts as multiplication by a scalar on $\pi^{\bk{1}}_\nu$ and $\pi^{\bk{-1}}_\nu$.
	The normalized intertwining operator $N_\nu\bk{w_0,\chi\boxtimes\chi, s}$ acts on $\pi^{\bk{\epsilon}}$ as $\epsilon Id$.
\end{itemize}

We now discuss the case $\nu\vert\infty$ (from the above and \cite[Chapters II and VII]{MR1880691}).

If $F_\nu=\R$ then $\Pi_{\epsilon_\nu, s}=\Ind_{\mathcal{B}\bk{F_\nu}}^{SL_2\bk{F_\nu}}\sigma_{\nu}\modf{B}^{s+\frac{1}{2}}$ is reducible if and only if $2s=n\in\Z$ and 
\[
\epsilon_\nu \equiv n+1 \mod{2}
\]
In which case, the decomposition series for $\Pi_{\epsilon_\nu, s}$ is as follows:
\begin{itemize}
	\item
	For $s=0$ it holds that
	\[
	\Pi_{\epsilon_\nu, s} = \mathcal{D}_1^{+} \oplus \mathcal{D}_1^{-},
	\]
	where $\mathcal{D}_1^{+}$ and $\mathcal{D}_1^{-}$ are irreducible representations known as the holomorphic and non-holomorphic limits of discrete series (respectively).
	\item
	For $2s=n\in\N$ it holds that
	\[
	\mathcal{D}_{n-1}^{+} \oplus \mathcal{D}_{n-1}^{-} \hookrightarrow \Pi_{\epsilon_\nu, s} \twoheadrightarrow \Phi_{n-1},
	\]
	where $\Phi_{n-1}$ is the unique irreducible representation of $SL_2\bk{\R}$ of dimension $n-1$ and $\mathcal{D}_{n-1}^{+}$, $\mathcal{D}_{n-1}^{-}$ are the irreducible representations known as the holomorphic and non-holomorphic discrete series of highest weight $n-1$.
	\item
	For $-2s=n\in\N$ it holds that
	\[
	\Phi_{n-1} \hookrightarrow \Pi_{\epsilon_\nu, s} \twoheadrightarrow \mathcal{D}_{n-1}^{+} \oplus \mathcal{D}_{n-1}^{-} .
	\]
\end{itemize}

If $F_\nu=\C$ then $\Pi_{n,s}=\Ind_{\mathcal{B}\bk{F_\nu}}^{SL_2\bk{F_\nu}}\sigma_{n,\nu}\FNorm{\cdot}^s$ is reducible if and only if $n=l-k$ and $4s=2+k+l$ or $n=k-l$ and $4s=-\bk{2+k+l}$ for $k,l\in\N\cup\set{0}$, in which case
\begin{itemize}
	\item
	If $n=l-k$ and $4s=2+k+l$ then
	\[
	\Phi_{k,l} \hookrightarrow \Pi_{n,s} \twoheadrightarrow \mathcal{E}_{n-1}^{+} \oplus \mathcal{E}_{n-1}^{-} ,
	\]
	where $\Phi_{k,l}$ is the finite-dimensional representation realized as polynomials in the complex variables $\bk{z_1,z_2,\overline{z_1},\overline{z_2}}$, homogeneous of degree $k$ in $\bk{z_1,z_2}$ and homogeneous of degree $l$ in $\bk{\overline{z_1},\overline{z_2}}$.
	$\mathcal{E}_{n-1}^{+}$ and $\mathcal{E}_{n-1}^{-}$ are analogous to $\mathcal{D}_{n-1}^{+}$ and $\mathcal{D}_{n-1}^{-}$.
\end{itemize}

The following lemma follows from the above discussion.
\begin{Lem}
	\label{Lem: Local composition of rank-one intertwining operators}
	For a place $\nu\in\Places$, it holds that
	\[
	N\bk{w_0,s,\sigma_\nu} \circ N\bk{w_0, -s,\overline{\sigma_\nu}} =\operatorname{Id} \quad \forall s\in\C .
	\]
	For $s_0\in\R$ such that $N\bk{w_0,s,\sigma_\nu}$ admits a pole at $s_0$ or $-s_0$ this should be understood as
	\[
	\lim_{s\to s_0} N\bk{w_0,s,\sigma_\nu} \circ N\bk{w_0, -s,\overline{\sigma_\nu}} =\operatorname{Id} \quad \forall s\in\C .
	\]
\end{Lem}

We finish the discussion of the rank-one case by recalling two results regarding the global intertwining operator, one of them is the global analog of \Cref{Lem: Local composition of rank-one intertwining operators}.
As $M\bk{w_0,s,\sigma}\circ M\bk{w_0,-s,\sigma}$ is an endomorphism of irreducible representations for all $s\in\C$ such that $2s\notin\Z$, it equals a constant.

\begin{Lem}[\cite{MR0579181}, Lemma 6.3]
	\label{Langlands identity for composition of intertwining operators}
	It holds that
	\[
	M\bk{w_0,s,\sigma} \circ M\bk{w_0, -s,\overline{\sigma}} =\operatorname{Id} \quad \forall s\in\C .
	\]
	For $s=\pm \frac{1}{2}$ with $\sigma=\Id$ this should be understood as
	\[
	\lim_{s\to\pm\frac{1}{2}} M\bk{w_0,s,\sigma} \circ M\bk{w_0, -s,\overline{\sigma}} =\operatorname{Id} .
	\]
\end{Lem}

We would also like to recall \cite[Lemma 1.5]{MR1174424}:
\begin{Lem}
	\label{Lem: Ikeda for SL2}
	For $\sigma=\Id$, the operator $M\bk{w_0,s,\Id}$ is holomorphic at $s_0=0$ and is equal to the scalar multiplication by -1 at $s_0=0$.
\end{Lem}

\subsection{Intertwining Operators for Induced Representations of $H_E$}

At this point, it will be beneficial to consider a more general point of view.
Let $\mathfrak{a}_\C^\ast=X^\ast\bk{T_E}\otimes\C$.
We equip $\mathfrak{a}_\C^{\ast}$ with the following system of coordinates:
\begin{itemize}
	\item If $E=F\times F\times F$ we have $\mathfrak{a}_\C^\ast\cong\C^4$ and we write $\lambda=\bk{s_1,s_2,s_3,s_4}\in\mathfrak{a}_\C^\ast$ for
	\[
	\lambda\bk{h_{\alpha_1}\bk{t_1}h_{\alpha_2}\bk{t_2}h_{\alpha_3}\bk{t_3}h_{\alpha_4}\bk{t_4}} = \FNorm{t_1}_F^{s_1} \FNorm{t_2}_F^{s_2}\FNorm{t_3}_F^{s_3}\FNorm{t_4}_F^{s_4} \quad \forall t_1,t_2,t_3,t_4\in F^\times .
	\]
	\item If $E=F\times K$ we have $\mathfrak{a}_\C^\ast\cong\C^3$ and we write $\lambda=\bk{s_1,s_2,s_3}\in\mathfrak{a}_\C^\ast$ for
	\[
	\lambda\bk{h_{\alpha_1}\bk{t_1}h_{\alpha_2}\bk{t_2}h_{\alpha_3}\bk{t_3}h_{\alpha_4}\bk{t_3^\sigma}} = \FNorm{t_1}_F^{s_1} \FNorm{t_2}_F^{s_2}\FNorm{t_3}_K^{s_3} \quad \forall t_1,t_2\in F^\times, \ \forall t_3\in K^\times .
	\]	
	\item If $E$ is a field we have $\mathfrak{a}_\C^\ast\cong\C^2$ and we write $\lambda=\bk{s_1,s_2}\in\mathfrak{a}_\C^\ast$ for 
	\[
	\lambda\bk{h_{\alpha_1}\bk{t_1}h_{\alpha_2}\bk{t_2}h_{\alpha_3}\bk{t_1^\sigma}h_{\alpha_4}\bk{t_1^{\sigma^2}}} = \FNorm{t_1}_E^{s_1} \FNorm{t_2}_F^{s_2} \quad \forall t_2\in F^\times,\ \forall t_1\in E^\times.
	\]
\end{itemize}
For any finite order character $\chi=\placestimes\chi_\nu$ of $T_E\bk{\A}$ and any $\lambda\in \mathfrak{a}_\C^\ast$ we let
\begin{align*}
& I_{B_E}\bk{\chi,\lambda} = \Ind_{B_E\bk{\A}}^{H_E\bk{\A}} \bk{\chi\circ\operatorname{det}_{M_E}} \cdot\bk{\lambda+\rho_{B_E}} = \placestimes I_{B_E}\bk{\chi_\nu,\lambda} \\
& I_{B_E}\bk{\chi_\nu,\lambda} = \Ind_{B_E\bk{F_\nu}}^{H_E\bk{F_\nu}} \bk{\chi_\nu\circ\operatorname{det}_{M_E}} \cdot \bk{\lambda+\rho_{B_E}} ,
\end{align*}
where $\rho_{B_E}$ is half the sum of positive roots in $H_E$ with respect to $B_E$.
We note that, as above, the induction on the right hand side is unnormalized while the induced representation on the left hand side is normalized.
This is not the most general principal series representation but it will suffice for our needs.
We note that
\[
I_{P_E}\bk{\chi,s} \hookrightarrow I_{B_E}\bk{\chi_s} = I_{B_E}\bk{\chi,\lambda_s},
\]
where
\begin{equation}
\label{Eq: lambda-s}
\lambda_s = \piece{\bk{-1,s+\frac{3}{2},-1,-1},& E=F\times F\times F \\ \bk{-1,s+\frac{3}{2},-1},& E=F\times K \\ \bk{-1,s+\frac{3}{2}}& E\rmod F \text{ is a cubic field extension}} .
\end{equation}
For $w\in W$ and a holomorphic section $f_\lambda\in I_B\bk{\chi,\lambda}$ let
\begin{equation}
M\bk{w,\chi,\lambda}f_\lambda\bk{g} = \intl_{N_E\bk{\A}\cap w^{-1}N_E\bk{\A}w\lmod N_E\bk{\A}} f_\lambda\bk{wng} dn .
\end{equation}
This integral converges absolutely to an analytic function in the positive Weyl chamber
\[
C^{+} = \set{\lambda\in\mathfrak{a}_\C^\ast \mvert \Real\gen{\lambda,\check{\alpha}}>0\ \forall \alpha>0}
\]
and admits a meromorphic continuation to $\mathfrak{a}_\C^\ast$.

\begin{Remark}
	\label{Rem: Connection between two types of intertwining operators}
	Due to the choice of representatives in $W\bk{P_E,H_E}$, the intertwining operators $M\bk{w,\chi_s}$ defined in \Cref{Eq: definition of degenerate intertwining operator} are generically (at points of holomorphicity) restrictions of $M\bk{w,\chi,\lambda}$ to the line $\lambda_s$ as above.
\end{Remark}

Note that by abuse of notation, for a Hecke character $\chi$ we identify $\chi$ and $\chi\circ\operatorname{det}_{M_E}$.

We recall that $M\bk{w,\chi,\lambda}$ and $M\bk{w,\chi_s}$ can be decomposed as follows
\begin{equation}
\begin{split}
& M\bk{w,\chi_\nu,\lambda} = \placestimes M_\nu\bk{w,\chi_{\nu},\lambda} \\
& M\bk{w,\chi_s} = \placestimes M_\nu\bk{w,\chi_{\nu,s}} ,
\end{split}
\end{equation}
where for any $\nu\in\Places$, $\lambda\in C^{+}$ and $\Real\bk{s}\gg 0$, the local intertwining operators $M\bk{w,\chi_\nu,\lambda}$ and $M_\nu\bk{w,\chi_s}$ are defined via
\begin{equation}
\begin{split}
& M\bk{w,\chi_\nu,\lambda} f_{\lambda,\nu}\bk{g} = \intl_{N_E\bk{F_\nu}\cap w^{-1}N_E\bk{F_\nu}w\lmod N_E\bk{F_\nu}} f_{\lambda,\nu}\bk{wng} dn \\
& M\bk{w,\chi_{s,\nu}} f_{s,\nu}\bk{g} = \intl_{N_E\bk{F_\nu}\cap w^{-1}N_E\bk{F_\nu}w\lmod N_E\bk{F_\nu}} f_{s,\nu}\bk{wng} dn .
\end{split}
\end{equation}
These integrals converge for $\lambda\in C^{+}$ and $\Real\bk{s}\gg 0$ respectively and admits a meromorphic continuation to $\mathfrak{a}_\C^\ast$ and $\C$ respectively.

We now recall the connection between the rank-one case and the intertwining operators $M_{w_\alpha}$, where $w_\alpha$ is the simple reflection with respect to a simple root $\alpha$.
For any simple root $\alpha$, we have an embedding $\iota_\alpha:SL_2\to H_E$, defined over $F_\alpha$, so that
\[
\iota_\alpha\bk{\begin{pmatrix}t&0\\0&t^{-1}\end{pmatrix}} = h_\alpha\bk{t},\quad
\iota_\alpha\bk{\begin{pmatrix} 1&x\\0 &1 \end{pmatrix}} = x_\alpha\bk{x},\quad
\iota_\alpha\bk{\begin{pmatrix} 1&0\\x &1 \end{pmatrix}} = x_{-\alpha}\bk{x},\quad 
\iota_\alpha\bk{\begin{pmatrix}0&1\\-1&0\end{pmatrix}} = w_\alpha .
\]
We denote by $T_\alpha$ the image of $h_\alpha$.

\begin{Lem}
	\label{Lemma: intertwining operator of simple reflections}
	The following diagram is commutative
	\[
	\xymatrix{
		I_{B_E}\bk{\chi_{\nu},\lambda} \ar@{->}[r]^{M_\nu\bk{w_\alpha,\chi_\nu,\lambda}} \ar@{->}[d]_{\iota_\alpha^\ast} &
		I_{B_E}\bk{w_\alpha\cdot \chi_\nu, w_\alpha\cdot \lambda} \ar@{->}[d]^{\iota_\alpha^\ast} \\
		\Ind_{\mathcal{B}\bk{F_\nu}}^{SL_2\bk{F_\nu}} \bk{\coset{\chi_{\nu} \otimes \lambda}\res{T_\alpha}} \ar@{->}[r]^{M_{w_0}} &
		\Ind_{\mathcal{B}\bk{F_\nu}}^{SL_2\bk{F_\nu}} \bk{w_\alpha\cdot\coset{\chi_{\nu} \otimes \lambda}\res{T_\alpha}}
		} ,
	\]	
	where the vertical maps should be understood as the pull-back map.
	By restriction to $I_{P_E}\bk{\chi,s}$, this is also true for $M\bk{w_\alpha,\chi_s}$.
\end{Lem}

\begin{proof}
	We note that 
	\[
	N_E\bk{F_\nu}\cap w_\alpha^{-1}N_E\bk{F_\nu}w_\alpha\lmod N_E\bk{F_\nu}= 
	\iota_\alpha \bk{\mathcal{N}\bk{F_\nu}\cap w_0^{-1}\mathcal{N}\bk{F_\nu}w_0\lmod \mathcal{N}\bk{F_\nu}}
	\]
	and that
	\[
	\mathcal{N}\bk{F_\nu}\cap w_0^{-1}\mathcal{N}\bk{F_\nu}w_0\lmod \mathcal{N}\bk{F_\nu} \cong \mathcal{N}\bk{F_\nu} .
	\]
	Consequently, for $f_{s,\nu}\in I_{B_E}\bk{\chi_{\nu,s}}$ and $g\in SL_2\bk{F_\nu}$ it holds that
	\begin{align*}
	M_{w_0}\iota_\alpha^\ast \bk{f_{s,\nu}} \bk{g} & = 
	\intl_{\mathcal{N}\bk{F_\nu}} \iota_\alpha^\ast \bk{f_{s,\nu}} \bk{w_0ng} dn \\
	& = 
	\intl_{\mathcal{N}\bk{F_\nu}} \bk{f_{s,\nu}} \bk{\iota_\alpha\bk{w_0ng}} dn \\
	& = 
	\intl_{N_E\bk{F_\nu}\cap w_\alpha^{-1}N_E\bk{F_\nu}w_\alpha\lmod N_E\bk{F_\nu}} \bk{f_{s,\nu}} \bk{w_\alpha n' \iota_\alpha\bk{g}} dn' \\
	& = \bk{M_\nu\bk{w_\alpha,\chi_{s,\nu}} f}\bk{\iota_\alpha\bk{g}}  = \iota_\alpha^\ast\bk{M_\nu\bk{w_\alpha,\chi_{s,\nu}} f}\bk{g} .
	\end{align*}
\end{proof}

The following is a corollary of the previous lemma and \Cref{Eq: Rank one Gindikin Karpelevich}.
\begin{Cor}[The Gindikin-Karpelevich formula]
	\label{Gindikin-Karpelevich formula}
	Let $\nu\in\Places$ be a place such that $\chi_\nu$ is unramified.
	Also, let $w\in W$.
	
	\begin{itemize}
		\item
		Let $f_\nu^0 \in I_{B_E}\bk{\chi_\nu,\lambda}$ be an unramified vector.
		It then holds that
		\begin{equation}
		M_\nu\bk{w,\chi_\nu,\lambda} f_\nu^0 = \prodl_{\alpha>0,\ w^{-1}\alpha<0} \frac{\Lfun_{F_{\alpha,\nu}}\bk{\gen{\lambda,\check{\alpha}} ,\chi_\nu\circ\operatorname{det}_{M_E}\circ\check{\alpha}}}{\Lfun_{F_{\alpha,\nu}}\bk{\gen{\lambda,\check{\alpha}}+1 ,\chi_\nu\circ\operatorname{det}_{M_E}\circ\check{\alpha}}} f_\nu^0 .
		\end{equation}
		
		\item
		Let $f_\nu^0 \in I_{B_E}\bk{\chi_{s,\nu}}$ be an unramified vector.
		It then holds that
		\begin{equation}
		M_\nu\bk{w,\chi_{s,\nu}} f_\nu^0 = \prod_{\alpha>0,\ w^{-1}\alpha<0} \frac{\Lfun_{F_{\alpha,\nu}}\bk{\chi_{s,\nu}\circ\check{\alpha}}}{\Lfun_{F_{\alpha,\nu}}\bk{q_{\alpha,\nu}^{-1}\chi_{s,\nu}\circ\check{\alpha}}} f_\nu^0 .
		\end{equation}
	\end{itemize}
\end{Cor}
We denote the Gindikin-Karpelevich term by 
\begin{equation}
\begin{split}
& J_\nu\bk{w,\chi,\lambda} = \prod_{\alpha>0,\ w^{-1}\alpha<0} \frac{\Lfun_{F_{\alpha,\nu}}\bk{\gen{\lambda,\check{\alpha}} ,\chi_\nu\circ\operatorname{det}_{M_E}\circ\check{\alpha}}}{\Lfun_{F_{\alpha,\nu}}\bk{\gen{\lambda,\check{\alpha}}+1 ,\chi_\nu\circ\operatorname{det}_{M_E}\circ\check{\alpha}}} \\
& J_\nu\bk{w,\chi_s} = \prod\limits_{\alpha>0,\ w^{-1}\alpha<0} \frac{\Lfun_{F_{\alpha,\nu}}\bk{\chi_{s,\nu}\circ\check{\alpha}}}{\Lfun_{F_{\alpha,\nu}}\bk{q_{\alpha,\nu}^{-1}\chi_{s,\nu}\circ\check{\alpha}}} .
\end{split}
\end{equation}
Denote
\begin{equation}
J\bk{w,\chi_s} = \prod_{\nu\in\Places} J_\nu\bk{w,\chi_s}.
\end{equation}
We list the various Gindikin-Karpelevich terms and their poles in the tables in \Cref{Sec: Tables}.

The following is a corollary of \Cref{Lem: Ikeda for SL2} and \Cref{Lemma: intertwining operator of simple reflections}.
We note that it can also be viewed as the application of \cite[Proposition 6.3]{MR944102} to a simple reflection associated to a simple root.
\begin{Cor}
	\label{Cor: Keys-Shahidi}
	Let $\alpha$ be a simple root and $w_\alpha$ the associated simple reflection and fix $w\in W$.
	Further assume that
	\[
	w_\alpha^{-1}\cdot \coset{w^{-1}\cdot \bk{\lambda_0\otimes\chi\circ\operatorname{det}_{M_E}}} =
	w^{-1}\cdot \bk{\lambda_0\otimes\chi\circ\operatorname{det}_{M_E}} .
	\]
	Then $M\bk{w_\alpha,\Id,\lambda_0}$ is holomorphic at $\lambda_0$ and
	\[
	M\bk{w_\alpha,\Id,\lambda_0}: Ind_{B_E\bk{\A}}^{H_E\bk{\A}} w^{-1}\cdot \bk{\lambda_0\otimes\chi\circ\operatorname{det}_{M_E}} \to Ind_{B_E\bk{\A}}^{H_E\bk{\A}} w^{-1}\cdot \bk{\lambda_0\otimes\chi\circ\operatorname{det}_{M_E}}
	\]
	acts as a scalar multiplication by $-1$.
\end{Cor}


\subsection{Normalized Intertwining Operators}

It is customary to define the \emph{normalized intertwining operator} to be
\begin{equation}
N_\nu\bk{w,\chi_\nu,\lambda} =
\frac{J_\nu\bk{w,\chi_s}^{-1}}{\prodl_{\alpha>0,\ w^{-1}\alpha<0} \epsilon_{F_{\alpha,\nu}}\bk{\gen{\lambda,\check{\alpha}},\chi_\nu\circ\operatorname{det}_{M_E}\circ\check{\alpha}, \psi_\nu}}
M\bk{w,\chi_\nu,\lambda} .
\end{equation}

\begin{Lem}
	The normalized intertwining operator satisfy the local functional equation
	\[
	N_\nu\bk{ww',\chi_\nu,\lambda} = N_\nu\bk{w',w^{-1}\cdot\chi_\nu,w^{-1}\cdot\lambda}\circ N_\nu\bk{w,\chi_\nu,\lambda}
	\quad \forall w,w'\in W_{H_E}
	\]
\end{Lem}
For simplicity we write:
\[
N_\nu\bk{w,\chi_s} = N_\nu\bk{w,\chi_\nu,\lambda_s} .
\]

By \Cref{Gindikin-Karpelevich formula} and \Cref{Eq: Global functional equation of Hecke L-function and epsilon-factor}, it holds that
\begin{equation}
\label{Eq: Global Gindikin-Karpelevich}
\begin{split}
M\bk{w,\chi_s} f_{s}
& = \bk{\displaystyle\operatorname*{\otimes}_{\nu\in S}M_\nu\bk{w,\chi_s} f_{s,\nu} }\bigotimes \bk{\displaystyle\operatorname*{\otimes}_{\nu\notin S} J_\nu\bk{w,\chi_s} f_{s,\nu}^0} \\
& = J\bk{w,\chi_s} \bk{\displaystyle\operatorname*{\otimes}_{\nu\in S}  J_\nu\bk{w,\chi_s}^{-1} M_\nu\bk{w,\chi_s}f_{s,\nu} }\bigotimes \bk{\displaystyle\operatorname*{\otimes}_{\nu\notin S} f_{s,\nu}^0} \\
& = \bk{\prodl_{\alpha>0,\ w^{-1}\alpha<0} \epsilon_{F_\alpha}\bk{\gen{\lambda,\check{\alpha}},\chi_\nu\circ\operatorname{det}_{M_E}\circ\check{\alpha}}} J\bk{w,\chi_s} \bk{\displaystyle\operatorname*{\otimes}_{\nu\in S}  N_\nu\bk{w,\chi_s}f_{s,\nu} }\bigotimes \bk{\displaystyle\operatorname*{\otimes}_{\nu\notin S} f_{s,\nu}^0} \ .
\end{split}
\end{equation}
Hence the analytic behavior of $M\bk{w,\chi_s} f_{s}\bk{g}$ is governed by that of $J\bk{w,\chi_s}$ and $N_\nu\bk{w,\chi_s}f_{s,\nu}\bk{g}$ for $\nu\in S$.
Note that according to \Cref{Lem: Holomorphicity of Intertwining operator over L-factor}, $N_\nu\bk{w,\chi_s}f_{s,\nu}$ is holomorphic whenever $\Real\bk{\gen{\chi_s,\check{\alpha}}}>-1$ for all $\alpha>0$ such that $w\cdot\alpha<0$.
In light of Tables
\ref{Constant Term along the Borel, Cubic, General Character}, \ref{Constant Term along the Borel, Quadratic, Gneral Character}, and \ref{Constant Term along the Borel, Split, Gneral Character} and the discussion in \Cref{Subsec: Rank-one} the following holds:

\begin{Lem}
	\label{Lem: Normalized intertwining oeprator is non-vanishing}
	For any $\Real\bk{s_0}>0$ and $\nu\in\Places$ it holds that	$N\bk{w,\chi_{s,\nu}} f_{s,\nu}$
	is analytic at $s_0$.
	Moreover, there exists an $f_{s,\nu}$ such that $N\bk{w,\chi_{s,\nu}} f_{s,\nu}$ is non-zero at $s_0$.
\end{Lem}

\section{Poles of the Eisenstein Series}
\label{Sec: Poles of the Eisenstein series}

In this section we make use of \Cref{Eq: Constant term} to study the poles of $\Eisen_E\bk{\chi, f_s, s, g}_{B_E}$.
By \Cref{Thm: poles of Eisenstein series are the poles of the constant term}, these are the poles of $\Eisen_E\bk{\chi, f_s, s, g}$.
We start by considering the poles of the various intertwining operators, thus getting a bound on the order of the poles.
In the following table we list the possible triples $\bk{E,\chi,s_0}$ for which $\Eisen_E\bk{\chi,f_s,s,g}$ might admit a pole at $s_0$ and give bounds on the orders of the poles at these points.
Here $E$ is an \'etale cubic algebra over $F$, $\chi$ is a Hecke character of $F^\times\lmod \A^\times$ and $s_0\in\C$ with $\Real\bk{s_0}>0$.
More precisely, due to \Cref{Thm: poles of Eisenstein series are the poles of the constant term} and \Cref{Eq: Constant term}, for a given \'etale cubic algebra $E$ over $F$ and a finite order automorphic character $\chi$ we have
\[
\set{\text{Poles of } \Eisen_E\bk{\chi,\cdot,s,\cdot}} =
\set{\text{Poles of } \Eisen_E\bk{\chi,\cdot,s,\cdot}_{B_E}} \subseteq \set{\text{Poles of } M\bk{w,\chi_s} \mvert w\in W\bk{P_E,H_E}} .
\]
We note that for $\Real\bk{s}>0$ the poles of $M\bk{w,\chi_s}$ for various values of $w\in W\bk{P_E,H_E}$ and $\chi$ can occur only at $s_0\in\set{\frac{1}{2}, \frac{3}{2}, \frac{5}{2}}$.
For such triples $\bk{E,\chi,s_0}$, the following table lists
\[
\max\set{\operatorname{ord}_{s=s_0} M\bk{w,\chi_s}f_s\bk{g} \mvert w\in W\bk{P_E,H_E},\ f_s\in I_{P_E}\bk{\chi,s},\ g\in H_E\bk{\A}} .
\]
If this has positive value, we list this value in Table \ref{Trivial Bounds on the order of Poles of Eisenstein Series} in the cell corresponding  to $\bk{E,\chi,s_0}$.
The orders of poles of the intertwining operators are given by Tables \ref{Constant Term along the Borel, Cubic}, \ref{Constant Term along the Borel, Quadratic} and \ref{Constant Term along the Borel, Split} in \Cref{Sec: Tables}.

Note that whenever $E\rmod F$ is a non-Galois field extension the character $\chi_E$ is not defined.
Indeed, if $E\rmod F$ is a non Galois extension and $\chi\circ\Nm_{E\rmod F}=\Id$ then $\chi=\Id$.

\begin{table}[H]
	\begin{tabular}{|c|c|c|c|c|c|c|c|}
		\hline
		& \multicolumn{3}{c|}{$s=\frac{1}{2}$} & \multicolumn{2}{c|}{$s=\frac{3}{2}$} & $s=\frac{5}{2}$ \\ \hline
		& $\chi=\Id$ & $\chi=\chi_E$ & $\chi\neq 1, \chi_E$ and $\chi^2=\Id$ & $\chi=\Id$ & $\chi=\chi_E$ &  $\chi=\Id$ \\ 
		\hline
		$E=F\times F\times F$ & \multicolumn{2}{c|}{4} & 1 & \multicolumn{2}{c|}{3} & 1 \\ \hline
		$E=F\times K$ & 3 & 2 & 1 & 2 & 1  & 1 \\ \hline
		$E$ Galois field extension & 2 & 1 & 1 & 1 & 1 & 1 \\ \hline
		$E$ non-Galois field extension & 2 & \notableentry & 1 & 1 & \notableentry & 1 \\ \hline
	\end{tabular}
	\caption{Trivial Bounds on the Order of Poles of $\Eisen_E\bk{\chi, f_s, s, g}$}
	\label{Trivial Bounds on the order of Poles of Eisenstein Series}
\end{table}

\begin{Thm}
	\label{Thm: Poles of Eisenstein series}
	The order of the poles of $\Eisen_E\bk{\chi, \cdot , s, \cdot }$ for $\Real\bk{s} > 0$ are bounded by the following numbers:
	\begin{table}[H]
		\begin{tabular}{|c|c|c|c|c|c|c|c|}
			\hline
			& \multicolumn{3}{c|}{$s=\frac{1}{2}$} & \multicolumn{2}{c|}{$s=\frac{3}{2}$} & $s=\frac{5}{2}$ \\ \hline
			& $\chi=\Id$ & $\chi=\chi_E$ & $\chi\neq 1, \chi_E$ and $\chi^2=\Id$ & $\chi=\Id$ & $\chi=\chi_E$ &  $\chi=\Id$ \\ 
			\hline
			$E=F\times F\times F$ & \multicolumn{2}{c|}{3} & 1 & \multicolumn{2}{c|}{2} & 1 \\ \hline
			$E=F\times K$ & 2 & 2 & 1 & 1 & 1  & 1 \\ \hline
			$E$ Galois field extension & 1 & 1 & 1 & 0 & 1 & 1 \\ \hline
			$E$ non-Galois field extension & 1 & \notableentry & 1 & 0 & \notableentry & 1 \\ \hline
		\end{tabular}
	\end{table}
	
	Assuming that for $\nu=\infty$ the degenerate principal series representation $I_{P_E}\bk{\Id_\nu,\frac{1}{2}}$ is generated by the spherical vector, the orders of these poles are bounded by the following numbers:
\begin{table}[H]
	\begin{tabular}{|c|c|c|c|c|c|c|c|}
		\hline
		& $s=\frac{1}{2}$ & \multicolumn{2}{c|}{$s=\frac{3}{2}$} & $s=\frac{5}{2}$ \\ 
		\hline
		& $\chi^2=\Id$ & $\chi=\Id$ & $\chi=\chi_E$ &  $\chi=\Id$ \\ \hline
		$E=F\times F\times F$  & 1 & \multicolumn{2}{c|}{2} & 1 \\ \hline
		$E=F\times K$ & 1  & 1  & 1 & 1 \\ \hline
		$E$ Galois field extension  & 1 & 0 & 1 & 1 \\ \hline
		$E$ non-Galois field extension  & 1 & 0 & \notableentry & 1 \\ \hline
	\end{tabular}
	\caption{Bounds on the Order of Poles of $\Eisen_E\bk{\chi, f_s, s, g}$}
	\label{Poles of Eisenstein Series}
\end{table}
	And the orders in the table are attained by some section.
	In particular, when $\chi$ is everywhere unramified a pole of the above-mentioned order is obtained for the spherical vector.
	For any triple $\bk{E,\chi,s_0}$ in the table, the pole of the prescribed order is attained.
	For any other triple $\bk{E,\chi,s_0}$, not appearing in the table, with $\Real\bk{s_0}\geq 0$ the degenerate Eisenstein series $\Eisen_E\bk{\chi, f_s, s, g}$ is holomorphic at $s_0$.
	
	Furthermore, for a triple $\bk{E,\chi,s_0}$ appearing in Table \Cref{Poles of Eisenstein Series}, the residual representation of $\Eisen_E\bk{\chi, \cdot , s, \cdot }$ is square-integrable with the exception of the following cases:
	\begin{itemize}
		\item $E=F\times K$ where $K$ is a field:
		\begin{itemize}
			\item $s=\frac{1}{2}$ with $\chi=\Id,\chi_K$.
			\item $s=\frac{3}{2}$ with $\chi=\Id$.
		\end{itemize}
		\item $E=F\times F\times F$, $s=\frac{1}{2}$ with $\chi=\Id$.
	\end{itemize}

\end{Thm}

The proof of \Cref{Thm: Poles of Eisenstein series} relies on a conjecture, \Cref{Conj: Archimedean places1}, regarding the structure of the degenerate principal series at the Archimedean places.
Before proving \Cref{Thm: Poles of Eisenstein series} we wish to describe the key ideas of the proof and during this discussion we shall present \Cref{Conj: Archimedean places1}.
It is worth pointing out that the proof of \Cref{Thm: CAP representations} is independent from \Cref{Conj: Archimedean places1} and can be carried out using only the bounds in Table \ref{Trivial Bounds on the order of Poles of Eisenstein Series}.

In the course of the proof we use \Cref{Eq: Constant term} to evaluate the constant term and check the cancellation of the poles of the various intertwining operators.
Fix a triple $\bk{E,\chi,s_0}$ as above such that 
\[
\max\set{\operatorname{ord}_{s=s_0} M\bk{w,\chi_s}f_s\bk{g} \mvert w\in W\bk{P_E,H_E},\ f_s\in I_{P_E}\bk{\chi,s},\ g\in H_E\bk{\A}} = n > 0 .
\]
We denote this integer by $ord_{s=s_0} M\bk{w,\chi_s}$.
For $0 < m\leq n$ let
\[
\Sigma_{\bk{E,\chi,s_0,m}} = \set{w\in W\bk{P_E,H_E}\mvert \operatorname{ord}_{s=s_0} M\bk{w,\chi_s} \geq m} .
\]
We say that the pole of order $m$ cancels if
\[
\lim\limits_{s\to s_0} \bk{s-s_0}^m \suml_{w\in W\bk{P_E,H_E}} M\bk{w,\chi_s}\res{I_{P_E}\bk{\chi,s}} = \lim\limits_{s\to s_0} \bk{s-s_0}^m \suml_{w\in \Sigma_{\bk{E,\chi,s_0,m}}} M\bk{w,\chi_s}\res{I_{P_E}\bk{\chi,s}} \equiv 0 .
\]


The cancellation of poles in the proof happens for two reasons:

\paragraph{$\bigstar$ \underline{\textbf{Reason 1 for cancellation of  poles:}}}
If $I_{P_E}\bk{\chi,s}$ is generated by the global spherical section $f_s^0$ we evaluate the leading terms of $\Eisen_E\bk{\chi, f_s^0,s,t}_{B_E}$ at $s_0$.
By doing this, one can determine the order of the pole of $\Eisen_E\bk{\chi, f_s^0,s,t}$ at $s_0$ since the residual representation at this point is a quotient of $I_{P_E}\bk{\chi,s}$ and hence must be generated by the spherical vector.

We evaluate the terms in the Laurent series of $\Eisen_E\bk{\chi, f_s^0,s,t}_{B_E}$ by an application of \Cref{Eq: Constant term}, The Gindikin-Karpelevich formula (\Cref{Gindikin-Karpelevich formula}), the Functional Equation (\Cref{Functional Equation - Zeta Function}) and the Laurent expansion of $\zfun_L$ for various extensions $L$ of $F$. 
In fact, the first terms of the various Gindikin-Karpelevich factors cancel due to simple algebraic computations.
	
This reason is applied to the cases where $\chi=\Id$, $s_0=\frac{1}{2}$ and $E$ is not a field.
In these cases we prove a reduction of poles of order 3 and 4 to a pole of order 1.	
We note that usually poles of intertwining operators cancel in pairs; here it happens that the cancellation of poles is in triples or quintuples of intertwining operators.

The global representation $I_{P_E}\bk{\chi,s}$ is generated by the spherical vector if and only if $I_{P_E}\bk{\chi_\nu,s}$ is generated by the spherical vector for any $\nu\in\Places$.
Indeed, for $\nu\nmid\infty$, $\chi_\nu=\Id_\nu$, $s=\frac{1}{2}$ and $E_\nu$ not a field we have
\begin{Prop}[\cite{Gan-Savin-D4}]
	\label{Prop: Representation at 1/2 is generated by spherical vector}
	For $\nu\nmid\infty$ and $E_\nu$ not a field the degenerate principal series representation $I_{P_E}\bk{\Id_\nu,\frac{1}{2}}$ has a unique irreducible quotient which is unramified.
	In particular it is generated by the spherical vector.
\end{Prop}

As for the Archimedean places:
\begin{Conj}
	\label{Conj: Archimedean places1}
	For $\nu\vert\infty$, the degenerate principal series representation $I_{P_E}\bk{\Id_\nu,\frac{1}{2}}$ is generated by the spherical vector.
\end{Conj}
We note that the proof of \Cref{Thm: Poles of Eisenstein series} relies on \Cref{Conj: Archimedean places1} in the following cases:
\begin{itemize}
	\item $E$ is a field, $s=\frac{1}{2}$ with $\chi=\chi_E$ and $s=\frac{3}{2}$ with $\chi=\Id$.
	\item $E=F\times K$, $s=\frac{1}{2}$ with $\chi^2=\Id$.
	\item $E=F\times F\times F$, $s=\frac{1}{2}$ with $\chi=\Id$.
\end{itemize}

\paragraph{$\bigstar$ \underline{\textbf{Reason 2 for cancellation of poles:}}}
Assume that we can decompose $\Sigma_{\bk{E,\chi,s_0,n}}$ into a disjoint union of pairs $\set{w',w'w''}$.
We then have
\[
M_{w'w''} = M_{w''} \circ M_{w'} .
\]
Assume that for any such pair we can further show that $M_{w''}$ is an endomorphism of the image of $M_{w'}$ acting as $-Id$; this is done using \Cref{Cor: Keys-Shahidi} or \Cref{Thm: Results regarding local intertwining operators}.
Then
\[
\lim\limits_{s\to s_0} \bk{s-s_0}^n \coset{M_{w'}+M_{w'w''}} \equiv 0 .
\]
It then follows that the pole of order $n$ cancels.

\begin{Remark}
	As mentioned above, parts of the proof require data on the behavior of some local intertwining operators.
	 obtained in \Cref{Thm: Results regarding local intertwining operators}.
	So as not to disturb the discussion of the proof of \Cref{Thm: Poles of Eisenstein series} this discussion is postponed to \Cref{Thm: Results regarding local intertwining operators} in  \Cref{App: Local Intertwining Operators}.
\end{Remark}

\begin{Remark}
The two reasons for cancellations of poles are not unrelated and not redundant.
The first reason is applicable only when the global degenerate principle series is generated by the spherical vector and under this assumption it allows us to reduce more than one order of the pole.
On the other hand, the second reason is applicable regardless of the structure of the degenerate principle series but can only be used to reduce the order of the pole by one.
\end{Remark}

After, maybe, cancellation of higher orders of a pole, we wish to determine its actual order.
Namely, for $0<m\leq n$ assume that
\[
\lim\limits_{s\to s_0} \bk{s-s_0}^{m+1} \suml_{w\in W\bk{P_E,H_E}} M\bk{w,\chi_s} \res{I_{P_E}\bk{\chi,s}} \equiv 0 .
\]
Then $\Eisen_E\bk{\chi, \cdot, s, \cdot}$ attains a pole of order $m$ at $s_0$ if
\[
\lim\limits_{s\to s_0} \bk{s-s_0}^m \suml_{w\in W\bk{P_E,H_E}} M\bk{w,\chi_s}\res{I_{P_E}\bk{\chi,s}} =
\lim\limits_{s\to s_0} \bk{s-s_0}^m \suml_{w\in \Sigma_{\bk{E,\chi,s_0,m}}} M\bk{w,\chi_s}\res{I_{P_E}\bk{\chi,s}} \not\equiv 0.
\]
In particular, for any holomorphic section $f_s\in I_{P_E}\bk{\chi,s}$ and any $t\in T_E\bk{\A}$ it holds that
\[
\lim\limits_{s\to s_0} \bk{s-s_0}^m \suml_{w\in W\bk{P_E,H_E}} M\bk{w,\chi_s}f_s\bk{t} \in \C 
\]
and the limit is non-zero for some $f_s$ and $t$.

We prove this using one of the following reasonings:
\paragraph{$\bigstar$ \underline{\textbf{Reason 1 for non-vanishing of the leading term:}}}
One can prove the non-vanishing of the leading term by providing a section $f_s\in I_{P_E}\bk{\chi,s}$ such that
\[
\lim\limits_{s\to s_0} \bk{s-s_0}^m \suml_{w\in \Sigma_{\bk{E,\chi,s_0,m}}} M\bk{w,\chi_s}f_s \not\equiv 0 .
\]

\begin{Remark}
	A global spherical section exists if and only if $\chi$ is everywhere unramified.
	In case that $\chi$ is everywhere unramified one can check that the orders of poles in Table \ref{Poles of Eisenstein Series} are realized by the global spherical section.
\end{Remark}


\paragraph{$\bigstar$ \underline{\textbf{Reason 2 for non-vanishing of the leading term:}}}
The representation
\[
Res\bk{s_0,\chi,E}_{B_E} = \operatorname{Span}_\C \set{ \lim_{s\to s_0} \bk{s-s_0}^m \Eisen_E\bk{\chi,f_s,s,t}_{B_E} \mvert f_s\in I_{P_E}\bk{\chi,s}}
\]
decomposes into a sum of copies of the one dimensional representations of $T_E$
\[
\set{w^{-1}\cdot\chi_s \mvert w\in \Sigma_{\bk{E,\chi,s_0,m}}} .
\]
The elements $M\bk{w,\chi_s}f_s\bk{t}$ lie in the representation $w^{-1}\cdot\chi_s$ as representations of $T_E$.
We define an equivalence class on $\Sigma_{\bk{E,\chi,s_0,m}}$ by:
\begin{equation}
\label{Equivalence relation on Sigma-m}
w\sim_{s_0} w' \Longleftrightarrow w^{-1}\cdot\chi_{s_0}=w'^{-1}\cdot\chi_{s_0} .
\end{equation}
Clearly, cancellations of poles of intertwining operators can occur only within the same equivalence class.
And so, if there exists $w\in \Sigma_{\bk{E,\chi,s_0,m}}$ such that $w\not\sim_{s_0} w'$ \textbf{for all} $w\neq w'\in \Sigma_{\bk{E,\chi,s_0,m}}$ then the term $\lim\limits_{s\to s_0} \bk{s-s_0}^m M\bk{w,\chi_s}f_s$ cannot by canceled by other terms in the sum, while it is non-zero due to \Cref{Eq: Global Gindikin-Karpelevich} and \Cref{Lem: Normalized intertwining oeprator is non-vanishing}.


\begin{Remark}
In fact, Reason 2 for non-vanishing of the leading term 
\end{Remark}

\begin{proof}[Proof of \Cref{Thm: Poles of Eisenstein series}]

For any $E$ the poles corresponding to the triple $\bk{E,\chi_E,\frac{3}{2}}$ are treated in \cite{MR1918673}.
Also, since the poles at $s=\frac{5}{2}$ and $\chi=\Id$ arise only from the intertwining operator associated with the longest element of the Weyl group and hence they cannot be canceled.

In what follows we treat the rest of the points in \Cref{Trivial Bounds on the order of Poles of Eisenstein Series}.
We leave the discussion of the square integrability of the residual representations to the end of this section.
In what follows, we denote by $t$ an element of $T_E\bk{\A_F}$ of the  form
\[
t = \piece{
	h_{\alpha_1}\bk{t_1} h_{\alpha_2}\bk{t_2} h_{\alpha_3}\bk{t_3} h_{\alpha_4}\bk{t_4},& E=F\times F\times F,& t_1, t_2, t_3, t_4\in \A_F^\times \\
	h_{\alpha_1}\bk{t_1} h_{\alpha_2}\bk{t_2} h_{\alpha_3}\bk{t_3} h_{\alpha_4}\bk{t_3^{\sigma}},& E=F\times K,& t_1, t_2, \in \A_F^\times,\ t_3\in \A_K^\times \\
	h_{\alpha_1}\bk{t_1} h_{\alpha_2}\bk{t_2} h_{\alpha_3}\bk{t_1^\sigma} h_{\alpha_4}\bk{t_1^{\sigma^2}},& E\text{ is a field},& t_1\in \A_E^\times,\ t_2\in \A_F^\times
	}
\]

\subsection{$E$ a field, $s=\frac{1}{2}$, $\chi=\Id$}
The intertwining operators in this case have poles at most of order $2$.
We show that the pole of order 2 cancels and that the pole of order 1 does not.
\begin{enumerate}
	\item We have
	\[
	\Sigma_{\bk{E,\Id,\frac{1}{2},2}} = \set{w_{212}, w_{2121}} .
	\]
	Since $w_{212}^{-1}\cdot\chi_{\frac{1}{2}}\bk{t} = w_{212}^{-1}\cdot\chi_{\frac{1}{2}}\bk{t} = \frac{1}{\FNorm{t_2}_F}$ we have $w_{212}\sim_{s_0} w_{2121}$.
	We write $w_{2121}=w_{212}w_1$.
	It follows from \Cref{Cor: Keys-Shahidi} that
	\begin{equation*}
	\lim\limits_{s\to \frac{1}{2}} \bk{s-\frac{1}{2}}^2 M\bk{w_{2121},\chi_s} =  - \lim\limits_{s\to \frac{1}{2}} \bk{s-\frac{1}{2}}^2 M\bk{w_{212},\chi_s} .
	\end{equation*}
	Following reason 1 for cancellation of poles
	\begin{equation*}
	\lim\limits_{s\to \frac{1}{2}} \bk{s-\frac{1}{2}}^2 \Eisen_E\bk{\chi, f_s, s, g}_{B_E} = 0 \qquad  \forall f_s\in I_{P_E}\bk{\chi,s} .
	\end{equation*}
	Thus, the pole of order $2$ is canceled.
	
	\item
	We have
	\[
	\Sigma_{\bk{E,\Id,\frac{1}{2},1}} = \set{w_{21}, w_{212}, w_{2121}, w_{21212}} .
	\]
	We note that $w_{21}^{-1}\cdot\chi_{\frac{1}{2}}\bk{t} = w_{21}^{-1}\cdot\chi_{\frac{1}{2}}\bk{t} = \frac{\FNorm{t_2}_F}{\FNorm{t_1}_E}$.
	We prove that $\lim\limits_{s\to \frac{1}{2}} \bk{s-\frac{1}{2}} \Eisen_E\bk{\chi, f_s, s, g}_{B_E} \not\equiv 0$ by proving that for the global spherical section $f_s^0$ it holds that $\lim\limits_{s\to\frac{1}{2}} \bk{s-\frac{1}{2}} \bk{M_{w_{21}}+M_{w_{21212}}}f_s^0 \neq 0$ thus applying both reason 1 and 2 for non-vanishing of the leading term.
	Indeed, we write
	\[
	\zfun_F\bk{s} = \frac{\gamma_{-1}}{s-1} + \gamma_0+...,\quad 
	\zfun_K\bk{s} = \frac{\epsilon_{-1}}{s-1} + \epsilon_0+...
	\]
		
	It holds that
	\begin{align*}
	&\lim\limits_{s\to \frac{1}{2}} \bk{s-\frac{1}{2}} \bk{M_{w_{21}}+M_{w_{21212}}}f_s^0 \\
	& = \lim\limits_{s\to \frac{1}{2}} \bk{s-\frac{1}{2}} \bk{
		\frac{\zfun_F\bk{s+\frac{3}{2}}\zfun_E\bk{s+\frac{1}{2}}}{\zfun_F\bk{s+\frac{5}{2}}\zfun_E\bk{s+\frac{3}{2}}} +
		\frac{\zfun_F\bk{s-\frac{3}{2}}\zfun_F\bk{s+\frac{3}{2}}\zfun_E\bk{s-\frac{1}{2}}\zfun_F\bk{2s}}{\zfun_F\bk{s-\frac{1}{2}}\zfun_F\bk{s+\frac{5}{2}}\zfun_E\bk{s+\frac{3}{2}}\zfun_F\bk{2s+1}} } \\
	& = \frac{\zfun_F\bk{2}}{\zfun_F\bk{3}\zfun_E\bk{2}} \lim\limits_{s\to \frac{1}{2}} \bk{s-\frac{1}{2}} \bk{\zfun_F\bk{s+\frac{1}{2}} + \frac{\zfun_F\bk{-1}}{\zfun_F\bk{2}} \frac{\zfun_E\bk{s-\frac{1}{2}}\zfun_F\bk{2s}}{\zfun_F\bk{s-\frac{1}{2}}}}\\
	& = \frac{\zfun_F\bk{2}}{\zfun_F\bk{3}\zfun_E\bk{2}} \bk{-\gamma_{-1} + \frac{\epsilon_{-1} \frac{1}{2}\gamma_{-1}}{-\gamma_{-1}}} = \frac{\zfun_F\bk{2}}{\zfun_F\bk{3}\zfun_E\bk{2}} \bk{\gamma_{-1} + \frac{1}{2}\epsilon_{-1}} \neq 0 .
	\end{align*}
	Here we use the fact that $\gamma_{-1}$ and $\epsilon_{-1}$ are positive numbers due to the class number formula \cite[pg. 37]{MR1421575}.
\end{enumerate}

\subsection{$E$ a field, $s=\frac{1}{2}$, $\chi=\chi_E$}
The intertwining operators in this case have poles of order at most $1$.
We show that the Eisenstein series is in fact holomorphic at this point.
We have
\[
\Sigma_{\bk{E,\chi_E,\frac{1}{2},1}} = \set{w_{21}, w_{212}, w_{2121}, w_{21212}}
\]
and
\begin{align*}
& w_{21}^{-1}\cdot\chi_{\frac{1}{2}}\bk{t} = w_{21212}^{-1}\cdot\chi_{\frac{1}{2}}\bk{t}  = \overline{\chi_E\bk{t_2}} \frac{\FNorm{t_2}_F}{\FNorm{t_1}_E} \\
& w_{212}^{-1}\cdot\chi_{\frac{1}{2}}\bk{t} = w_{2121}^{-1}\cdot\chi_{\frac{1}{2}}\bk{t} = \chi_E\bk{t_2} \frac{1}{\FNorm{t_2}_F} .
\end{align*}

We write $w_{2121}=w_{212}w_1$.
It follows from \Cref{Cor: Keys-Shahidi} and reason 2 for cancellation of poles that
\begin{equation*}
\lim\limits_{s\to \frac{1}{2}} \bk{s-\frac{1}{2}} \coset{M\bk{w_{2121},\chi_s} + M\bk{w_{212},\chi_s}} = 0 .
\end{equation*}
On the other hand,
\begin{equation*}
M\bk{w_{21212},\chi_s} = M\bk{w_{212},w_{21}^{-1}\cdot\chi_s}\circ M\bk{w_{21},\chi_s} .
\end{equation*}

As in \Cref{App: Local Intertwining Operators}, for $\nu\in\Places$ denote
\[
R_\nu=R_\nu\bk{w_{212}, \chi_{E_\nu}} = \lim\limits_{s\to\frac{1}{2}} N_\nu\bk{w_{212},w_{21}^{-1}\cdot\chi_{E_\nu,s}} .
\]
Fix factorizable data $f_s=\placestimes f_{s,\nu}$ such that $M\bk{w_{21},\chi_s}f_s=\placestimes \widetilde{f}_{s,\nu}$ and assume that $S\subseteq\Places$ is a finite set of places such that $f_{s,\nu}=f_{s,\nu}^0$ for $\nu\notin S$.
It follows from \Cref{Thm: Results regarding local intertwining operators} that $R_\nu \widetilde{f}_{\frac{1}{2},\nu} = \widetilde{f}_{\frac{1}{2},\nu}$ for any $\nu\in S$.

It follows from \Cref{Eq: Functional equation of intertwining operators} that
\[
M\bk{w_{21212},\chi_s} = M\bk{w_{212},w_{21}^{-1}\cdot\chi_s}\circ M\bk{w_{21},\chi_s}
\]
On the other, by \Cref{Eq: Global functional equation of Hecke L-function and epsilon-factor},  and \Cref{Eq: Global Gindikin-Karpelevich}, it holds that
\begin{align*}
& M\bk{w_{21212},\chi_s} f_s = \frac{J\bk{w_{212},w_{21}^{-1}\cdot\chi_s}}{\prodl_{\nu\in S} \epsilon_{F_\nu}\bk{-1,\chi_{E_\nu}, \psi_\nu} \epsilon_{F_\nu}\bk{1,\chi_{E_\nu}^2, \psi_\nu}} M\bk{w_{21},\chi_s} f_s \\
& = \frac{\zfun_E\bk{0} \Lfun_F\bk{-1,\chi_E} \Lfun_F\bk{1,\chi_E^2}}{\zfun_E\bk{1} \Lfun_F\bk{0,\chi_E} \Lfun_F\bk{2,\chi_E^2} \prodl_{\nu\in S} \epsilon_{F_\nu}\bk{-1,\chi_{E_\nu}, \psi_\nu} \epsilon_{F_\nu}\bk{1,\chi_{E_\nu}^2, \psi_\nu}} M\bk{w_{21},\chi_s} f_s \\
& = - M\bk{w_{21},\chi_s} f_s .
\end{align*}
Namely, $M\bk{w_{212},w_{21}^{-1}\cdot\chi_{\frac{1}{2}}}$ is an endomorphism of the image of the residue of $M\bk{w_{21},\chi_{s}}$, at $s=\frac{1}{2}$, acting as $-Id$.
Following reason 2 for cancellation of poles we have
\[
\lim\limits_{s\to \frac{1}{2}} \bk{s-\frac{1}{2}} \coset{M\bk{w_{21212},\chi_s} + M\bk{w_{21},\chi_s}} = 0 .
\]

In conclusion,
\[
\lim\limits_{s\to \frac{1}{2}} \bk{s-\frac{1}{2}} \Eisen_E\bk{\chi, f_s, s, g}_{B_E} = 0
\]
and hence $\Eisen_E\bk{\chi_K, f_s, s, g}$ is holomorphic at $s=\frac{1}{2}$.

\subsection{$E$ a field, $s=\frac{1}{2}$, $\chi^2=\Id$, $\chi\neq \Id$}
The intertwining operators in this case have poles of order at most $1$.
We show that indeed the Eisenstein series admits a simple pole at this point.
We have
\[
\Sigma_{\bk{E,\chi,\frac{1}{2},1}} = \set{w_{212}, w_{2121}, w_{21212}}
\]
and
\begin{align*}
& w_{212}^{-1}\cdot\chi_{\frac{1}{2}}\bk{t} = \chi\bk{\Nm_{E\rmod F}\bk{t_1}} \frac{1}{\FNorm{t_2}_F} \\
& w_{2121}^{-1}\cdot\chi_{\frac{1}{2}}\bk{t} = \chi\bk{t_2\Nm_{E\rmod F}\bk{t_1}} \frac{1}{\FNorm{t_2}_F} \\
& w_{21212}^{-1}\cdot\chi_{\frac{1}{2}}\bk{t} = \chi\bk{t_2} \frac{\FNorm{t_2}_F}{\FNorm{t_1}_E} .
\end{align*}

Following reason 2 for non-cancellation of poles, $\Eisen_E\bk{\chi, f_s, s, g}$ admits a simple pole at $s=\frac{1}{2}$.

\subsection{$E$ a field, $s=\frac{3}{2}$, $\chi=\Id$}
The intertwining operators in this case have poles of order at most $1$.
We show that the Eisenstein series is holomorphic at this point.
We have
\[
\Sigma_{\bk{E,\Id,\frac{3}{2},1}} = \set{w_{2121}, w_{21212}} .
\]
Since $w_{2121}^{-1}\cdot\chi_{\frac{3}{2}}\bk{t} = w_{21212}^{-1}\cdot\chi_{\frac{3}{2}}\bk{t} = \frac{1}{\FNorm{t_1}_E}$ we have $w_{2121}\sim_{s_0} w_{21212}$.
We write $w_{21212}=w_{2121}w_2$.
It follows from \Cref{Cor: Keys-Shahidi} that
\begin{equation*}
\lim\limits_{s\to \frac{3}{2}} \bk{s-\frac{3}{2}}^2 M\bk{w_{21212},\chi_s} =  - \lim\limits_{s\to \frac{3}{2}} \bk{s-\frac{3}{2}}^2 M\bk{w_{2121},\chi_s} .
\end{equation*}
Following reason 1 for cancellation of poles
\begin{equation*}
\lim\limits_{s\to \frac{3}{2}} \bk{s-\frac{3}{2}}^2 \Eisen_E\bk{1, f_s, s, g}_{B_E} = 0 \qquad \forall f_s\in I_{P_E}\bk{\chi,s} .
\end{equation*}
Thus $\Eisen_E\bk{\Id, f_s, s, g}$ is holomorphic at $s=\frac{3}{2}$.

\subsection{$E=F\times K$, $K$ a field, $s=\frac{1}{2}$, $\chi=\Id$}
The intertwining operators in this case have poles of order at most $3$.
We show that the Eisenstein series admits a simple pole at this point.
We have
\begin{align*}
& \Sigma_{\bk{F\times K,\Id,\frac{1}{2},3}} = \set{w_{2132}, w_{21321}, w_{21323}, w_{213213}} \\
& \Sigma_{\bk{F\times K,\Id,\frac{1}{2},2}} \setminus \Sigma_{\bk{F\times K,\Id,\frac{1}{2},3}} = \set{ w_{213}, w_{2321}, w_{2132132}} \\
& \Sigma_{\bk{F\times K,\Id,\frac{1}{2},1}} \setminus \Sigma_{\bk{F\times K,\Id,\frac{1}{2},2}} = \set{w_{21}, w_{23}, w_{232}} .
\end{align*}
and also
\begin{align*}
& w_{2132}^{-1}\cdot\chi_{\frac{1}{2}}\bk{t} = w_{21321}^{-1}\cdot\chi_{\frac{1}{2}}\bk{t} = w_{21323}^{-1}\cdot\chi_{\frac{1}{2}}\bk{t} = w_{213213}^{-1}\cdot\chi_{\frac{1}{2}}\bk{t} = \frac{1}{\FNorm{t_2}_F} \\
& w_{213}^{-1}\cdot\chi_{\frac{1}{2}}\bk{t} = w_{2321}^{-1}\cdot\chi_{\frac{1}{2}}\bk{t} = w_{2132132}^{-1}\cdot\chi_{\frac{1}{2}}\bk{t} = \frac{\FNorm{t_2}_F}{\FNorm{t_1}_F \FNorm{t_3}_K} \\
& w_{23}^{-1}\cdot\chi_{\frac{1}{2}}\bk{t} = w_{232}^{-1}\cdot\chi_{\frac{1}{2}}\bk{t} = \frac{\FNorm{t_1}_F}{\FNorm{t_3}_K} \\
& w_{21}^{-1}\cdot\chi_{\frac{1}{2}}\bk{t} = \frac{\FNorm{t_3}_K}{\FNorm{t_1}_F \FNorm{t_2}_F} .
\end{align*}
We note that it follows from the above that
\[
\Sigma_{\bk{F\times K,\Id,\frac{1}{2},2}}\rmod\sim_{s_0} = 
\set{ \Sigma_{\bk{F\times K,\Id,\frac{1}{2},3}}, \Sigma_{\bk{F\times K,\Id,\frac{1}{2},2}} \setminus \Sigma_{\bk{F\times K,\Id,\frac{1}{2},3}}} .
\]

After proving that the poles of order 3 and 2 cancel, the fact that the simple pole is attained follows from reason 2 for non-vanishing of the leading term. More precisely, if $f_s^0$ is the global spherical vector one has
\[
\lim\limits_{s\to\frac{1}{2}} M\bk{w_{21},\chi_s}f_s^0 \neq 0 .
\]
Namely, $\Eisen_{F\times K}\bk{\Id, f_s^0, s, g}$ admits a simple pole at $s=\frac{1}{2}$.

We now turn to prove that the poles of higher order are canceled.
We first note that the pole of order 3 cancels due to reason 2 for cancellation of poles, regardless of \Cref{Conj: Archimedean places1}.
According to \Cref{Prop: Representation at 1/2 is generated by spherical vector} and \Cref{Conj: Archimedean places1} the representation $I_{P_E}\bk{\Id,\frac{1}{2}}$ is generated by $f_{\frac{1}{2}}^0$, where $f_s^0$ is the global spherical section.

We write
\[
\zfun_F\bk{s} = \frac{\gamma_{-1}}{s-1} + \gamma_0+...,\quad 
\zfun_K\bk{s} = \frac{\delta_{-1}}{s-1} + \delta_0+...
\]

The functional equation, \Cref{Functional Equation - Zeta Function}, yields
\begin{equation}
\frac{\zfun_F\bk{s-\frac{1}{2}}}{\zfun_F\bk{s+\frac{1}{2}}} = \frac{\zfun_F\bk{\frac{3}{2}-s}}{\zfun_F\bk{s+\frac{1}{2}}} \underset{s\to\frac{1}{2}}{\longrightarrow} -1
\end{equation}
and also
\begin{equation}
\gamma_{-1} = \Res\bk{\zfun_F\bk{s},1} = -\Res\bk{\zfun_F\bk{s},0},\quad 
\delta_{-1} = \Res\bk{\zfun_K\bk{s},1} = -\Res\bk{\zfun_K\bk{s},0} .
\end{equation}


We note that since $\Sigma_{\bk{F\times K,\Id,\frac{1}{2},3}}$ is an equivalence class with respect to $\sim_{s_0}$ it suffices to check the cancellations at $t=1\in T_E\bk{\A}$.

\begin{align*}
& \coset{M\bk{w_{2132}}+ M\bk{w_{21321}}+ M\bk{w_{21323}}+ M\bk{w_{213213}}} f_s^0\bk{1} \\
& = \frac{\zfun_F\bk{s+\frac{1}{2}}\zfun_K\bk{s+\frac{1}{2}}\zfun_F\bk{2s}}{\zfun_F\bk{s+\frac{5}{2}} \zfun_K\bk{s+\frac{3}{2}} \zfun_F\bk{2s+1}} + \frac{\zfun_F\bk{s-\frac{1}{2}} \zfun_K\bk{s+\frac{1}{2}} \zfun_F\bk{2s}}{\zfun_F\bk{s+\frac{5}{2}} \zfun_K\bk{s+\frac{3}{2}} \zfun_F\bk{2s+1}} \\
& + \frac{\zfun_F\bk{s+\frac{1}{2}} \zfun_K\bk{s-\frac{1}{2}} \zfun_F\bk{2s}}{\zfun_F\bk{s+\frac{5}{2}} \zfun_K\bk{s+\frac{3}{2}} \zfun_F\bk{2s+1}} + \frac{\zfun_F\bk{s-\frac{1}{2}} \zfun_K\bk{s-\frac{1}{2}} \zfun_F\bk{2s}}{\zfun_F\bk{s+\frac{5}{2}}\zfun_K\bk{s+\frac{3}{2}} \zfun_F\bk{2s+1}} .
\end{align*}
As the denumerators are equal, holomorphic and non-vanishing at $s=\frac{1}{2}$ it is enough to prove that the sum of the numerators admit at most a simple pole.
Indeed,
\begin{align*}
& = \zfun_F\bk{s+\frac{1}{2}}\zfun_K\bk{s+\frac{1}{2}}\zfun_F\bk{2s} + \zfun_F\bk{s-\frac{1}{2}} \zfun_K\bk{s+\frac{1}{2}} \zfun_F\bk{2s} \\
& + \zfun_F\bk{s+\frac{1}{2}} \zfun_K\bk{s-\frac{1}{2}} \zfun_F\bk{2s} + \zfun_F\bk{s-\frac{1}{2}} \zfun_K\bk{s-\frac{1}{2}} \zfun_F\bk{2s} \\
& = \frac{\gamma_{-1}^2\delta_{-1} - \gamma_{-1}^2\delta_{-1} - \gamma_{-1}^2\delta_{-1} + \gamma_{-1}^2\delta_{-1}}{2\bk{s-\frac{1}{2}}^3} \\
& \quad + \frac{\gamma_{-1}\bk{3\delta_{-1}\gamma_0+\gamma_{-1}\delta_0} - \gamma_{-1}\bk{\gamma_{-1}\delta_0+\delta_{-1}\gamma_0} \gamma_{-1}\bk{\gamma_{-1}\delta_0-3\delta_{-1}\gamma_0}+\gamma_{-1}\bk{\delta_{-1}\gamma_0-\gamma_{-1}\delta_0}   }{2\bk{s-\frac{1}{2}}^2} \\
& + o\bk{\bk{s-\frac{1}{2}}^{-2}} = o\bk{\bk{s-\frac{1}{2}}^{-2}} .
\end{align*}


We note that since $\Sigma_{\bk{F\times K,\Id,\frac{1}{2},3}} \setminus \Sigma_{\bk{F\times K,\Id,\frac{1}{2},2}}$ is an equivalence class with respect to $\sim_{s_0}$ it suffices to check the cancellations at $t=1\in T_E\bk{\A}$.

We consider the leading term in the Laurent series of the corresponding intertwining operators:
\begin{align*}
& \frac{\zfun_F\bk{s+\frac{1}{2}} \zfun_K\bk{s+\frac{1}{2}}}{\zfun_F\bk{s+\frac{5}{2}} \zfun_K\bk{s+\frac{3}{2}}} + \frac{\zfun_F\bk{s+\frac{3}{2}} \zfun_K\bk{s+\frac{1}{2}} \zfun_F\bk{s-\frac{1}{2}} \zfun_F\bk{2s}}{\zfun_F\bk{s+\frac{5}{2}} \zfun_K\bk{s+\frac{3}{2}} \zfun_F\bk{s+\frac{1}{2}} \zfun_F\bk{2s+1}} \\
& + \frac{\zfun_F\bk{s-\frac{3}{2}} \zfun_K\bk{s-\frac{1}{2}} \zfun_F\bk{2s}}{\zfun_F\bk{s+\frac{5}{2}} \zfun_K\bk{s+\frac{3}{2}} \zfun_F\bk{2s+1}} \\
& \quad = \frac{\zfun_F\bk{2} \gamma_{-1}\delta_{-1} - \frac{1}{2}\zfun_F\bk{2}\gamma_{-1}\delta_{-1} - \frac{1}{2}\zfun_F\bk{-1} \gamma_{-1}\delta_{-1}}{\zfun_F\bk{2} \zfun_K\bk{2} \zfun_F\bk{3} \bk{s-\frac{1}{2}}^2} + o\bk{\bk{s-\frac{1}{2}}^{-2}} \\
& = o\bk{\bk{s-\frac{1}{2}}^{-2}} .
\\
\end{align*}
In conclusion, $\Eisen_E\bk{1, f_s^0, s, g}$ admits a pole of order $1$ at $s=\frac{1}{2}$.

\begin{Remark}
Applying reason 2 for cancellation of poles one can show, independently \Cref{Conj: Archimedean places1}, that the pole here is of order at most $2$.
\end{Remark}

\subsection{$E=F\times K$, $K$ a field, $s=\frac{1}{2}$, $\chi=\chi_K$}
The intertwining operators in this case have poles of order at most $2$.
We show that the Eisenstein series admits a simple pole at this point.
We have
\begin{align*}
\Sigma_{\bk{F\times K,\chi_K,\frac{1}{2},2}} = \set{w_{2321}, w_{2132}, w_{21321}, w_{21323}, w_{213213}, w_{2132132}}
\end{align*}
and
\begin{align*}
& w_{2321}^{-1}\cdot\chi_{\frac{1}{2}}\bk{t} = w_{2132132}^{-1}\cdot\chi_{\frac{1}{2}}\bk{t} = \chi_K\bk{t_2} \frac{\FNorm{t_2}_F}{\FNorm{t_1}_F \FNorm{t_3}_K} \\
& w_{2132}^{-1}\cdot\chi_{\frac{1}{2}}\bk{t} = w_{21323}^{-1}\cdot\chi_{\frac{1}{2}}\bk{t} = \chi_K\bk{t_1} \frac{1}{\FNorm{t_2}_F} \\ 
& w_{21321}^{-1}\cdot\chi_{\frac{1}{2}}\bk{t} = w_{213213}^{-1}\cdot\chi_{\frac{1}{2}}\bk{t} = \chi_K\bk{t_1t_2} \frac{1}{\FNorm{t_2}_F} .
\end{align*}
We write $w_{21323}=w_{2132}w_3$ and $w_{213213}=w_{21321}w_3$.
It follows from \Cref{Cor: Keys-Shahidi} that
\begin{align*}
& \lim\limits_{s\to \frac{1}{2}} \bk{s-\frac{1}{2}}^2 \coset{M\bk{w_{21323},\chi_s} + M\bk{w_{2132},\chi_s}} = 0 \\
& \lim\limits_{s\to \frac{1}{2}} \bk{s-\frac{1}{2}}^2 \coset{M\bk{w_{213213},\chi_s} + M\bk{w_{21321},\chi_s}} = 0 .
\end{align*}

On the other hand,
\[
M\bk{w_{2132132},\chi_s} = M\bk{w_{232},w_{2321}^{-1}\cdot\chi_s}\circ M\bk{w_{2321},\chi_s} .
\]
As in \Cref{App: Local Intertwining Operators}, for $\nu\in\Places$ denote
\[
R_\nu=R_\nu\bk{w_{232}, \chi_{K_\nu}} = \lim\limits_{s\to\frac{1}{2}} N_\nu\bk{w_{232},w_{2321}^{-1}\cdot\chi_{K_\nu,s}} .
\]
Fix factorizable data $f_s=\placestimes f_{s,\nu}$ such that $M\bk{w_{2321},\chi_s}f_s=\placestimes \widetilde{f}_{s,\nu}$ and assume that $S\subseteq\Places$ is a finite set of places such that $f_{s,\nu}=f_{s,\nu}^0$ for $\nu\notin S$.
It follows from \Cref{Thm: Results regarding local intertwining operators} that $R_\nu \widetilde{f}_{s,\nu} = \widetilde{f}_{s,\nu}$ for any $\nu\in S$.
It follows from \Cref{Eq: Global functional equation of Hecke L-function and epsilon-factor}, \Cref{Eq: Functional equation of intertwining operators} and \Cref{Eq: Global Gindikin-Karpelevich} that
\begin{align*}
& M\bk{w_{2132132},\chi_s} f_s = \frac{J\bk{w_{232},w_{2321}^{-1}\cdot\chi_s}}{\prodl_{\nu\in S} \epsilon_{F_\nu}\bk{-1,\chi_{K_\nu}, \psi_\nu} \epsilon_{F_\nu}\bk{1,\chi_{K_\nu}, \psi_\nu} } M\bk{w_{2321},\chi_s} f_s \\
& = \frac{\zfun_K\bk{0} \Lfun_F\bk{-1,\chi_K} \Lfun_F\bk{1,\chi_K}}{\zfun_K\bk{1} \Lfun_F\bk{0,\chi_K} \Lfun_F\bk{2,\chi_K} \prodl_{\nu\in S} \epsilon_{F_\nu}\bk{-1,\chi_{K_\nu}, \psi_\nu} \epsilon_{F_\nu}\bk{1,\chi_{K_\nu}, \psi_\nu}} M\bk{w_{2321},\chi_s} f_s \\
& = - M\bk{w_{2321},\chi_s} f_s .
\end{align*}
Namely, $M\bk{w_{232}}$ is an endomorphism of the image of $M\bk{w_{2321}}$ acting as $-Id$.
Following reason 2 for cancellation of poles we have
\[
\lim\limits_{s\to \frac{1}{2}} \bk{s-\frac{1}{2}}^2 \coset{M\bk{w_{2132132},\chi_s} + M\bk{w_{2321},\chi_s}} = 0 .
\]
In conclusion,
\[
\lim\limits_{s\to \frac{1}{2}} \bk{s-\frac{1}{2}}^2 \Eisen_E\bk{\chi, f_s, s, g}_{B_E} = 0 .
\]

We now turn to prove that the simple pole does not cancel.
It holds that
\begin{align*}
& \Sigma_{\bk{F\times K,\chi_K,\frac{1}{2},1}} = \set{w_{23}, w_{232}, w_{213}} \\
& w_{23}^{-1}\cdot\chi_{\frac{1}{2}}\bk{t} = \chi_K\bk{t_1t_2} \frac{\FNorm{t_1}_F}{\FNorm{t_3}_K},\quad w_{232}^{-1}\cdot\chi_{\frac{1}{2}}\bk{t} = \chi_K\bk{t_2} \frac{\FNorm{t_1}_F}{\FNorm{t_3}_K},\quad w_{213}^{-1}\cdot\chi_{\frac{1}{2}}\bk{t} = \chi_K\bk{t_1} \frac{\FNorm{t_2}_F}{\FNorm{t_1}_F \FNorm{t_3}_K} .
\end{align*}
Reason 2 for non-vanishing of the leading term then implies that $\Eisen_{F\times K}\bk{\chi_K, f_s^0, s, g}$ admits a simple pole at $s=\frac{1}{2}$.

\subsection{$E=F\times K$, $K$ a field, $s=\frac{1}{2}$, $\chi^2=\Id$, $\chi\neq \Id,\chi_K$}
The intertwining operators in this case have poles at most of order $1$.
We show that indeed the Eisenstein series admits a simple pole at this point.
We apply reason 1 for non-cancellation of poles, namely we construct a section $f_s\in I_{P_E}\bk{\chi,s}$ such that $\Eisen\bk{\chi,f_s,s,g}$ admits a simple pole at $s=\frac{1}{2}$.
We have
\[
\Sigma_{\bk{E,\chi,\frac{1}{2},1}} = \set{w_{2321}, w_{2132}, w_{21321}, w_{21323}, w_{213213}, w_{2132132}}
\]
and
\begin{align*}
& w_{2321}^{-1}\cdot\chi_{\frac{1}{2}}\bk{t} = w_{2132132}^{-1}\cdot\chi_{\frac{1}{2}}\bk{t}  = \chi\bk{t_2} \frac{\FNorm{t_2}_F}{\FNorm{t_1}_F \FNorm{t_3}_K} \\
& w_{2132}^{-1}\cdot\chi_{\frac{1}{2}}\bk{t} = w_{21323}^{-1}\cdot\chi_{\frac{1}{2}}\bk{t} = \chi\bk{t_1} \frac{1}{\FNorm{t_2}_F} \\
& w_{21321}^{-1}\cdot\chi_{\frac{1}{2}}\bk{t} = w_{213213}^{-1}\cdot\chi_{\frac{1}{2}}\bk{t} = \chi\bk{t_1t_2} \frac{1}{\FNorm{t_2}_F} .
\end{align*}

We write $w_{21323}=w_{2132}w_3$ and $w_{213213}=w_{21321}w_3$.
It follows from \Cref{Cor: Keys-Shahidi} that
\begin{align*}
& \lim\limits_{s\to \frac{1}{2}} \bk{s-\frac{1}{2}} \coset{M\bk{w_{21323},\chi_s} + M\bk{w_{2132},\chi_s}} = 0 \\
& \lim\limits_{s\to \frac{1}{2}} \bk{s-\frac{1}{2}} \coset{M\bk{w_{213213},\chi_s} + M\bk{w_{21321},\chi_s}} = 0 .
\end{align*}

On the other hand, 
\[
M\bk{w_{2132132},\chi_s} = M\bk{w_{232},w_{2321}^{-1}\cdot \chi_s}\circ M\bk{w_{2321},\chi_s} .
\]


We shall prove the existence of an holomorphic section $f_s\in I_{P_E}\bk{\chi,s}$ such that
\[
0\neq\lim\limits_{s\to\frac{1}{2}} \bk{s-\frac{1}{2}} M\bk{w_{2321},\chi_s} f_s \neq 
\lim\limits_{s\to\frac{1}{2}} \bk{s-\frac{1}{2}} M\bk{w_{2321232},\chi_s} f_s
\]
and hence
\[
\lim\limits_{s\to\frac{1}{2}} \bk{s-\frac{1}{2}} \coset{M\bk{w_{2321},\chi_s} + M\bk{w_{2321232},\chi_s}} f_s \neq 0 .
\]

For any $\nu\in\Places$ we denote
\[
R_\nu\bk{w_{232},\chi_{\nu}} = \lim\limits_{s\to\frac{1}{2}} N_\nu\bk{w_{232},w_{2321}^{-1}\cdot\chi_{s}} .
\]

In \Cref{App: Local Intertwining Operators} we prove:
\begin{itemize}
\item
For any $\nu\in\Places$ such that $\chi_\nu\neq \Id_\nu,\chi_{K_\nu}$ we have
\[
\Ind_{B_E\bk{F_\nu}}^{H_E\bk{F_\nu}} \bk{w_{2321}^{-1}\cdot\bk{\chi_\nu\circ \operatorname{det}_{M_E}}\otimes\lambda_{\frac{1}{2}}} = \Pi_{1,\nu} \oplus \Pi_{-1,\nu},
\]
where $N_\nu\bk{w_{2321},\chi_{\frac{1}{2}}}\bk{\Pi_{\epsilon,\nu}}$ is the $\epsilon$-eigenspace of $R_\nu\bk{w_{232},\chi_\nu}$; namely, $R_\nu\bk{w_{232},\chi_\nu}$ acts on $N_\nu\bk{w_{2321},\chi_{\frac{1}{2}}}\bk{\Pi_{\epsilon,\nu}}$ as $\epsilon Id$.

\item
For any $\nu$ such that $\bk{\chi\circ\chi_{K\rmod F}}_\nu$ is unramified then
\[
\exists v\in N_\nu\bk{w_{2321},\chi_\nu,\lambda_{\frac{1}{2}}}\bk{I_{P_E}\bk{\chi_\nu,\frac{1}{2}}}:\quad
R_\nu\bk{w_{232},\chi}v_\nu=v_\nu .
\]

\item
If $\nu\nmid\infty$ then there exists $v_\nu\in I_{P_E}\bk{\chi_\nu,\frac{1}{2}}$ such that $N_\nu\bk{w_{2321},\chi_\nu,\lambda_{\frac{1}{2}}}v_\nu\neq 0$ is not an eigenvector of $R_\nu\bk{w_{232},\chi_\nu}$.
\end{itemize}
For a place $\nu\mid\infty$ such that $K_\nu=\R\times\R$ and $\chi_\nu=\sgn_\nu$ we fix any $v_\nu \in I_{P_E}\bk{\sgn_\nu,\frac{1}{2}}$ such that $v_\nu\notin Ker\bk{N_\nu\bk{w_{2321},\chi_\nu,\lambda_{\frac{1}{2}}}}$.
By the computation in \Cref{App: Local Intertwining Operators}, there exists such $v_\nu$.

For any place $\nu\in\Places$ we let $f_{s,\nu}$ denote the standard section such that $f_{\frac{1}{2},\nu}=v_\nu$, where $v_\nu$ is as in the list above.
We denote the restricted tensor product $f_s=\placestimes f_{s,\nu}$ and note that $f_s$ is a standard section of $I_{P_E}\bk{\chi,s}$.

We recall that given two vectors $\placestimes x^{\bk{1}}_\nu\neq 0$ and $\placestimes x^{\bk{2}}_\nu$ we have $\placestimes x^{\bk{1}}_\nu=-\placestimes x^{\bk{1}}_\nu$ if and only if $x^{\bk{1}} = \alpha_\nu x^{\bk{2}}_\nu$ for any $\nu\in\Places$ and $\prodl_{\nu\in\Places} \alpha_\nu = -1$

We consider the character $\chi\circ\Nm_{K\rmod F}$ of $\Res_{K\rmod F}\bk{\A_F^\times}$, by the assumption $\chi\circ\Nm_{K\rmod F}\neq\Id$.
By the Strong Multiplicity One Theorem, \cite{MR546599}, there are infinitely many places $\nu$ such that $\chi_{\nu_0}\circ\Nm_{K_{\nu_0}\rmod F_{\nu_0}}\neq\Id_\nu$.
We fix $\nu_0\nmid \infty$ such that $\chi_\nu\neq\Id_\nu,\chi_{K_\nu}$.

It follows from the discussion that
\[
\lim\limits_{s\to\frac{1}{2}} \bk{s-\frac{1}{2}} M\bk{w_{2321},\chi_s} f_s \neq
- J\bk{w_{232},w_{2321}^{-1}\cdot\chi_{\frac{1}{2}}} \lim\limits_{s\to\frac{1}{2}} \bk{s-\frac{1}{2}} M\bk{w_{2321232},\chi_s} f_s .
\]
Indeed, note that the two vectors are pure tensors and write
\[
\placestimes x^{\bk{1}}_\nu = \lim\limits_{s\to\frac{1}{2}} \bk{s-\frac{1}{2}} M\bk{w_{2321},\chi_s} f_s,\quad
\placestimes x^{\bk{2}}_\nu = \lim\limits_{s\to\frac{1}{2}} \bk{s-\frac{1}{2}} M\bk{w_{2321232},\chi_s} f_s .
\]
By construction, it holds that:
\begin{itemize}
\item $\placestimes x^{\bk{1}}_\nu\neq 0$.
\item $\lim_{s\to\frac{1}{2}} J\bk{w_{232},w_{2321}^{-1}\cdot\chi_s}=1$.
\item There is no $\alpha_\nu\in\C$ such that $x^{\bk{1}}_{\nu_0} = \alpha_\nu x^{\bk{2}}_{\nu_0}$ .
\end{itemize}
Hence, the claim follows.


\begin{Remark}
If $\chi$ is unramified at all places then the global spherical section realizes the pole.
\end{Remark}

\subsection{$E=F\times K$, $K$ a field, $s=\frac{3}{2}$, $\chi=\Id$}
The intertwining operators in this case have poles at most of order $2$.
We show that the Eisenstein series admits a simple pole at this point.
We have
\begin{align*}
\Sigma_{\bk{F\times K,\chi_K,\frac{1}{2},2}} = \set{w_{213213}, w_{2132132}}
\end{align*}
and
\begin{align*}
w_{213213}^{-1}\cdot\chi_{\frac{1}{2}}\bk{t} = w_{2132132}^{-1}\cdot\chi_{\frac{1}{2}}\bk{t} = \frac{1}{\FNorm{t_1}_F \FNorm{t_3}_K} .
\end{align*}

We write $w_{2132132}=w_{213213}w_2$.
It follows from \Cref{Cor: Keys-Shahidi} that
\[
\lim\limits_{s\to \frac{3}{2}} \bk{s-\frac{3}{2}}^2 \coset{M\bk{w_{2132132},\chi_s} + M\bk{w_{213213},\chi_s}} = 0 .
\]
and hence $\lim\limits_{s\to \frac{3}{2}} \bk{s-\frac{3}{2}}^2 \Eisen_E\bk{\Id, f_s, s, g}_{B_E} = 0$.

We now turn to prove that the simple pole does not cancel.
It holds that
\begin{align*}
& \Sigma_{\bk{F\times K,\chi_K,\frac{1}{2},1}} = \set{w_{232}, w_{2321}, w_{21321}, w_{21323}} \\
& w_{232}^{-1}\cdot\chi_{\frac{1}{2}}\bk{t} = \frac{\FNorm{t_1}_F^3}{\FNorm{t_2}_F \FNorm{t_3}_K}, \\
& w_{2321}^{-1}\cdot\chi_{\frac{1}{2}}\bk{t} = \frac{\FNorm{t_2}_F^2}{\FNorm{t_1}_F^3 \FNorm{t_3}_K}, \\
& w_{21321}^{-1}\cdot\chi_{\frac{1}{2}}\bk{t} = \frac{\FNorm{t_3}_K}{\FNorm{t_1t_2^2}_F}, \\
& w_{21323}^{-1}\cdot\chi_{\frac{1}{2}}\bk{t} = \frac{\FNorm{t_1}_F}{\FNorm{t_2}_F \FNorm{t_3}_K} .
\end{align*}
Reason 2 for non-vanishing of the leading term then implies that $\Eisen_{F\times K}\bk{\chi_K, f_s^0, s, g}$ admits a simple pole at $s=\frac{1}{2}$.

\subsection{$E=F\times F\times F$, $s=\frac{1}{2}$, $\chi=\Id$}
The intertwining operators in this case have poles at most of order $4$.
We show that the Eisenstein series admits a simple pole at this point.
We have
\begin{align*}
& \Sigma_{\bk{F\times F\times F,\Id,\frac{1}{2},4}} = \set{w_{21342}, w_{213421}, w_{213423}, w_{213424}, w_{2134213}, w_{2134214}, w_{2134234}, w_{21342134}} \\
& \Sigma_{\bk{F\times F\times F,\Id,\frac{1}{2},3}} \setminus \Sigma_{\bk{F\times F\times F,\Id,\frac{1}{2},4}} = \set{w_{2134}, w_{21324}, w_{21423}, w_{23421}, w_{213421342}} \\
& \Sigma_{\bk{F\times F\times F,\Id,\frac{1}{2},2}} \setminus \Sigma_{\bk{F\times F\times F,\Id,\frac{1}{2},3}} = \set{w_{213}, w_{214}, w_{234}, w_{2132}, w_{2142}, w_{2342}} \\
& \Sigma_{\bk{F\times F\times F,\Id,\frac{1}{2},1}} \setminus \Sigma_{\bk{F\times F\times F,\Id,\frac{1}{2},2}} = \set{w_{21}, w_{23}, w_{24}} .
\end{align*}
and also
\begin{align*}
& w_{21342}^{-1}\cdot\chi_{\frac{1}{2}}\bk{t} = w_{213421}^{-1}\cdot\chi_{\frac{1}{2}}\bk{t} = w_{213423}^{-1}\cdot\chi_{\frac{1}{2}}\bk{t} = w_{213424}^{-1}\cdot\chi_{\frac{1}{2}}\bk{t} \\
& = w_{2134213}^{-1}\cdot\chi_{\frac{1}{2}}\bk{t} = w_{2134214}^{-1}\cdot\chi_{\frac{1}{2}}\bk{t} = w_{2134234}^{-1}\cdot\chi_{\frac{1}{2}}\bk{t} = w_{21342134}^{-1}\cdot\chi_{\frac{1}{2}}\bk{t} = \FNorm{\frac{1}{t_2}} \\
& w_{2134}^{-1}\cdot\chi_{\frac{1}{2}}\bk{t} = w_{23421}^{-1}\cdot\chi_{\frac{1}{2}}\bk{t} = w_{21423}^{-1}\cdot\chi_{\frac{1}{2}}\bk{t} = w_{21324}^{-1}\cdot\chi_{\frac{1}{2}}\bk{t} = w_{213421342}^{-1}\cdot\chi_{\frac{1}{2}}\bk{t} = \FNorm{\frac{t_2}{t_1t_3t_4}} \\
& w_{21}^{-1}\cdot\chi_{\frac{1}{2}}\bk{t} = \FNorm{\frac{t_3t_4}{t_1t_2}}, \quad w_{23}^{-1}\cdot\chi_{\frac{1}{2}}\bk{t} = \FNorm{\frac{t_1t_4}{t_2t_3}}, \quad w_{24}^{-1}\cdot\chi_{\frac{1}{2}}\bk{t} = \FNorm{\frac{t_1t_3}{t_2t_4}} \\
& w_{213}^{-1}\cdot\chi_{\frac{1}{2}}\bk{t} = w_{2132}^{-1}\cdot\chi_{\frac{1}{2}}\bk{t} = \FNorm{\frac{t_4}{t_1t_3}} \\
& w_{214}^{-1}\cdot\chi_{\frac{1}{2}}\bk{t} = w_{2142}^{-1}\cdot\chi_{\frac{1}{2}}\bk{t} = \FNorm{\frac{t_3}{t_1t_4}} \\
& w_{234}^{-1}\cdot\chi_{\frac{1}{2}}\bk{t} = w_{2342}^{-1}\cdot\chi_{\frac{1}{2}}\bk{t} = \FNorm{\frac{t_1}{t_3t_4}} .
\end{align*}
We note that it follows from the above that
\begin{align*}
\Sigma_{\bk{F\times F\times F,\Id,\frac{1}{2},2}}\rmod\sim_{s_0} = 
& \left\{ \Sigma_{\bk{F\times F\times F,\Id,\frac{1}{2},4}}, \Sigma_{\bk{F\times F\times F,\Id,\frac{1}{2},3}}\setminus\Sigma_{\bk{F\times F\times F,\Id,\frac{1}{2},4}}, \right. \\
& \left. \set{w_{213}, w_{2132}}, \set{w_{214}, w_{2142}}, \set{w_{234}, w_{2342}} \right\} .
\end{align*}
After proving that the poles of order 4, 3 and 2 cancel, the fact that the simple pole is attained follows from reason 2 for non-vanishing of the leading term. More precisely, if $f_s^0$ is the global spherical vector one has
\[
\lim\limits_{s\to\frac{1}{2}} \bk{s-\frac{1}{2}} M\bk{w_{21},\chi_s}f_s^0 \neq 0 .
\]
Namely, $\Eisen_{F\times F\times F}\bk{1, f_s^0, s, g}$ admits a simple pole at $s=\frac{1}{2}$.
We now turn to prove that the poles of higher order are canceled.
According to \Cref{Prop: Representation at 1/2 is generated by spherical vector} and \Cref{Conj: Archimedean places1} the representation $I_{P_E}\bk{\Id,\frac{1}{2}}$ is generated by $f_{\frac{1}{2}}^0$ where $f_s^0$ is the global spherical section.

We write
\[
\zfun\bk{s} = \frac{\gamma_{-1}}{s-1} + \gamma_0+\gamma_1\bk{s-1}...
\]
The functional equation, \Cref{Functional Equation - Zeta Function}, yields
\begin{equation}
\zfun\bk{s}=\zfun\bk{1-s} = -\frac{\gamma_{-1}}{s} + \gamma_0 - \gamma_1 - ...
\end{equation}
and also
\begin{equation}
\gamma_{-1} = \Res\bk{\zfun\bk{s},1} = -\Res\bk{\zfun\bk{s},0} .
\end{equation}

We note that for $w\in \Sigma_{\bk{F\times F\times F,\Id,\frac{1}{2},4}}$ one has
\[
M\bk{w} f_{s_0}^0\bk{t} = \frac{1}{\FNorm{t_2}_F} M\bk{w} f_{s_0}^0\bk{1} \quad \forall t\in T_E\bk{\A} .
\]
Hence, it is enough to consider the coefficients of the Laurent series of the corresponding intertwining operators at $t=1$.

\begin{align*}
\suml_{w\in \Sigma_{\bk{F\times F\times F,\Id,\frac{1}{2},4}}} M\bk{w}f_{s}^0\bk{1}
& = \frac{\zfun\bk{s+\frac{1}{2}}^3 \zfun\bk{2s}}{\zfun\bk{s+\frac{3}{2}}^2 \zfun\bk{s+\frac{5}{2}} \zfun\bk{2s+1}} +
3 \frac{\zfun\bk{s-\frac{1}{2}} \zfun\bk{s+\frac{1}{2}}^2 \zfun\bk{2s}}{\zfun\bk{s+\frac{3}{2}}^2 \zfun\bk{s+\frac{5}{2}} \zfun\bk{2s+1}} \\
& + 3 \frac{\zfun\bk{s-\frac{1}{2}}^2 \zfun\bk{s+\frac{1}{2}} \zfun\bk{2s}}{\zfun\bk{s+\frac{3}{2}}^2 \zfun\bk{s+\frac{5}{2}} \zfun\bk{2s+1}} +
\frac{\zfun\bk{s-\frac{1}{2}}^3 \zfun\bk{2s}}{\zfun\bk{s+\frac{3}{2}}^2 \zfun\bk{s+\frac{5}{2}} \zfun\bk{2s+1}} .
\end{align*}
As the denumerators are equal, holomorphic and non-vanishing at $s=\frac{1}{2}$, it is enough to prove that the sum of the numerators admits at most a simple pole.
Indeed,
\begin{align*}
& \zfun\bk{s+\frac{1}{2}}^3 \zfun\bk{2s} +
3 \zfun\bk{s-\frac{1}{2}} \zfun\bk{s+\frac{1}{2}}^2 \zfun\bk{2s} \\
& + 3 \zfun\bk{s-\frac{1}{2}}^2 \zfun\bk{s+\frac{1}{2}} \zfun\bk{2s} +
\zfun\bk{s-\frac{1}{2}}^3 \zfun\bk{2s} \\
& = \frac{\gamma_{-1}^4 - 3\gamma_{-1}^4 + 3\gamma_{-1}^4 - \gamma_{-1}^4}{2 \bk{s-\frac{1}{2}}^4} + \frac{5\gamma_{-1}^3\gamma_0 - 9 \gamma_{-1}^3\gamma_0 + 3\gamma_{-1}^3\gamma_0+\gamma_{-1}^3\gamma_0}{2 \bk{s-\frac{1}{2}}^3} \\
& + \frac{\bk{7\gamma_{-1}^3\gamma_1 + 9\gamma_{-1}^2\gamma_0} + 3\bk{-7\gamma_{-1}^3\gamma_1 - \gamma_{-1}^2\gamma_0} + 3 \bk{7\gamma_{-1}^3\gamma_1 - 3\gamma_{-1}^2\gamma_0} + \bk{-7\gamma_{-1}^3\gamma_1 + 3\gamma_{-1}^2\gamma_0}}{2 \bk{s-\frac{1}{2}}^2} \\
& + o\bk{\bk{s-\frac{1}{2}}^{-2}} = o\bk{\bk{s-\frac{1}{2}}^{-2}} .
\end{align*}
In particular,
\[
\lim\limits_{s\to \frac{1}{2}} \bk{s-\frac{1}{2}}^4 \Eisen_E\bk{1, f_s^0, s, g}_{B_E} = 0 .
\]

Similarly, for $w\in \Sigma_{\bk{F\times F\times F,\Id,\frac{1}{2},3}}$ one has
\[
M\bk{w} f_{s_0}^0\bk{t} = \frac{\FNorm{t_2}_F}{\FNorm{t_1t_3t_4}_F} M_w f_{s_0}^0\bk{1} \quad \forall t\in T_E\bk{\A} .
\]
We consider the leading term in the Laurent series the corresponding intertwining operators:
\begin{align*}
& \suml_{w\in \Sigma_{\bk{F\times F\times F,\Id,\frac{1}{2},4}}} M\bk{w}f_{s}^0\bk{1} \\
& = \frac{\zfun\bk{s+\frac{1}{2}}^3}{\zfun\bk{s+\frac{3}{2}}^2 \zfun\bk{s+\frac{5}{2}}} +
3 \frac{\zfun\bk{s-\frac{1}{2}} \zfun\bk{s+\frac{1}{2}} \zfun\bk{2s}}{\zfun\bk{s+\frac{3}{2}} \zfun\bk{s+\frac{5}{2}} \zfun\bk{2s+1}} +
\frac{\zfun\bk{s-\frac{1}{2}}^2 \zfun\bk{s-\frac{3}{2}} \zfun\bk{2s}}{\zfun\bk{s+\frac{5}{2}} \zfun\bk{s+\frac{3}{2}}^2 \zfun\bk{2s+1}} \\
& = \frac{\zfun\bk{s+\frac{1}{2}}^3\zfun\bk{2s+1} + 3 \zfun\bk{s-\frac{1}{2}} \zfun\bk{s+\frac{1}{2}} \zfun\bk{2s}\zfun\bk{s+\frac{3}{2}} + \zfun\bk{s-\frac{1}{2}}^2 \zfun\bk{s-\frac{3}{2}} \zfun\bk{2s}}{\zfun\bk{s+\frac{5}{2}} \zfun\bk{s+\frac{3}{2}}^2 \zfun\bk{2s+1}} .
\end{align*}
It is enough to prove that the numerator admits at most a simple pole.
Indeed,
\begin{align*}
& \zfun\bk{s+\frac{1}{2}}^3\zfun\bk{2s+1} + 3 \zfun\bk{s-\frac{1}{2}} \zfun\bk{s+\frac{1}{2}} \zfun\bk{2s}\zfun\bk{s+\frac{3}{2}} \\
& + \zfun\bk{s-\frac{1}{2}}^2 \zfun\bk{s-\frac{3}{2}} \zfun\bk{2s} \\
& = \frac{\zfun\bk{2}\gamma_{-1}^3 - \frac{3}{2} \zfun\bk{2} \gamma_{-1}^3 + \frac{1}{2}\zfun\bk{-1}\gamma_{-1}^3}{\bk{s-\frac{1}{2}}^3} \\
& + \frac{\bk{3a_0\gamma_{-1}^2\gamma_0 + 2a_1\gamma_{-1}^3} - 3 \bk{a_0\gamma_{-1}^2\gamma_0+\frac{1}{2}\gamma_{-1}^3} -\frac{1}{2}a_1\gamma_{-1}^3 }{\bk{s-\frac{1}{2}}^2} \\
& + o\bk{\bk{s-\frac{1}{2}}^{-2}} = o\bk{\bk{s-\frac{1}{2}}^{-2}}.
\end{align*}
Here
\[
\zfun\bk{2+\epsilon} = a_0 + a_1\epsilon + o\bk{\epsilon},\quad a_0 = \zfun\bk{2} .
\]

We now turn to the operators with a double pole at $s=\frac{1}{2}$.
Here we have three equivalence classes of $\sim_{s_0}$ in $\Sigma_{\bk{F\times F\times F,\Id,\frac{1}{2},2}} \setminus \Sigma_{\bk{F\times F\times F,\Id,\frac{1}{2},3}}$.
We then have
\begin{align*}
& \coset{M\bk{w_{213}} + M\bk{w_{2132}}}f_s^0\bk{1} \\
& = \coset{M\bk{w_{214}} + M\bk{w_{2142}}}f_s^0\bk{1} \\
& = \coset{M\bk{w_{234}} + M\bk{w_{2342}}}f_s^0\bk{1} \\
& = \frac{\zfun\bk{s+\frac{1}{2}}^2}{\zfun\bk{s+\frac{3}{2}} \zfun\bk{s+\frac{5}{2}}} + \frac{\zfun\bk{s+\frac{1}{2}}\zfun\bk{s-\frac{1}{2}}}{\zfun\bk{s+\frac{5}{2}}\zfun\bk{s+\frac{3}{2}}}
\end{align*}
As the denumerators are equal, holomorphic and non-vanishing at $s=\frac{1}{2}$, it is enough to prove that the sum of the numerators admits at most a simple pole.
Indeed,
\begin{align*}
& \zfun\bk{s+\frac{1}{2}}^2 + \zfun\bk{s+\frac{1}{2}}\zfun\bk{s-\frac{1}{2}} \\
& = \frac{\zfun\bk{2}^2\gamma_{-1}^2 - \zfun\bk{2}^2\gamma_{-1}^2}{\bk{s-\frac{1}{2}}^2} + o\bk{\bk{s-\frac{1}{2}}^{-2}} = o\bk{\bk{s-\frac{1}{2}}^{-2}}.
\end{align*}

In conclusion, $\Eisen_E\bk{1, f_s^0, s, g}$ admits a pole of order $1$ at $s=\frac{1}{2}$.

\begin{Remark}
Applying reason 2 for cancellation of poles one can show, independently \Cref{Conj: Archimedean places1}, that the pole here is of order at most $3$.
\end{Remark}

\subsection{$E=F\times F\times F$, $s=\frac{1}{2}$, $\chi^2=\Id$, $\chi\neq \Id$}
The intertwining operators in this case have poles of order at most $1$.
We show that indeed the Eisenstein series admits a simple pole at this point.
Here we apply reason 2 for non-cancellation of poles.
We have
\begin{align*}
\Sigma_{\bk{F\times F\times F,\chi,\frac{1}{2},1}}
= & \left\{ w_{21324}, w_{21423}, w_{23421}, w_{21342}, w_{213421}, w_{213423}, w_{213424}, \right. \\
& \left. w_{2134213}, w_{2134214}, w_{2134234}, w_{21342134}, w_{213421342} \right\}
\end{align*}
and
\begin{align*}
& w_{21324}^{-1}\cdot\chi_{\frac{1}{2}}\bk{t} = w_{21423}^{-1}\cdot\chi_{\frac{1}{2}}\bk{t} = w_{23421}^{-1}\cdot\chi_{\frac{1}{2}}\bk{t} = w_{213421342}^{-1}\cdot\chi_{\frac{1}{2}}\bk{t} =\chi\bk{t_2}\FNorm{\frac{t_2}{t_1t_3t_4}} \\
& w_{21342}^{-1}\cdot\chi_{\frac{1}{2}}\bk{t} = w_{2134213}^{-1}\cdot\chi_{\frac{1}{2}}\bk{t} = w_{2134214}^{-1}\cdot\chi_{\frac{1}{2}}\bk{t} = w_{2134234}^{-1}\cdot\chi_{\frac{1}{2}}\bk{t} =\chi\bk{t_1t_3t_4}\FNorm{\frac{1}{t_2}} \\
& w_{213421}^{-1}\cdot\chi_{\frac{1}{2}}\bk{t} = w_{213423}^{-1}\cdot\chi_{\frac{1}{2}}\bk{t} = w_{213424}^{-1}\cdot\chi_{\frac{1}{2}}\bk{t} = w_{21342134}^{-1}\cdot\chi_{\frac{1}{2}}\bk{t} =\chi\bk{t_1t_2t_3t_4}\FNorm{\frac{1}{t_2}} .
\end{align*}

We write
We note that
\begin{align*}
& w_{213421342} = w_{21324} w_{2132} = w_{21423} w_{2142} = w_{23421} w_{2342} \\
& w_{21342134} = w_{213424} w_{13} = w_{213423} w_{14} = w_{213421} w_{34} \\
& w_{2134213} = w_{21342} w_{13} \\
& w_{2134213} = w_{21342} w_{14} \\
& w_{2134213} = w_{21342} w_{34} .
\end{align*}

According to \Cref{Cor: Keys-Shahidi}, we conclude that
\begin{align*}
& \lim\limits_{s\to \frac{1}{2}} \bk{s-\frac{1}{2}} M\bk{w_{21342},\chi_s} = \lim\limits_{s\to \frac{1}{2}} \bk{s-\frac{1}{2}} M\bk{w_{2134213},\chi_s} \\
&\qquad\qquad = \lim\limits_{s\to \frac{1}{2}} \bk{s-\frac{1}{2}} M\bk{w_{2134214},\chi_s} = \lim\limits_{s\to \frac{1}{2}} \bk{s-\frac{1}{2}} M\bk{w_{2134234},\chi_s} , \\
& \lim\limits_{s\to \frac{1}{2}} \bk{s-\frac{1}{2}} M\bk{w_{213421},\chi_s} = \lim\limits_{s\to \frac{1}{2}} \bk{s-\frac{1}{2}} M\bk{w_{213423},\chi_s} \\
&\qquad\qquad = \lim\limits_{s\to \frac{1}{2}} \bk{s-\frac{1}{2}} M\bk{w_{213424},\chi_s} = \lim\limits_{s\to \frac{1}{2}} \bk{s-\frac{1}{2}} M\bk{w_{21342134},\chi_s} .
\end{align*}

We do not treat the rest of the terms as they have different exponents; in particular, the pole of order $1$ does not cancel.

\subsection{Square Integrability of the Residual Representations}

We now determine, for a point where $\Eisen_E\bk{\chi, f_s, s, g}$ admits a pole, whether the residual representation is square-integrable or not.
Before doing so, we recall the following criterion from \cite[pg. 104]{MR0579181}.

For $w\in W\bk{P_E,H_E}$ the element $\Real\bk{w^{-1}\cdot\chi_s}\in \mathfrak{a}_\R^\ast$ is known as the \emph{exponent of $I_{P_E}\bk{\chi,s}$ corresponding to $w$}.

Assume $\Eisen_E\bk{\chi, f_s, s, g}$ admits a pole of order $n$ at $s_0$.
We recall the equivalence relation defined in \Cref{Equivalence relation on Sigma-m} and define the quotient set
\[
\Sigma_{s_0} = \Sigma_{\bk{E,\chi,s_0,n}}\rmod \sim_{s_0} .
\]
Note that the exponent is well defined for equivalence classes, namely $\Real\bk{w^{-1}\cdot\chi_s}=\Real\bk{w'^{-1}\cdot\chi_s}$ when $w\sim_{s_0} w'$.
And we consider the elements contributing to the residual representation at $s_0$, namely:
\[
\Sigma_{s_0}^0 = \set{\Omega \in \Sigma_{s_0} \mvert \lim\limits_{s\to s_0} \bk{s-s_0}^n \suml_{w\in \Omega} M\bk{w,\chi_s} \not\equiv 0 } .
\]

\begin{Lem}[Langlands' Criterion for Square Integrability]
	Assume $\Eisen_E\bk{\chi, f_s, s, g}$ admits a pole of order $n$ at $s_0$.
	The residual representation $\Res_{s=s_0} \Eisen_E\bk{\chi, f_s, s, g}$ is a square-integrable representation if and only if $\Real\bk{\Omega^{-1}\cdot\chi_s}<0$ for all $\Omega\in \Sigma_{s_0}^0$.
\end{Lem}

\begin{Cor}
	The residual representation of $\Eisen_E\bk{\chi, \cdot , s, \cdot }$ is square-integrable with the exception of the following case:
	\begin{itemize}
		\item $E=F\times K$ where $K$ is a field:
		\begin{itemize}
			\item $s=\frac{1}{2}$ with $\chi=\Id,\chi_K$.
			\item $s=\frac{3}{2}$ with $\chi=\Id$.
		\end{itemize}
		\item $E=F\times F\times F$, $s=\frac{1}{2}$ with $\chi=\Id$.
	\end{itemize}

\end{Cor}

This follows from the proof of \Cref{Thm: Poles of Eisenstein series}, from Langlands' criterion to square integrability and from the information in Tables \ref{Table: Exponents, Cubic}, \ref{Table: Exponents, Quadratic} and \ref{Table: Exponents, Split}.

\end{proof}

\section{Special Cases of Local Intertwining Operators}
\label{App: Local Intertwining Operators}
The proof of \Cref{Thm: Poles of Eisenstein series} requires an analysis of the behavior of the following global intertwining operators:
\begin{enumerate}
	\item $\lim_{s\to\frac{1}{2}} M\bk{w_{212},w_{21}^{-1}\cdot\chi_s}$ when $E$ is a Galois field extension of $F$ and $\chi=\chi_E$.
	\item $\lim_{s\to\frac{1}{2}} M\bk{w_{232},w_{2321}^{-1}\cdot\chi_s}$ when $E=F\times K$ and $\chi$ is a non-trivial quadratic character.
\end{enumerate}
Applying \Cref{Eq: Global Gindikin-Karpelevich} we may write
\[
\begin{split}
& M\bk{w_{212},w_{21}^{-1}\cdot\chi_s} = J\bk{w,\chi_s} \coset{\placestimes N_\nu\bk{w_{212},w_{21}^{-1}\cdot\chi_s}}, \\
& M\bk{w_{232},w_{2321}^{-1}\cdot\chi_s} = J\bk{w,\chi_s} \coset{\placestimes N_\nu\bk{w_{232},w_{2321}^{-1}\cdot\chi_s}} .
\end{split}
\]
In this section we treat the behavior of these normalized local intertwining operators at $s=\frac{1}{2}$.
We make the following notations:
\begin{enumerate}
	\item The intertwining operator $M\bk{w_{212},w_{21}^{-1}\cdot\chi_s}$ when $E$ is a field, $s=\frac{1}{2}$ and $\chi=\chi_E$ (In particular, $E\rmod F$ is a Galois extension).
	In this case, for $\nu\in\Places$, we denote
	\[
	R_\nu\bk{w_{212},\chi_{E_\nu}} = \lim\limits_{s\to\frac{1}{2}} N\bk{w_{212},w_{21}^{-1}\cdot\chi_{s,\nu}} =
	\lim_{s\to\frac{1}{2}} \coset{N\bk{w_{212},w_{21}^{-1}\cdot\lambda} \res{\lambda=\lambda_s} } .
	\]
	
	\item The intertwining operator $M\bk{w_{232},w_{2321}^{-1}\cdot\chi_s}$ when $E=F\times K$, $s=\frac{1}{2}$ and $\chi^2=\Id$ with $\chi\neq 1$.
	In this case, for $\nu\in\Places$, we denote
	\[
	R_\nu\bk{w_{232},\chi_\nu} = \lim\limits_{s\to\frac{1}{2}} N\bk{w_{232},w_{2321}^{-1}\cdot\chi_{s,\nu}} =
	\lim_{s\to\frac{1}{2}} \coset{N\bk{w_{232},w_{2321}^{-1}\cdot\lambda} \res{\lambda=\lambda_s} } .
	\]
\end{enumerate}
We calculate the action of $R_\nu\bk{w_{212},\chi_{E_\nu}}$ on $N_\nu\bk{w_{21},\lambda_{\frac{1}{2}}}\coset{I_{P_E}\bk{\chi_{E_\nu},\frac{1}{2}}}$ and the action of $R_\nu\bk{w_{232},\chi_\nu}$ on $N_\nu\bk{w_{2321},\lambda_{\frac{1}{2}}}\coset{I_{P_E}\bk{\chi_{\nu},\frac{1}{2}}}$.

We note that for $\nu\in\Places$, according to Tables \ref{Table: Action of W on characters, Cubic}, \ref{Table: Action of W on characters, Quadratic} and \ref{Table: Action of W on characters, Split}, $R_\nu\bk{w_{212},\chi_{E_\nu}}$ is an endomorphism of 
$$\Ind_{P_E^{w_{21}}\bk{F_\nu}}^{H_E\bk{F_\nu}} w_{21}^{-1}\cdot \bk{\chi_{E_\nu}\circ\operatorname{det}_{M_E}\otimes \FNorm{det_{M_E}}^{s+\frac{5}{2}}}$$
and $R_\nu\bk{w_{232},\chi_\nu}$ is an endomorphism of 
$$\Ind_{P_E^{w_{2321}}\bk{F_\nu}}^{H_E\bk{F_\nu}} w_{2321}^{-1}\cdot \bk{\chi_\nu\circ\operatorname{det}_{M_E}\otimes \FNorm{det_{M_E}}^{s+\frac{5}{2}}} . $$

In each of the two cases above, for $\nu\in\Places$, $E_\nu$ and $\chi_\nu$ can take the following values:
\begin{enumerate}
	\item
	\begin{itemize}
		\item $E_\nu=F_\nu\times F_\nu \times F_\nu$ and $\chi_{E,\nu} = \Id_\nu$.
		\item $E_\nu$ is a Galois field extension of $F_\nu$ and $\chi_{E,\nu} = \chi_{E_\nu}$.
	\end{itemize}
	\item
	If $\chi=\chi_K$:
	\begin{itemize}
		\item $K_\nu=F_\nu\times F_\nu$ and $\chi_{K,\nu} = \Id_\nu$.
		\item $K_\nu$ and $\chi_{K,\nu}=\chi_{K_\nu}$.
	\end{itemize}
	If $\chi\neq \chi_K$:
	\begin{itemize}
		\item $K_\nu=F_\nu\times F_\nu$ and $\chi_\nu=\Id_\nu$.
		\item $K_\nu=F_\nu\times F_\nu$ and $\chi_\nu^2=\Id$ with $\chi_\nu\neq \Id_\nu$.
		\item $K_\nu$ is a field and $\chi_\nu=\Id_\nu$.
		\item $K_\nu$ is a field and $\chi_\nu=\chi_{K_\nu}$.
		\item $K_\nu$ is a field and $\chi_\nu^2=\Id$ with $\chi_\nu\neq \Id_\nu, \chi_{K_\nu}$.
	\end{itemize}
\end{enumerate}

\begin{Remark}
We note that when $E_\nu=F_\nu\times F_\nu\times F_\nu$ we have
\[
\begin{split}
& R_\nu\bk{w_{212},\chi_{E_\nu}} = \lim\limits_{s\to\frac{1}{2}} N\bk{w_{21342},w_{2134}^{-1}\cdot\chi_{s,\nu}} \\
& R_\nu\bk{w_{232},\chi_{\nu}} = \lim\limits_{s\to\frac{1}{2}} N\bk{w_{2342},w_{23421}^{-1}\cdot\chi_{s,\nu}} .
\end{split}
\]
\end{Remark}

We summarize the results of this section in the following theorem:
\begin{Thm}
\label{Thm: Results regarding local intertwining operators}
Fix $\nu\in\Places$.
It holds that:
\begin{enumerate}
	\item
	If $E$ is a Galois field extension of $F$ and $\chi=\chi_E$ then
	\[
	R_\nu\bk{w_{212},\chi_{E_\nu}} v = v \quad \forall v\in v\in N_\nu\bk{w_{21},\chi_s}\bk{I_{P_E}\bk{\chi_\nu,\frac{1}{2}}} .
	\]
	 

	\item
	Assume that $E=F\times K$, where $K$ is a quadratic \'etale algebra over $F$.
	\begin{itemize}
	\item
	If $\chi_\nu=\Id_\nu$ then 
	\[
	R_\nu\bk{w_{232},\chi}v=v \quad \forall v\in N_\nu\bk{w_{2321},\chi_{\nu},\lambda_{\frac{1}{2}}}\bk{I_{P_E}\bk{\chi_\nu,\frac{1}{2}}} .
	\]
	\item 
	If $\chi_\nu$ is unramified or $\chi_\nu=\chi_{K_\nu}$ then
	\[
	\exists v\in N_\nu\bk{w_{2321},\chi_\nu,\lambda_{\frac{1}{2}}}\bk{I_{P_E}\bk{\chi_\nu,\frac{1}{2}}}:\quad
	R_\nu\bk{w_{232},\chi}v=v
	\]	
	
	\item
	If $\nu\nmid\infty$ then there exists $v\in I_{P_E}\bk{\chi_\nu,\frac{1}{2}}$ such that $N_\nu\bk{w_{2321},\chi_\nu,\lambda_{\frac{1}{2}}}v\neq 0$ is not an eigenvector of $R_\nu\bk{w_{232},\chi_\nu}$.
	\end{itemize}
\end{enumerate}
\end{Thm}
For the rest of this section we fix a place $\nu\in\Places$ and drop $\nu$ from all notations.

\subsection{$E$ is a Field}
For $E$ a field, we need only to consider the case where $\chi=\chi_E$ and the intertwining operator $M\bk{w_{212},w_{21}^{-1}\cdot\bk{\chi_E}_s}$.
We note that in this case, $E$ is automatically non-Archimedean.

\begin{Lem}
\label{Lem: Local intertwining Operator E is a field}
	For any $v\in N\bk{w_{21},\chi_s}\bk{I_{P_E}\bk{\chi_E,\frac{1}{2}}}$ it holds that $R\bk{w_{212},\chi_E}v=v$.
\end{Lem}

\begin{proof}
	We start by recalling that in a neighborhood of $s_0=\frac{1}{2}$ we have
	\[
	N\bk{w_{212},w_{21}^{-1}\cdot\chi_s} = N\bk{w_2,w_{2121}^{-1}\cdot\chi_s} \circ N\bk{w_1,w_{212}^{-1}\cdot\chi_s} \circ N\bk{w_2,w_{21}^{-1}\cdot\chi_s} .
	\]
	On the other hand, note that
	\[
	w_{212}^{-1}\cdot\chi_{s}\bk{h_{\alpha_1}\bk{t_1}h_{\alpha_3}\bk{t_1^\sigma}h_{\alpha_4}\bk{t_1^{\sigma^2}}} = \FNorm{t_1}^{s-\frac{1}{2}}
	\]
	and hence, for any $s\in\C$, $\Ind_{B_E\bk{F}}^{M_E\bk{F}} w_{212}^{-1}\cdot\chi_s$ is an unramified representation of the standard Levi subgroup $\gen{T,x_{\alpha_1}\bk{r},x_{-\alpha_1}\bk{r}\mvert r\in F}$.
	In particular
	\[
	w_{212}^{-1}\cdot\chi_{\frac{1}{2}}\bk{h_{\alpha_1}\bk{t_1}h_{\alpha_3}\bk{t_1^\sigma}h_{\alpha_4}\bk{t_1^{\sigma^2}}} = 1 .
	\]
	It follows that $N\bk{w_1,w_{212}^{-1}\cdot\chi_{\frac{1}{2}}}$ acts on $\Ind_{B_E\bk{F}}^{H_E\bk{F}} w_{212}^{-1}\cdot\chi_{\frac{1}{2}}$ as $Id$.
	The claim then follows from \Cref{Lem: Local composition of rank-one intertwining operators}, namely
	\[
	N\bk{w_2,w_{2121}^{-1}\cdot\chi_{\frac{1}{2}}} \circ N\bk{w_2,w_{21}^{-1}\cdot\chi_{\frac{1}{2}}} = Id .
	\]
\end{proof}

\subsection{$E=F\times K$}

\subsubsection{The Case of $\chi=\Id$}
\begin{Prop}
	\label{Case of trivial character E non-split}
	For any $v\in N\bk{w_{2321},\Id,\lambda_{\frac{1}{2}}}\bk{I_{P_E}\bk{\Id,\frac{1}{2}}}$ it holds that $R\bk{w_{232},\Id}v = v$.
\end{Prop}

\begin{proof}
	When $F$ is non-Archimedean, $I_{P_E}\bk{\chi, s}$ is generated by the spherical vector $v^0$, due to \Cref{Prop: Representation at 1/2 is generated by spherical vector}.
	When $F$ is Archimedean, this is the content of \Cref{Conj: Archimedean places1}.
	Hence $N\bk{w_{2321},\chi_s}\bk{I_{P_E}\bk{\chi, s}}\neq\set{0}$ is also generated by the spherical vector.
	Indeed, the image of the normalized spherical vector of $I_{P_E}\bk{\chi, \frac{1}{2}}$ under $N_\nu\bk{w_{2321},\chi_{\nu},\lambda_{\frac{1}{2}}}$ is a non-zero spherical vector since $J\bk{w_{232},\chi_{\frac{1}{2}}}\in\C^\times$.
	By the definition of the normalized intertwining operator, $N\bk{w_{232},w_{2321}^{-1}\cdot\chi_s}v^0 = v^0$ from which the claim follows.
\end{proof}

\subsubsection{The Case of $\chi=\chi_K$}

\begin{Lem}
	\label{Lem: chi=chi_K intertwining operator}
	For any $v\in N\bk{w_{2321},\chi_K,\lambda_{\frac{1}{2}}}\bk{I_{P_E}\bk{\chi_K,\frac{1}{2}}}$ it holds that $R\bk{w_{232},\chi_K}v = v$.
\end{Lem}

The proof of this is similar to that of \Cref{Lem: Local intertwining Operator E is a field}.

\subsubsection{The Case of $\chi^2=\Id$ with $\chi\neq 1, \chi_K$}
\label{Subsubsection: Quadratic character non-split case}
Consider $\lambda=\bk{s_1,s_2,s_3}\in\mathfrak{a}_\C^\ast$.
In this case,
\[
w_{2321}^{-1}\cdot\bk{\bk{\chi\circ\operatorname{det}_{M_E}}\otimes\lambda} \bk{t_1,t_2,t_3} = 
\chi\bk{t_2} \FNorm{t_1}_F^{-s_1-2s_2-2s_3} \FNorm{t_2}_F^{s_1+s_2} \FNorm{t_3}_K^{s_3} .
\]
Write
\[
\begin{split}
& N\bk{w_{232},w_{2321}^{-1}\cdot\chi,w_{2321}^{-1}\cdot\lambda} = \\ 
& N\bk{w_2,w_{232123}^{-1}\cdot\chi,w_{232123}^{-1}\cdot\lambda} \circ N\bk{w_3,w_{23212}^{-1}\cdot\chi,w_{23212}^{-1}\cdot\lambda} \circ N\bk{w_2,w_{2321}^{-1}\cdot\chi,w_{2321}^{-1}\cdot\lambda} .
\end{split}
\]
We fix $\lambda_{\frac{1}{2}}=\bk{-1,2,-1}\in\mathfrak{a}_\C^\ast$.

Let $P_3$ be the parabolic subgroup of $H_E$ whose Levi subgroup $M_{3}$ is generated by $\alpha_3+\alpha_4^\sigma$.
The Levi $M_3$ is isomorphic to $GL_1\times \bk{Res_{K\rmod F}GL_2}^0$, where
\[
\bk{Res_{K\rmod F}GL_2}^0\bk{F} = \set{g\in GL_2\bk{K}\mvert \operatorname{det} g\in F^\times} .
\]
Note that $M_{3}\cap B_E=GL_1\times \mathcal{B}^0$, where
\[
\mathcal{B}^0 = \mathcal{B}\bk{K}\cap \bk{Res_{K\rmod F}GL_2}^0\bk{F} .
\]

As a consequence of this,
\[
\Ind_{M_3\cap B_E\bk{F}}^{M_3\bk{F}} \coset{w_{23212}^{-1}\cdot\bk{\bk{\chi\circ \operatorname{det}_{M_E}}\otimes\lambda_{\frac{1}{2}}} \res{GL_1\times \Res_{K\rmod F}SL_2}} = \pi^{\bk{1}} \oplus \pi^{\bk{-1}} ,
\]
where $\pi^{\bk{\epsilon}}$ are as in \Cref{Subsec: Rank-one}.
It follows from the discussion there that $N\bk{w_3,w_{23212}^{-1}\cdot\chi,w_{23212}^{-1}\cdot\lambda}$ acts on $\pi^{\bk{\epsilon}}$ by $\epsilon Id$.

We let
\begin{align*}
& \Pi_1 = \Ind_{P_3\bk{F}}^{H_E\bk{F}} \pi^{\bk{1}} \\
& \Pi_{-1} = \Ind_{P_3\bk{F}}^{H_E\bk{F}} \pi^{\bk{-1}}
\end{align*}
and then
\[
\Ind_{B_E\bk{F}}^{H_E\bk{F}} \bk{w_{2321}^{-1}\cdot\bk{\chi\circ \operatorname{det}_{M_E}}\otimes\lambda_{\frac{1}{2}}} = \Pi_1 \oplus \Pi_{-1}
\]
is a direct sum of two representations.
It then holds that
\begin{itemize}
	\item
	$N\bk{w_2, w_{2321}^{-1}\cdot\chi, w_{2321}^{-1}\cdot\lambda}$  is holomorphic and bijective at $\lambda_{\frac{1}{2}}$.
	\item
	$N\bk{w_{3}, w_{23212}^{-1}\cdot\chi, w_{23212}^{-1}\cdot\lambda}$ acts on $\Pi_\epsilon$ as $\epsilon Id$ due to \Cref{Lemma: intertwining operator of simple reflections}.
	\item
	$N\bk{w_2, w_{232123}^{-1}\cdot\chi, w_{232123}^{-1}\cdot\lambda}$  is holomorphic and bijective at $\lambda_{\frac{1}{2}}$.
\end{itemize}

It follows from \Cref{Langlands identity for composition of intertwining operators} that $N\bk{w_{232}, w_{2321}^{-1}\cdot\chi, w_{2321}^{-1}\cdot\lambda_{\frac{1}{2}}}$ acts on $N\bk{w_{2321}}\bk{\Pi_\epsilon}$ as $\epsilon Id$.

We now compute the action of $R\bk{w_{232},\chi}$ on $N\bk{w_{2321}}\bk{I_{P_E}\bk{\chi,\frac{1}{2}}}$ in the following lemma.
\begin{Lem}
	\label{Lem: Local ramified character non-split}
	The following holds:
	\begin{enumerate}
	\item
	If $\chi$ is unramified then there exists $v\in N\bk{w_{2321},\chi_{\frac{1}{2}}}\bk{I_{P_E}\bk{\chi, {\frac{1}{2}}}}$ such that $R\bk{w_{232},\chi}v = v$.
	\item
	If $F$ is non-Archimedean then there exists $v\in I_{P_E}\bk{\chi,\frac{1}{2}}$ such that $N\bk{w_{2321},\chi_{\frac{1}{2}}}v\neq 0$ is not an eigenvector of $R\bk{w_{232},\chi}$.
	\end{enumerate}
\end{Lem}

\begin{proof}
	\begin{enumerate}
	\item
	If $\chi$ is unramified then $I_{P_E}\bk{\chi,s}$ admits a non-zero spherical section $f_s^0$ and so
	\[
	N\bk{w_{232},w_{2321}^{-1}\cdot\chi_s}f_s^0 = f_s^0 \quad \forall s\in\C .
	\]
	In particular, for $v=f_{\frac{1}{2}}^0$ we have $R\bk{w_{232},\chi}v = v$.
	
	\item
	We recall a corollary to the results of \cite[Section 6.3]{Casselman}:
	\begin{Cor}
	Let $Q=L\cdot V$ be a standard parabolic subrgoup of $H_E$
	Let $\Omega$ be an admissible representation of $L$.
	Then, the Jacquet functor $\mathcal{J}_{B_E}^{H_E} \bk{Ind_{Q}^{H_E} \Omega}$ of $Ind_{Q}^{H_E} \Omega$ (normalized induction) has a composition series with factors $w^{-1}\cdot \mathcal{J}_{B_E\cap L}^{L}\Omega$, where $w$ runs over the set of minimal representatives of the cosets of $W_{L} \rmod W_{H_E}$.
	\end{Cor}
	
	We note that:
	\begin{itemize}
	\item
	$N\bk{w_{232},\chi,\lambda_{\frac{1}{2}}}$ is an isomorphism.
	\item
	$N\bk{w_{1},w_{232}^{-1}\cdot\lambda_{\frac{1}{2}}}:\Ind_{B\bk{F}}^{H_E\bk{F}}\bk{\chi\circ\alpha_1}\otimes w_{232}^{-1}\cdot\lambda_{\frac{1}{2}}\to \Ind_{B\bk{F}}^{H_E\bk{F}}\bk{\chi\circ\alpha_1}\otimes w_{2321}^{-1}\cdot\lambda_{\frac{1}{2}}$ is not an isomorphism.
	The kernel and image are the parabolic inductions from the parabolic subgroup $P_1$, whose Levi subgroup is $L_1=\gen{T,x_{\alpha_1}\bk{r},x_{-\alpha_1}\bk{r}\mvert r\in F}$, to $H_E\bk{F}$ associated with the Steinberg and trivial representations respectively.
	Namely, we have a short exact sequence
	\[
	Ind_{P_1}^{H_E}\bk{St_{L_1}\otimes \Omega} \hookrightarrow \Ind_{B\bk{F}}^{H_E\bk{F}}\bk{\chi\circ\alpha_1}\otimes w_{232}^{-1}\cdot\lambda_{\frac{1}{2}} \twoheadrightarrow Ind_{P_1}^{H_E}\bk{\Omega} ,
	\]
	where $\Omega$ is a character of $L_1$ such that $\mathcal{J}_{B_E\cap L_1}^{L_1} \Omega = \bk{\chi\circ\alpha_1}\otimes w_{232}^{-1}\cdot\lambda_{\frac{1}{2}}$.
	\item The Jacquet modules $\mathcal{J}_{B_E}^{H_E}\Pi_1$ and $\mathcal{J}_{B_E}^{H_E}\Pi_{-1}$ of $\Pi_1$ and $\Pi_{-1}$ are isomorphic.
	\end{itemize}
	The claim then follows from the fact that the Jacquet module is exact and the multiplicities of $\Lambda=\bk{\chi\circ\omega_2}\otimes\bk{-1,1,-1}$ in various representations, given as follows:
	\begin{itemize}
	\item $m_\Lambda\bk{\Ind_{B\bk{F}}^{H_E\bk{F}}\bk{\chi\circ det_{M_E}}\otimes \lambda_{\frac{1}{2}}} = m_\Lambda\bk{\Ind_{B\bk{F}}^{H_E\bk{F}}\bk{\chi\circ\alpha_1}\otimes w_{2321}^{-1}\cdot\lambda_{\frac{1}{2}}} = 2$.
	\item $m_{\Lambda}\bk{I_{P_E}\bk{\chi,\frac{1}{2}}} = 2$.
	\item $m_\Lambda\bk{Im N\bk{w_{23212},\chi,\lambda_{\frac{1}{2}}}} = 2$.
	\item $m_\Lambda\bk{\Pi_1} = m_\Lambda\bk{\Pi_{-1}} = 1$.
	\end{itemize}
	
	\end{enumerate}
\end{proof}

\subsection{$E=F\times F\times F$}

\subsubsection{The Case of $\chi=\Id$}
Here we consider two different cases; one is $N\bk{w_{21342},w_{2134}^{-1}\cdot\chi_s}$ the other is $N\bk{w_{2342},w_{23421}^{-1}\cdot\chi_s}$.

\begin{Prop}
	The following holds:
	\begin{enumerate}
		\item For any $v\in N\bk{w_{2134},\Id,\lambda_{\frac{1}{2}}}\bk{I_{P_E}\bk{\Id,\frac{1}{2}}}$ it holds that $R\bk{w_{212},\Id}v=v$.
		\item For any $v\in N\bk{w_{23421},\Id,\lambda_{\frac{1}{2}}}\bk{I_{P_E}\bk{\Id,\frac{1}{2}}}$ it holds that $R\bk{w_{232},\Id}v = v$.
\end{enumerate}
\end{Prop}
The proof of this is similar to \Cref{Case of trivial character E non-split}.

\subsubsection{The Case of $\chi^2=\Id$ with $\chi\neq 1$}
This case is similar to \Cref{Subsubsection: Quadratic character non-split case}.
In this case,
\[
w_{23421}^{-1}\cdot\bk{\chi\otimes\lambda} \bk{t_1,t_2,t_3,t_4} = 
\chi\bk{\frac{t_2}{t_1^2}} \FNorm{t_1}^{-s_1-2s_2-s_3-s_4} \FNorm{t_2}^{s_1+s_2} \FNorm{t_3}^{s_4} \FNorm{t_4}^{s_3} .
\]

The Levi subgroup $M_{3,4}$ generated by $\alpha_3$ and $\alpha_4$ is isomorphic to $GL_1\times \bk{GL_2\times GL_2}^0$, where
\[
\bk{GL_2\times GL_2}^0\bk{F} = \set{\bk{g_1,g_2}\in GL_2\bk{F} \mvert \operatorname{det} g_1 = \operatorname{det} g_2} .
\]
Note that $M_{3,4}\cap B_E=GL_1\times\bk{\mathcal{B}\times \mathcal{B}}^0$.
We recall the the decomposition 
\[
\Ind_{\mathcal{B}\bk{F}}^{SL_2\bk{F}}\bk{\chi\boxtimes\chi} = \pi^{\bk{1}}\oplus\pi^{\bk{-1}}
\]
introduced \Cref{Subsubsection: Quadratic character non-split case}.

As a consequence of this,
\begin{align*}
&\Ind_{M_{3,4}\bk{F}\cap B_E\bk{F}}^{M_{3,4}\bk{F}}\bk{w_{23421}^{-1}\cdot\bk{\chi\otimes\lambda}} \res{GL_1\times SL_2\times SL_2} \\
& = \bk{\pi^{\bk{1}}\boxtimes \pi^{\bk{1}}} \oplus \bk{\pi^{\bk{-1}}\boxtimes \pi^{\bk{-1}}} \oplus \bk{\pi^{\bk{1}}\boxtimes \pi^{\bk{-1}}} \oplus \bk{\pi^{\bk{-1}}\boxtimes \pi^{\bk{1}}} .
\end{align*}

\begin{Remark}
	Note that the only two conjugacy classes of maximal compact subgroups of $\bk{GL_2\times GL_2}^0$ are of $\bk{GL_2\times GL_2}^0\bk{\mO}$ and $\bk{GL_2\times GL_2}^0\bk{\mO}^{\Delta d}$, where $\Delta$ is the diagonal embedding.
\end{Remark}

\begin{Remark}
	$N\bk{w_3}$ acts on $\pi^{\bk{\epsilon}}\boxtimes \pi^{\bk{\eta}}$ by $\epsilon\eta Id$, where $\epsilon,\eta\in\set{1,-1}$ and $N\bk{w_4}$ acts on it by $\eta Id$.
	As a result, $N\bk{w_{34}}$ acts on $\bk{\pi^{\bk{1}}\boxtimes \pi^{\bk{1}}} \oplus \bk{\pi^{\bk{-1}}\boxtimes \pi^{\bk{-1}}}$ by $Id$ and on $\bk{\pi^{\bk{1}}\boxtimes \pi^{\bk{-1}}} \oplus \bk{\pi^{\bk{-1}}\boxtimes \pi^{\bk{1}}}$ by $-Id$.
\end{Remark}

We let
\begin{align*}
& \Pi_1 = \Ind_{P_{3,4}\bk{F}}^{H_E\bk{F}} \bk{\pi^{\bk{1}}\boxtimes \pi^{\bk{1}}} \oplus \bk{\pi^{\bk{-1}}\boxtimes \pi^{\bk{-1}}} \\
& \Pi_{-1} = \Ind_{P_{3,4}\bk{F}}^{H_E\bk{F}} \bk{\pi^{\bk{1}}\boxtimes \pi^{\bk{-1}}} \oplus \bk{\pi^{\bk{-1}}\boxtimes \pi^{\bk{1}}}
\end{align*}
and then
\[
\Ind_{B_E\bk{F}}^{H_E\bk{F}} \coset{w_{23421}^{-1}\cdot \bk{\bk{\chi\circ \operatorname{det}_{M_E}}\otimes\lambda_{\frac{1}{2}}}} = \Pi_1 \oplus \Pi_{-1}
\]
is a direct sum of two irreducible representations.

We then have
\begin{itemize}
	\item
	$N\bk{w_2, w_{23421}^{-1}\cdot\chi, w_{23421}^{-1}\cdot\lambda}$ is holomorphic and bijective at $\lambda_{\frac{1}{2}}$.
	\item
	$N\bk{w_{34}, w_{234212}^{-1}\cdot\chi, w_{234212}^{-1}\cdot\lambda}$ acts on $\bk{\Pi_1}$ as $Id$ and on $\bk{\Pi_{-1}}$ as $-Id$ due to \Cref{Lemma: intertwining operator of simple reflections}.
	\item
	$N\bk{w_2, w_{23421234}^{-1}\cdot\chi, w_{23421234}^{-1}\cdot\lambda}$ is holomorphic and bijective at $\lambda_{\frac{1}{2}}$.
\end{itemize}

It follows from \Cref{Langlands identity for composition of intertwining operators} that $N\bk{w_{2342}, w_{23421}^{-1}\cdot\chi,w_{23421}^{-1}\cdot \lambda_{\frac{1}{2}}}$ acts on $\Pi_1$ as $Id$ and on $\Pi_{-1}$ as $-Id$.

\begin{Lem}
	\label{Lem: Local ramified character split}
	The following holds:
	\begin{enumerate}
	\item
	For $\chi$ unramified there exists $v\in N\bk{w_{23421},\chi_s}\bk{I_{P_E}\bk{\chi, s}}$ such that $R\bk{w_{2342},\chi}v = v$.
	\item
	If $F$ is non-Archimedean then there exists $v\in I_{P_E}\bk{\chi,\frac{1}{2}}$ such that $N\bk{w_{23421},\chi,w_{23421}\cdot\lambda_{\frac{1}{2}}}v\neq 0$ is not an eigenvector of $N\bk{w_{2342},\chi,\lambda_{\frac{1}{2}}}$.
	\end{enumerate}
\end{Lem}
The proof is similar to that of \Cref{Lem: Local ramified character non-split}.

{\LARGE \part{Applications}}
\label{Part 2} \mbox{}
\section{The Twisted Standard $\Lfun$-function of a Cuspidal Representation of $G_2$}
In this section we recall the main result of \cite{SegalRSGeneral}.

\subsection{The Group $G_2$}

Let $G$ be the simple, split group of type $G_2$ defined over $F$.
In particular, $G$ is adjoint and simply connected.
Let $B$ be a Borel subgroup of $G$ and $T$ a maximal torus in $B$.
Let $\alpha$ and $\beta$ be the short and long simple roots of $G$ with respect to $\bk{B,T}$.
The Dynkin diagram of $G$ is
\[
\xygraph{
!{<0cm,0cm>;<0cm,1cm>:<1cm,0cm>::}
!{(0.4,-1)}*{\alpha}="label1"
!{(0,-1)}*{\bigcirc}="1"
!{(0,-0.1)}="c"
!{(0.2,  0.1)}="c1"
!{(-0.2,0.1)}="c2"
!{(0.4,1)}*{\beta}="label2"
!{(0,1)}*{\bigcirc}="2"
"1"-@3"2" "c1"-"c" "c"-"c2" 
} .
\]
We have a short exact sequence
\[
1\to H_E\to \Aut\bk{H_E} \to S_E \to 1.
\]
Forming the semidirect product $H_E \rtimes S_E$ it holds that $G\cong \operatorname{Cent}_{H_E\rtimes S_E}\bk{S_E}$.
This gives a natural embedding
\[
G \hookrightarrow H_E .
\]
Moreover, $\bk{G,S_E}$ forms a dual reductive pair in $H_E\rtimes S_E$.
Under this embedding, it holds that $B$ can be chosen so that $B=G\cap B_E$.
The set of positive roots of $G$ is
\[
\Phi^{+} = \set{\alpha, \beta, \alpha+\beta, 2\alpha+\beta, 3\alpha+\beta, 3\alpha+2\beta} .
\]
For any root $\gamma$ we fix a one-parametric subgroup $x_\gamma:\Ga\to G$.
Also, let $h_\gamma:\Gm\to T$ be the coroot subgroup such that for any root $\epsilon$
\[
\epsilon\bk{h_\gamma\bk{t}} = t^{\gen{\epsilon,\check{\gamma}}} .
\]
The group $G$ contains an Heisenberg maximal parabolic subgroup $P=M\cdot U$.
The Levi subgroup $M$ is isomorphic to $GL_2$ and is generated by the simple root $\alpha$, while $U$ is a five-dimensional Heisenberg group.
It holds that $P=G\cap P_E$.

Finally, we let $\st:G\hookrightarrow GL_7$ be the standard 7-dimensional embedding.

\subsection{The Twisted Standard $\Lfun$-function and an Integral Representation}
The dual Langlands group $\ldual{G}$ of $G$ is isomorphic to $G_2\bk{\C}$.

Let $\pi=\placestimes \pi_\nu$ be an irreducible cuspidal representation of $G\bk{\A}$ and let $\chi=\placestimes\chi_\nu:F^\times\lmod \A^\times\to \C^\times$ be a Hecke character, both unramified outside of a finite subset $S\subset\Places$.
For $\nu\notin S$ we denote its Satake parameter by $t_{\pi_\nu}$.
We let
\[
\Lfun^S\bk{s,\pi,\chi,\st} = \prodl_{\nu\notin S} \frac{1}{\operatorname{det}\bk{I-\st\bk{t_{\pi_\nu}} \chi\bk{\unif_\nu} q_\nu^{-s} }} .
\]
This product converges for $\Real\bk{s}\gg0$ to an analytic function.

For factorizable data $\varphi=\placestimes \varphi_\nu\in\pi$ and $f_s=\placestimes f_{\nu}\in I_{P_E}\bk{\chi,s}$ we consider the following integral
\begin{equation}
\label{Eq:Zeta Integral}
\zint_E \bk{\chi, s, \varphi, f} = \intl_{G\bk{F}\lmod G\bk{\A}} \varphi\bk{g} \Eisen_E^\ast\bk{\chi,s,f,g} dg .
\end{equation}
It holds that
\begin{Thm}[\cite{SegalRSGeneral}]
\label{Thm: Integral representation}
Given a finite subset $S\subset\Places$ such that for any $\nu\notin S$ all data is unramified, then
\begin{equation}
\zint_E \bk{\chi, s, \varphi, f} = \Lfun^S\bk{s+\frac{1}{2},\pi,\chi,\st} d_S\bk{\chi, s,\Psi_E,\varphi_S,f_S} .
\end{equation}
Moreover, for any $s_0$ there exist vectors $\varphi_S$, $f_S$ such that $d_S\bk{\chi, s,\Psi_E,\varphi_S,f_S}$ is analytic in a neighborhood of $s_0$ and $d_S\bk{\chi,s_0,\Psi_E,\varphi_S,f_S}\neq 0$.

In particular, the family of twisted partial $\Lfun$-function $\Lfun^S\bk{s,\pi,\chi,\st}$ admits a meromorphic continuation to the whole complex plane.
\end{Thm}

For our applications, we need only the following corollary.
\begin{Cor}
\label{Cor: inequality on order of poles}
$\Lfun^S\bk{s+\frac{1}{2},\pi,\chi,\st}$ is a meromorphic function on $\C$ and for any $s_0\in \C$ it holds that
\begin{equation}
ord_{s=s_0}\bk{\Lfun^S\bk{s,\pi,\chi,\st}} \leq ord_{s=s_0}\bk{\Eisen_E\bk{\chi, f_s, s, g}} .
\end{equation}
\end{Cor}

\section{A Conjecture of Ginzburg and Hundley}
In \cite{MR3359720}, D. Ginzburg and J. Hundley have constructed a doubling integral representing $\Lfun^S\bk{s,\pi,\chi,\st}$.
We recall the construction.

We first recall the computing pair $G_2\times G_2\subseteq E_8$.
Given a cuspidal representation $\pi$ of $G_2$, $\varphi\in\pi$ and $\widetilde{\varphi}\in\widetilde{\pi}$ we consider the integral
\begin{equation}
\label{eq:GinzburgHundleyIntegral}
\integral{G_2\times G_2} \varphi\bk{g_1} \widetilde{\varphi}\bk{g} 
\Eisen_{E_8}^{\Psi_1}\bk{\bk{g_1,g_2},f_{s,\chi}}.
 d\bk{g_1, g_2},
\end{equation}
where $\Eisen_{E_8}^{\Psi_1}$ is a certain Fourier coefficient of a degenerate Eisenstein series for $E_8$ associated with the maximal parabolic subgroup whose Levi factor is of type $A_7$.
In \cite{MR3359720}, Ginzburg and Hundley have shown that the integral in \Cref{eq:GinzburgHundleyIntegral} represents $\Lfun^S\bk{s,\pi,\chi,\st}$.

Considering the normalizing factor of this integral they conjectured the following:
\begin{Thm}
\label{Conj:GinzburgHundley}
The twisted partial standard $\Lfun$-function $\Lfun^S\bk{s,\pi,\chi,\st}$ can have at most a double pole at $\Real\bk{s}>0$.
\end{Thm}

\begin{proof}
Assuming that $\pi$ supports a $\Psi_E$-Fourier coefficient for non-split $E$, the claim follows from \Cref{Cor: inequality on order of poles} and \Cref{Poles of Eisenstein Series}.
Now assume that $\mathcal{WF}\bk{\pi}=\set{F\times F\times F}$.

From the same considerations it follows that the claim holds for $\chi\neq\Id$ and for $\chi=\Id$ and any $s_0\neq 1$ with $\Real\bk{s_0}>0$.
It also follows that for $\chi=\Id$ it holds that $ord_{s=1}\bk{\Lfun^S\bk{s,\pi,\chi,\st}}\leq 3$.

Assume that $ord_{s=1}\bk{\Lfun^S\bk{s,\pi,\chi,\st}}=3$.
From \cite[Theorem 3]{MR1203229} it follows that $\pi$ is not nearly-equivalent to a generic representation and hence, from \cite[Theorem 16.6]{MR2262172}, it is nearly-equivalent to $\Theta_{H_{F\times F\times F}}\bk{\mathbf{1}_{S_{F\times F\times F}}}$.
In this case,
\[
\Lfun^S\bk{s,\pi,\chi,\st} = 
\zfun^S_F\bk{s-1}^2 \zfun^S_F\bk{s+1}^2 \zfun^S_F\bk{s}^3 .
\]
The right-hand side has a pole of order $3-2\bk{\FNorm{S}-1}$.
Choosing $S\subseteq\Places$ with $\FNorm{S}>1$ bring us to a contradiction with the assumption that $ord_{s=1}\bk{\Lfun^S\bk{s,\pi,\chi,\st}}=3$.
\end{proof}

\section{CAP Representations With Respect to the Borel Subgroup}
\label{Sec: CAP representation}
We recall the definition of a \textbf{CAP} representation.
\begin{Def}
Let $Q=L\cdot V \subset G$ be a parabolic subgroup, $\sigma$ be a cuspidal unitary representation of the Levi part $L$ and $\chi$ be a character of $L$.
A cuspidal representation $\pi$ of $G\bk{\A}$ is called \textbf{CAP} (cuspidal attached to parabolic) with respect to $Q$, $\sigma$ and $\chi$ if $\pi$ is nearly equivalent to a subquotient of $\Ind_{Q\bk{\A}}^{G\bk{\A}}\sigma\otimes\chi$.
\end{Def}

\textbf{CAP} representations for $G_2$ were constructed in \cite{MR1918673} for the Borel subgroup,
 in \cite{MR1020830} for the Heisenberg parabolic subgroup $P$ and in \cite{MR2506316} for the non-Heisenberg maximal parabolic subgroup.
Using \Cref{Cor: inequality on order of poles} and \Cref{Poles of Eisenstein Series} we plan to prove that \cite{MR1918673} exhaust the list of \textbf{CAP} representations with respect to the Borel subgroup.

\begin{Thm}
\label{Thm: CAP representations}
Let $\pi$ be a cuspidal representation of $G\bk{\A}$ supporting a Fourier coefficient along $U$ corresponding to an \'etale cubic extension $E$ of $F$ which is not a non-Galois field extension.
The following are equivalent:
\begin{enumerate}
\item $\pi$ is a \textbf{CAP} representation with respect to $B$.
\item The partial $\Lfun$-function $\Lfun^S\bk{s,\pi,\chi_E,\st}$ has a pole of order $n_E$ at $s=2$ 
\item $\Theta_{S_E}\bk{\pi}\neq 0$. In particular $\pi$ is nearly equivalent to 
 $\Theta_{S_E}\bk{\Id}$, where $\Id$ here is the automorphic trivial representation of $S_E\bk{\A}$.
\end{enumerate}
In particular, for $\pi$ that satisfy this conditions we have $\mathcal{WF}\bk{\pi}=\set{E}$.
\end{Thm}

\begin{proof} 
The fact that \textit{3} implies \textit{1} and \textit{2} was proven in \cite{MR1918673}.
The fact that \textit{2} implies \textit{3} is proven in \cite{SegalRSGeneral}.
It is left to prove that \textit{1} implies \textit{2}.

Let $\pi$ be a \textbf{CAP} representation with respect to $B$ that supports the Fourier coefficient corresponding to an \'etale cubic algebra $E$ over $F$. 
We will prove that \textit{2} holds by proving that $\pi$ is nearly equivalent to $\Theta_{H_E}\bk{\mathbf{1}_{S_E}}$ where $\mathbf{1}_{S_E}$ is the trivial representation of $S_E\bk{\A}$.
\begin{Remark}
Note that all irreducible automorphic representations of $S_E\bk{\A}$ are nearly equivalent to $\mathbf{1}_{S_E}$.
\end{Remark}

By the assumption, there exist an automorphic character $\mu$ such that $\pi$ is nearly equivalent to a subquotient of $\Ind_{B\bk{\A}}^{G\bk{\A}}\mu$, where the induction here is unitary.
Let
\begin{equation}
\mu\bk{h_{2\alpha+\beta}\bk{a}h_{3\alpha+2\beta}\bk{b}} = \mu_1\bk{a} \mu_2\bk{b} .
\end{equation}

We denote by $\mu_i\bk{x} = \eta_i\bk{x} \FNorm{x}^{z_i}$, where $\eta_i$ are unitary characters of finite order and $z_i\in\R$.
By choosing a Weyl chamber we may assume that
\begin{eqnarray}
0\leq z_2\leq z_1\leq 2z_2 .
\end{eqnarray}

According to \cite{MR1918673} we need to show:
\begin{itemize}
\item If $E=F\times F\times F$ then $\mu_1\bk{t}=\mu_2\bk{t}=\FNorm{t}$ for any $t\in \A^\times$.
\item If $E=F\times K$ then $\mu_1\bk{t}=\FNorm{t}$ and $\mu_2\bk{t}=\chi_K\bk{t}\FNorm{t}$ for any $t\in \A^\times$, or vice versa.
\item If $E\rmod F$ is a cubic Galois extension, then $\mu_1\bk{t}=\mu_2\bk{t}=\chi_E\bk{t}\FNorm{t}$ for any $t\in \A^\times$.
\end{itemize}

It holds that
\[
\Lfun^S\bk{s,\pi,\chi,\st} = 
\Lfun^S_F\bk{\mu_1\chi,s} \Lfun^S_F\bk{\mu_1^{-1}\chi,s} \Lfun^S_F\bk{\mu_2\chi,s} \Lfun^S_F\bk{\mu_2^{-1}\chi,s} \Lfun^S_F\bk{\frac{\mu_1}{\mu_2}\chi,s} \Lfun^S_F\bk{\frac{\mu_2}{\mu_1}\chi,s} \Lfun^S_F\bk{\chi,s} .
\]

For $\chi\bk{t}=\mu_1\bk{t} \FNorm{t}^{-1}$, $\Lfun^S\bk{s,\pi,\mu_1\FNorm{\cdot}^{-1},\st}$ admits a pole at $s=2$ and hence $\Eisen_E\bk{\mu_1\FNorm{\cdot}^{-1}, f_s, s, g}$ admits a pole at $s=\frac{3}{2}$.
Similarly, $\Eisen_E\bk{\mu_2\FNorm{\cdot}^{-1}, f_s, s, g}$ also admits a pole at $s=\frac{3}{2}$.

We continue by considering different kinds of $E$.

\begin{itemize}
\item \underline{$E=F\times F\times F$:}
Since $\Eisen_E\bk{\mu_1\FNorm{\cdot}^{-1}, f_s, s, g}$ and $\Eisen_E\bk{\mu_2\FNorm{\cdot}^{-1}, f_s, s, g}$ admits a pole at $s=\frac{3}{2}$, it holds that
\[
\bk{z_1,\eta_1}, \bk{z_2,\eta_2}\in \set{\bk{0,\eta}\mvert \eta^2\equiv\Id}\cup \set{\bk{1,\Id}} .
\]

We assume that $z_1=0$ and hence also $z_2=0$.
In this case $\eta_1$ and $\eta_2$ are quadratic characters.
If $\eta_1=\Id$, then
\[
\Lfun^S\bk{s,\pi,\chi,\st} = 
\Lfun^S_F\bk{\chi,s}^3 \Lfun^S_F\bk{\mu_2\chi,s}^4 .
\]
If $\eta_2=\Id$ then $\Lfun^S\bk{s,\pi,\Id,\st}$ admits a pole of order $7$ at $s=1$, while $\Eisen_E\bk{\Id, f_s, s, g}$ admits a pole of order at most $1$ at $s=\frac{1}{2}$ which brings us to a contradiction.

Assume that $\eta_2\neq\Id$ then $\Lfun^S\bk{s,\pi,\eta_2,\st}$ admits a pole of order $4$ at $s=1$ while $\Eisen_E\bk{\eta_2, f_s, s, g}$ admits a pole of order at most $1$ at $s=\frac{1}{2}$ which again brings us to a contradiction.

We now assume that $\eta_1,\eta_2\not\equiv \Id$ are quadratic characters.
In this case
\[
\Lfun^S\bk{s,\pi,\chi,\st} = 
\Lfun^S_F\bk{\eta_1\chi,s}^2 \Lfun^S_F\bk{\eta_2\chi,s}^2 \Lfun^S_F\bk{\eta_1\eta_2\chi,s}^2 \Lfun^S_F\bk{\chi,s} .
\]
$\Lfun^S\bk{s,\pi,\eta_1,\st}$ admits a pole of order at least $2$ at $s=1$, while $\Eisen_E\bk{\eta_1, f_s, s, g}$ admits a pole of order at most $1$ which again brings us to a contradiction.

In conclusion, $z_1=1$ and hence also $z_2\geq \frac{1}{2}$. In particular, $z_2=1$. We conclude that $\eta_1\equiv\eta_2\equiv\Id$ which proves the assertion.

\item \underline{$E=F\times K$, where $K\rmod F$ is a quadratic extension:}
Since $\Eisen_E\bk{\mu_1\FNorm{\cdot}^{-1}, f_s, s, g}$ and $\Eisen_E\bk{\mu_2\FNorm{\cdot}^{-1}, f_s, s, g}$ admits a pole at $s=2$, it holds that
\[
\bk{z_1,\eta_1}, \bk{z_2,\eta_2}\in \set{\bk{0,\Id}, \bk{0,\chi_K}, \bk{1,\Id}, \bk{1,\chi_K}}.
\]
The proof that $z_1,z_2\neq 0$ is similar to the split case.
It then holds that $z_1=z_2=1$.
We need to prove that $\eta_1\equiv\eta_2\equiv\Id$ or $\eta_1\equiv\eta_2\equiv\chi_K$ cannot happen.

Assume that $\eta_1\equiv\eta_2\equiv\Id$, in this case
\[
\Lfun^S\bk{s,\pi,\chi,\st} = 
\Lfun^S_F\bk{\chi,s}^3 \Lfun^S_F\bk{\chi,s-1}^2 \Lfun^S_F\bk{\chi,s+1}^2.
\]
$\Lfun^S\bk{s,\pi,\Id,\st}$ would have a pole of order at least $3$ at $s=1$ while $\Eisen_E\bk{\Id, f_s, s, g}$ admits a pole of order at most $1$ at $s=\frac{1}{2}$, which brings us to a contradiction.

Assume that $\eta_1\equiv\eta_2\equiv\chi_K$, in this case
\[
\Lfun^S\bk{s,\pi,\chi,\st} = 
\Lfun^S_F\bk{\chi,s}^3 \Lfun^S_F\bk{\chi_K\chi,s-1}^2 \Lfun^S_F\bk{\chi_K\chi,s+1}^2.
\]
$\Lfun^S\bk{s,\pi,\Id,\st}$ would have a pole of order at least $3$ at $s=1$ while $\Eisen_E\bk{\Id, f_s, s, g}$ admits a pole of order at most $1$ at $s=\frac{1}{2}$, which brings us to a contradiction.

In conclusion, $\mu_1=\FNorm{\cdot}$ and $\mu_2=\FNorm{\cdot}\chi_K$, or vice versa, which proves the assertion.

\item \underline{$E\rmod F$ is a cubic Galois extension:}
Since $\Eisen_E\bk{\mu_1\FNorm{\cdot}^{-1}, f_s, s, g}$ and $\Eisen_E\bk{\mu_2\FNorm{\cdot}^{-1}, f_s, s, g}$ admits a pole at $s=2$, it holds that
\[
\bk{z_1,\eta_1}, \bk{z_2,\eta_2}\in \set{\bk{0,\eta}\mvert \eta^2\equiv\Id}\cup \set{\bk{1,\chi_E}} .
\]
The proof that $\bk{z_1,\eta_1},\bk{z_2,\eta_2}\neq \bk{0,\eta}$ for $\eta$ a quadratic character is similar to the split case.
Hence, $\mu_1\equiv\mu_2\equiv \chi_E \FNorm{\cdot}$ which proves the assertion.
\end{itemize}

\end{proof}

\addcontentsline{toc}{part}{References}
\TOCstop
\bibliographystyle{alpha}
\bibliography{bib}

\begin{thebibliography}{HMS98}

\bibitem[BT84]{MR756316}
F.~Bruhat and J.~Tits.
\newblock Groupes r\'eductifs sur un corps local. {II}. {S}ch\'emas en groupes.
  {E}xistence d'une donn\'ee radicielle valu\'ee.
\newblock {\em Inst. Hautes \'Etudes Sci. Publ. Math.}, (60):197--376, 1984.

\bibitem[Cas74]{Casselman}
W.~Casselman.
\newblock Introduction to the theory of admissible representations of p-adic
  reductive groups.
\newblock \url{http://www.math.ubc.ca/~cass/research/pdf/p-adic-book.pdf},
  1974.

\bibitem[Gan05]{MR2181091}
Wee~Teck Gan.
\newblock Multiplicity formula for cubic unipotent {A}rthur packets.
\newblock {\em Duke Math. J.}, 130(2):297--320, 2005.

\bibitem[GG06]{MR2262172}
Wee~Teck Gan and Nadya Gurevich.
\newblock Nontempered {A}-packets of {$G_2$}: liftings from {$\widetilde{\rm
  SL}_2$}.
\newblock {\em Amer. J. Math.}, 128(5):1105--1185, 2006.

\bibitem[GG09]{MR2506316}
Wee~Teck Gan and Nadya Gurevich.
\newblock C{AP} representations of {$G_2$} and the spin {$L$}-function of
  {${\rm PGSp}_6$}.
\newblock {\em Israel J. Math.}, 170:1--52, 2009.

\bibitem[GGJ02]{MR1918673}
Wee~Teck Gan, Nadya Gurevich, and Dihua Jiang.
\newblock Cubic unipotent {A}rthur parameters and multiplicities of square
  integrable automorphic forms.
\newblock {\em Invent. Math.}, 149(2):225--265, 2002.

\bibitem[GH06]{MR2268487}
Wee~Teck Gan and Joseph Hundley.
\newblock The spin {$L$}-function of quasi-split {$D_4$}.
\newblock {\em IMRP Int. Math. Res. Pap.}, pages Art. ID 68213, 74, 2006.

\bibitem[GH15]{MR3359720}
David Ginzburg and Joseph Hundley.
\newblock A doubling integral for {$G_2$}.
\newblock {\em Israel J. Math.}, 207(2):835--879, 2015.

\bibitem[Gin93]{MR1203229}
David Ginzburg.
\newblock On the standard {$L$}-function for {$G_2$}.
\newblock {\em Duke Math. J.}, 69(2):315--333, 1993.

\bibitem[GRS97]{MR1469105}
David Ginzburg, Stephen Rallis, and David Soudry.
\newblock On the automorphic theta representation for simply laced groups.
\newblock {\em Israel J. Math.}, 100:61--116, 1997.

\bibitem[GSa]{Gan-Savin-D4}
Wee~Teck Gan and Gordan Savin.
\newblock Untitled.
\newblock {\em Unpublished}.

\bibitem[GSb]{RallisSchiffmannPaper}
Nadya Gurevich and Avner Segal.
\newblock Poles of the standard {$\mathcal{L}$}-function of {$G_2$} and the
  rallis-schiffmann lift.
\newblock {\em Preprint}.

\bibitem[GS15]{MR3284482}
Nadya Gurevich and Avner Segal.
\newblock The {R}ankin-{S}elberg integral with a non-unique model for the
  standard {$\mathcal{L}$}-function of {$G_2$}.
\newblock {\em J. Inst. Math. Jussieu}, 14(1):149--184, 2015.

\bibitem[HMS98]{MR1637485}
Jing-Song Huang, Kay Magaard, and Gordan Savin.
\newblock Unipotent representations of {$G_2$} arising from the minimal
  representation of {$D_4^E$}.
\newblock {\em J. Reine Angew. Math.}, 500:65--81, 1998.

\bibitem[Ike92]{MR1174424}
Tamotsu Ikeda.
\newblock On the location of poles of the triple {$L$}-functions.
\newblock {\em Compositio Math.}, 83(2):187--237, 1992.

\bibitem[Kna01]{MR1880691}
Anthony~W. Knapp.
\newblock {\em Representation theory of semisimple groups}.
\newblock Princeton Landmarks in Mathematics. Princeton University Press,
  Princeton, NJ, 2001.
\newblock An overview based on examples, Reprint of the 1986 original.

\bibitem[KS88]{MR944102}
C.~David Keys and Freydoon Shahidi.
\newblock Artin {$L$}-functions and normalization of intertwining operators.
\newblock {\em Ann. Sci. \'Ecole Norm. Sup. (4)}, 21(1):67--89, 1988.

\bibitem[Kud03]{MR1990377}
Stephen~S. Kudla.
\newblock Tate's thesis.
\newblock In {\em An introduction to the {L}anglands program ({J}erusalem,
  2001)}, pages 109--131. Birkh\"auser Boston, Boston, MA, 2003.

\bibitem[Lan76]{MR0579181}
Robert~P. Langlands.
\newblock {\em On the functional equations satisfied by {E}isenstein series}.
\newblock Lecture Notes in Mathematics, Vol. 544. Springer-Verlag, Berlin-New
  York, 1976.

\bibitem[Lao]{LaoResidualSpectrum}
Jing~Feng Lao.
\newblock Residual spectrum of quasi-split {$Spin(8)$} defined by a cubic
  extension.
\newblock {\em Preprint}.

\bibitem[MW95]{MR1361168}
C.~M{\oe}glin and J.-L. Waldspurger.
\newblock {\em Spectral decomposition and {E}isenstein series}, volume 113 of
  {\em Cambridge Tracts in Mathematics}.
\newblock Cambridge University Press, Cambridge, 1995.
\newblock Une paraphrase de l'{\'E}criture [A paraphrase of Scripture].

\bibitem[PS79]{MR546599}
I.~I. Piatetski-Shapiro.
\newblock Multiplicity one theorems.
\newblock In {\em Automorphic forms, representations and {$L$}-functions
  ({P}roc. {S}ympos. {P}ure {M}ath., {O}regon {S}tate {U}niv., {C}orvallis,
  {O}re., 1977), {P}art 1}, Proc. Sympos. Pure Math., XXXIII, pages 209--212.
  Amer. Math. Soc., Providence, R.I., 1979.

\bibitem[RS89]{MR1020830}
S.~Rallis and G.~Schiffmann.
\newblock Theta correspondence associated to {$G_2$}.
\newblock {\em Amer. J. Math.}, 111(5):801--849, 1989.

\bibitem[Seg]{SegalRSGeneral}
A.~Segal.
\newblock A family of new-way integrals for the standard $\mathcal{L}$-function
  of cuspidal representations of the exceptional group of type {$G_2$}.
\newblock {\em To apear in International Mathematics Research Notices}.

\bibitem[Sha80]{MR563369}
Freydoon Shahidi.
\newblock Whittaker models for real groups.
\newblock {\em Duke Math. J.}, 47(1):99--125, 1980.

\bibitem[Spr79]{MR546587}
T.~A. Springer.
\newblock Reductive groups.
\newblock In {\em Automorphic forms, representations and {$L$}-functions
  ({P}roc. {S}ympos. {P}ure {M}ath., {O}regon {S}tate {U}niv., {C}orvallis,
  {O}re., 1977), {P}art 1}, Proc. Sympos. Pure Math., XXXIII, pages 3--27.
  Amer. Math. Soc., Providence, R.I., 1979.

\bibitem[Tad12]{MR2908042}
Marko Tadi{\'c}.
\newblock Reducibility and discrete series in the case of classical {$p$}-adic
  groups; an approach based on examples.
\newblock In {\em Geometry and analysis of automorphic forms of several
  variables}, volume~7 of {\em Ser. Number Theory Appl.}, pages 254--333. World
  Sci. Publ., Hackensack, NJ, 2012.

\bibitem[Was97]{MR1421575}
Lawrence~C. Washington.
\newblock {\em Introduction to cyclotomic fields}, volume~83 of {\em Graduate
  Texts in Mathematics}.
\newblock Springer-Verlag, New York, second edition, 1997.

\bibitem[Win78]{MR517138}
Norman Winarsky.
\newblock Reducibility of principal series representations of {$p$}-adic
  {C}hevalley groups.
\newblock {\em Amer. J. Math.}, 100(5):941--956, 1978.

\bibitem[{\v{Z}}am97]{Zampera1997}
Sini{\v{s}}a {\v{Z}}ampera.
\newblock The residual spectrum of the group of type {$G_2$}.
\newblock {\em J. Math. Pures Appl. (9)}, 76(9):805--835, 1997.

\end{thebibliography}
\TOCstart

\newpage

\appendix
\begin{landscape}
	{\LARGE \part*{Appendices}}
	\section{Tables of Intertwining Operators}
	\label{Sec: Tables}
	
	In this section we list useful tables containing information about the local intertwining operators, poles of global Gindikin-Karpelevich factors and the exponents of $w^{-1}\cdot\chi_s\bk{t}$ in the various cases.
	
	\subsection{Cubic Extension Case}
	Assume $E$ is a Cubic Field Extension of $F$.
	In this case we denote 
	\[
	t=h_{\alpha_1\alpha_3^\sigma\alpha_4^{\sigma^2}}\bk{t_1}h_{\alpha_2}\bk{t_2} = h_{\alpha_1}\bk{t_1} h_{\alpha_2}\bk{t_2} h_{\alpha_3}\bk{t_1^\sigma} h_{\alpha_4}\bk{t_1^{\sigma^2}} ,
	\]
	where $t_1\in E^\times, t_2\in F^\times$.
	
	In the following table we list $w^{-1}\cdot\chi_s\bk{t}$ for the various $w\in W\bk{P_E,H_E}$.
	\begin{longtable}{|c|c|}
		\hline
		$w \in W\bk{P_E,H_E}$ & $w^{-1}\cdot\chi_s\bk{t}$ \\ \hline
		\endhead
		$\coset{}$ & $\chi\bk{t_2}\frac{\FNorm{t_2}_F^{s+\frac{3}{2}}}{\FNorm{t_1}_E}$ \\ \hline
		$\coset{2}$ & $\chi\bk{\frac{\Nm_{E\rmod F}\bk{t_1}}{t_2}}\frac{\FNorm{t_1}_E^{s+\frac{1}{2}}}{\FNorm{t_2}_F^{s+\frac{3}{2}}}$ \\ \hline
		$\coset{2,1}$ & $\chi\bk{\frac{t_2^2}{\Nm_{E\rmod F}\bk{t_1}}}\frac{\FNorm{t_2}_F^{2s}}{\FNorm{t_1}_E^{s+\frac{1}{2}}}$ \\ \hline
		$\coset{2,1,2}$ & $\chi\bk{\frac{\Nm_{E\rmod F}\bk{t_1}}{t_2^2}}\frac{\FNorm{t_1}_E^{s-\frac{1}{2}}}{\FNorm{t_2}_F^{2s}}$ \\ \hline
		$\coset{2,1,2,1}$ & $\chi\bk{\frac{t_2}{\Nm_{E\rmod F}\bk{t_1}}}\frac{\FNorm{t_2}_F^{s-\frac{3}{2}}}{\FNorm{t_1}_E^{s-\frac{1}{2}}}$ \\ \hline
		$\coset{2,1,2,1,2}$ & $\chi\bk{\frac{1}{t_2}}\frac{1}{\FNorm{t_1}_E \FNorm{t_2}_F^{s-\frac{3}{2}}}$ \\ \hline
		\caption{$w^{-1}\cdot\chi_s$ for $w\in W\bk{P_E,H_E}$, $E$ is a field}
		\label{Table: Action of W on characters, Cubic}
	\end{longtable}

	In the following table we list the Gindikin-Karpelevich factor $J\bk{w,\chi,\lambda}$ and the poles of the global Gindikin-Karpelevich factor $J\bk{w,\chi_s}$ for $\Real\bk{s}>0$.
	\begin{longtable}{|c|c|c|c|c|c|c|c|c|c|c|}
		\hline
		\multicolumn{2}{|c|}{} & \multicolumn{3}{c|}{$s=\frac{1}{2}$} & \multicolumn{2}{c|}{$s=\frac{3}{2}$} & $s=\frac{5}{2}$ \\ \hline
		$w\in W\bk{P_E,H_E}$ & $J\bk{w,\chi_s}$ & $1$ & $\chi_E$ & $\chi^2=\Id$ & $1$ & $\chi_E$ & $1$ \\ \hline
		\endhead
		$\coset{}$ & $1$ & 0 & 0 & 0 & 0 & 0 & 0 \\ \hline
		$\coset{2}$ & $\frac{\Lfun_F\bk{s+\frac{3}{2},\chi}}{\Lfun_F\bk{s+\frac{5}{2},\chi}}$ & 0 & 0 & 0 & 0 & 0 & 0 \\ \hline
		$\coset{2,1}$ & $\frac{\Lfun_F\bk{s+\frac{3}{2},\chi}\Lfun_E\bk{s+\frac{1}{2},\chi\circ\Nm}}{\Lfun_F\bk{s+\frac{5}{2},\chi}\Lfun_E\bk{s+\frac{3}{2},\chi\circ\Nm}}$ & 1 & 1 & 0 & 0 & 0 & 0 \\ \hline
		$\coset{2,1,2}$ & 
		$\frac{\Lfun_F\bk{s+\frac{3}{2},\chi} \Lfun_E\bk{s+\frac{1}{2},\chi\circ\Nm}\Lfun_F\bk{2s,\chi^2}}{\Lfun_F\bk{s+\frac{5}{2},\chi}\Lfun_E\bk{s+\frac{3}{2},\chi\circ \Nm}\Lfun_F\bk{2s+1,\chi^2}}$ & 2 & 1 & 1 & 0 & 0 & 0 \\ \hline
		$\coset{2,1,2,1}$ & $\frac{\Lfun_F\bk{s+\frac{3}{2},\chi}\Lfun_E\bk{s-\frac{1}{2},\chi\circ\Nm}\Lfun_F\bk{2s,\chi^2}}{\Lfun_F\bk{s+\frac{5}{2},\chi}\Lfun_E\bk{s+\frac{3}{2},\chi\circ\Nm}\Lfun_F\bk{2s+1,\chi^2}}$ & 2 & 1 & 1 & 1 & 1 & 0 \\ \hline
		$\coset{2,1,2,1,2}$ & $\frac{\Lfun_F\bk{s-\frac{3}{2},\chi}\Lfun_F\bk{s+\frac{3}{2},\chi}\Lfun_E\bk{s-\frac{1}{2},\chi\circ\Nm}\Lfun_F\bk{2s,\chi^2}}{\Lfun_F\bk{s-\frac{1}{2},\chi}\Lfun_F\bk{s+\frac{5}{2},\chi}\Lfun_E\bk{s+\frac{3}{2},\chi\circ\Nm}\Lfun_F\bk{2s+1,\chi^2}}$ & 1 & 1 & 1 & 1 & 1 & 1 \\ \hline
		\caption{Poles of $J\bk{w,\chi_s}$ for $w\in W\bk{P_E,H_E}$, $E$ is a field}
		\label{Constant Term along the Borel, Cubic}
	\end{longtable}

	In the following table we list the Gindikin-Karpelevich factor $J\bk{w,\chi,\lambda}$.
	Here $\lambda\bk{t}=\FNorm{t_1}_{E}^{s_1}\FNorm{t_2}_F^{s_2}$.
	\begin{longtable}{|c|c|}
		\hline
		$w\in W\bk{P_E,H_E}$ & $J\bk{w,\chi,\lambda}$ \\ \hline
		\endhead
		$\coset{}$ & 1 \\ \hline
		$\coset{2}$ & $\frac{\Lfun_F\bk{s_2,\chi}}{\Lfun_F\bk{s_2+1,\chi}}$ \\ \hline
		$\coset{2,1}$ & $\frac{\Lfun_F\bk{s_2,\chi}\Lfun_E\bk{s_1+s_2,\chi\circ\Nm}}{\Lfun_F\bk{s_2+1,\chi}\Lfun_E\bk{s_1+s_2+1,\chi\circ\Nm}}$ \\ \hline
		$\coset{2,1,2}$ & $\frac{\Lfun_F\bk{s_2,\chi}\Lfun_E\bk{s_1+s_2,\chi\circ\Nm}\Lfun_F\bk{3s_1+2s_2,\chi^2}}{\Lfun_F\bk{s_2+1,\chi}\Lfun_E\bk{s_1+s_2+1,\chi\circ\Nm}\Lfun_F\bk{3s_1+2s_2+1,\chi^2}}$
		\\ \hline
		$\coset{2,1,2,1}$ & $\frac{\Lfun_F\bk{s_2,\chi}\Lfun_E\bk{s_1+s_2,\chi\circ\Nm}\Lfun_F\bk{3s_1+2s_2,\chi^2}\Lfun_E\bk{2s_1+s_2,\chi\circ\Nm}}{\Lfun_F\bk{s_2+1,\chi}\Lfun_E\bk{s_1+s_2+1,\chi\circ\Nm}\Lfun_F\bk{3s_1+2s_2+1,\chi^2}\Lfun_E\bk{2s_1+s_2+1,\chi\circ\Nm}}$ \\ \hline
		$\coset{2,1,2,1,2}$ & $\frac{\Lfun_F\bk{s_2,\chi}\Lfun_E\bk{s_1+s_2,\chi\circ\Nm}\Lfun_F\bk{3s_1+2s_2,\chi^2}\Lfun_E\bk{2s_1+s_2,\chi\circ\Nm}\Lfun_F\bk{3s_1+s_2,\chi}}{\Lfun_F\bk{s_2+1,\chi}\Lfun_E\bk{s_1+s_2+1,\chi\circ\Nm}\Lfun_F\bk{3s_1+2s_2+1,\chi^2}\Lfun_E\bk{2s_1+s_2+1,\chi\circ\Nm}\Lfun_F\bk{3s_1+s_2+1,\chi}}$ \\ \hline
		\caption{$J\bk{w,\chi,\lambda}$ for $w\in W\bk{P_E,H_E}$, $E$ is a field}
		\label{Constant Term along the Borel, Cubic, General Character}
	\end{longtable}

	In the following table we list the exponents $\Real\bk{w^{-1}\cdot\chi_s}$ for all $w\in W\bk{P_E,H_E}$ at the points $s=\frac{1}{2}$ and $\frac{3}{2}$.
	\begin{longtable}{|c|c|c|c|c|c|c|c|c|c|c|}
		\hline
		$w\in W\bk{P_E,H_E}$ & $s=\frac{1}{2}$ & $s=\frac{3}{2}$ \\ \hline
		\endhead
		$[]$ & $\coset{0,1}$ & $\Nm_{E\rmod F}\coset{1,0}+3\coset{0,1}$ \\ \hline
		$[2]$ & $-\coset{0,1}$ & $\Nm_{E\rmod F}\coset{1,0}$ \\ \hline
		$\coset{2,1}$ & $-\Nm_{E\rmod F}\coset{1,0}-\coset{0,1}$ & $-\Nm_{E\rmod F}\coset{1,0}$ \\ \hline
		$\coset{2,1,2}$ & 
		$-\Nm_{E\rmod F}\coset{1,0}-2\coset{0,1}$ & $-\Nm_{E\rmod F}\coset{1,0}-3\coset{0,1}$ \\ \hline
		$\coset{2,1,2,1}$ & $-\Nm_{E\rmod F}\coset{1,0}-2\coset{0,1}$ & $-2\Nm_{E\rmod F}\coset{1,0}-3\coset{0,1}$ \\ \hline
		$\coset{2,1,2,1,2}$ & $-\Nm_{E\rmod F}\coset{1,0}-\coset{0,1}$ & $-2\Nm_{E\rmod F}\coset{1,0}-3\coset{0,1}$ \\ \hline
		\caption{The exponents $\Real\bk{w^{-1}\cdot\chi_s}$ for $w\in W\bk{P_E,H_E}$, $E$ is a field}
		\label{Table: Exponents, Cubic}
	\end{longtable}

	\subsection{Quadratic Extension Case}
	Assume $E=F\times K$, where $K$ is a field.
	For this case we denote 
	\[
	t=h_{\alpha_1}\bk{t_1}h_{\alpha_2}\bk{t_2}h_{\alpha_3\alpha_4^\sigma}\bk{t_3} = h_{\alpha_1}\bk{t_1}h_{\alpha_2}\bk{t_2}h_{\alpha_3}\bk{t_3}h_{\alpha_3}\bk{t_3^\sigma} ,
	\]
	where $t_1,t_2\in F^\times, t_3\in K^\times$.
	
	In the following table we list $w^{-1}\cdot\chi_s\bk{t}$ for the various $w\in W\bk{P_E,H_E}$.
	\begin{longtable}{|c|c|}
		\hline
		$w \in W\bk{P_E,H_E}$ & $w^{-1}\cdot\chi_s\bk{t}$ \\ \hline
		\endhead
		$\coset{}$ & $\chi\bk{t_2}\frac{\FNorm{t_2}_F^{s+\frac{3}{2}}}{\FNorm{t_1}_F \FNorm{t_3}_K}$\\ \hline
		$\coset{2}$ & $\chi\bk{\frac{t_1\Nm_{K\rmod F}\bk{t_3}}{t_2}}\frac{\FNorm{t_1}_F^{s+\frac{1}{2}} \FNorm{t_3}_K^{s+\frac{1}{2}}}{\FNorm{t_2}_F^{s+\frac{3}{2}}}$ \\ \hline
		$\coset{2,1}$ & $\chi\bk{\frac{\Nm_{K\rmod F}\bk{t_3}}{t_1}} \frac{\FNorm{t_3}_K^{s+\frac{1}{2}}}{\FNorm{t_1}_F^{s+\frac{1}{2}} \FNorm{t_2}_F}$ \\ \hline
		$[2,3]$ & $\chi\bk{\frac{t_1t_2}{\Nm_{K\rmod F}\bk{t_3}}}\frac{\FNorm{t_1}_F^{s+\frac{1}{2}} \FNorm{t_2}_F^{s-\frac{1}{2}}}{\FNorm{t_3}_K^{s+\frac{1}{2}}}$ \\ \hline
		$[2,1,3]$ & $\chi\bk{\frac{t_2^2}{t_1\Nm_{K\rmod F}\bk{t_3}}}\frac{\FNorm{t_2}_F^{2s}}{\FNorm{t_1}_F^{s+\frac{1}{2}} \FNorm{t_3}_K^{s+\frac{1}{2}}}$ \\ \hline
		$\coset{2,3,2}$ & $\chi\bk{\frac{t_1^2}{t_2}}\frac{\FNorm{t_1}_F^{2s}}{\FNorm{t_2}_F^{s-\frac{1}{2}} \FNorm{t_3}_K}$ \\ \hline
		$\coset{2,1,3,2}$ & $\chi\bk{\frac{t_1\Nm_{K\rmod F}\bk{t_3}}{t_2^2}}\frac{\FNorm{t_1}_F^{s-\frac{1}{2}} \FNorm{t_3}_K^{s-\frac{1}{2}}}{\FNorm{t_2}_F^{2s}}$ \\ \hline
		$[2,3,2,1]$ & $\chi\bk{\frac{t_2}{t_1^2}}\frac{\FNorm{t_2}_F^{s+\frac{1}{2}}}{\FNorm{t_1}_F^{2s} \FNorm{t_3}_K}$ \\ \hline
		$\coset{2,1,3,2,1}$ & $\chi\bk{\frac{\Nm_{K\rmod F}\bk{t_3}}{t_1t_2}}\frac{\FNorm{t_3}_K^{s-\frac{1}{2}}}{\FNorm{t_1}_F^{s-\frac{1}{2}} \FNorm{t_2}_F^{s+\frac{1}{2}}}$ \\ \hline
		$[2,1,3,2,3]$ & $\chi\bk{\frac{t_1}{\Nm_{K\rmod F}\bk{t_3}}} \frac{\FNorm{t_1}_F^{s-\frac{1}{2}}}{\FNorm{t_2}_F \FNorm{t_3}_K^{s-\frac{1}{2}}}$ \\ \hline
		$[2,1,3,2,1,3]$ & $\chi\bk{\frac{t_2}{t_1\Nm_{K\rmod F}\bk{t_3}}}\frac{\FNorm{t_2}_F^{s-\frac{3}{2}}}{\FNorm{t_1}_F^{s-\frac{1}{2}} \FNorm{t_3}_K^{s-\frac{1}{2}}}$ \\ \hline
		$\coset{2,1,3,2,1,3,2}$ & $\chi\bk{\frac{1}{t_2}}\frac{1}{\FNorm{t_1}_F {\FNorm{t_2}_F^{s-\frac{3}{2}}} \FNorm{t_3}_K}$ \\ \hline
		\caption{$w^{-1}\cdot\chi_s$ for $w\in W\bk{P_E,H_E}$, \\ $E=F\times K$}
		\label{Table: Action of W on characters, Quadratic}
	\end{longtable}
	
	In the following table we list the Gindikin-Karpelevich factor $J\bk{w,\chi_s}$ and the poles of the global Gindikin-Karpelevich factor $J\bk{w,\chi_s}$ for $\Real\bk{s}>0$.
	\begin{longtable}{|c|c|c|c|c|c|c|c|c|c|c|}
		\hline
		\multicolumn{2}{|c|}{} & \multicolumn{3}{c|}{$s=\frac{1}{2}$} & \multicolumn{2}{c|}{$s=\frac{3}{2}$} & $s=\frac{5}{2}$ \\ \hline
		$w\in W\bk{P_E,H_E}$ & $J\bk{w,\chi_s}$ & $1$ & $\chi_K$ & $\chi^2=\Id$ & $1$ & $\chi_K$ & $1$ \\ \hline
		\endhead
		$\coset{}$ & $1$ & 0 & 0 & 0 & 0 & 0 & 0 \\ \hline
		$\coset{2}$ & $\frac{\Lfun_F\bk{s+\frac{3}{2},\chi}}{\Lfun_F\bk{s+\frac{5}{2},\chi}}$ & 0 & 0 & 0 & 0 & 0 & 0 \\ \hline
		$\coset{2,1}$ & $\frac{\Lfun_F\bk{s+\frac{1}{2},\chi}}{\Lfun_F\bk{s+\frac{5}{2},\chi}}$ & 1 & 0 & 0 & 0 & 0 & 0 \\ \hline
		$[2,3]$ & $\frac{\Lfun_F\bk{s+\frac{3}{2},\chi}\Lfun_K\bk{s+\frac{1}{2},\chi\circ\Nm}}{\Lfun_F\bk{s+\frac{5}{2},\chi}\Lfun_K\bk{s+\frac{3}{2},\chi\circ\Nm}}$ & 1 & 1 & 0 & 0 & 0 & 0 \\ \hline
		$[2,1,3]$ & $\frac{\Lfun_F\bk{s+\frac{1}{2},\chi}\Lfun_K\bk{s+\frac{1}{2},\chi\circ\Nm}}{\Lfun_F\bk{s+\frac{5}{2},\chi}\Lfun_K\bk{s+\frac{3}{2},\chi\circ\Nm}}$ & 2 & 1 & 0 & 0 & 0 & 0 \\ \hline
		$\coset{2,3,2}$ & $\frac{\Lfun_F\bk{s+\frac{3}{2},\chi} \Lfun_K\bk{s+\frac{1}{2},\chi\circ\Nm}\Lfun_F\bk{s-\frac{1}{2},\chi}}{\Lfun_F\bk{s+\frac{5}{2},\chi}\Lfun_K\bk{s+\frac{3}{2},\chi\circ\Nm}\Lfun_F\bk{s+\frac{1}{2},\chi}}$ & 1 & 1 & 0 & 1 & 0 & 0 \\ \hline
		$\coset{2,1,3,2}$ & $\frac{\Lfun_F\bk{s+\frac{1}{2},\chi}\Lfun_K\bk{s+\frac{1}{2},\chi\circ\Nm}\Lfun_F\bk{2s,\chi^2}}{\Lfun_F\bk{s+\frac{5}{2},\chi}\Lfun_K\bk{s+\frac{3}{2},\chi\circ\Nm}\Lfun_F\bk{2s+1,\chi^2}}$ & 3 & 2 & 1 & 0 & 0 & 0 \\ \hline
		$[2,3,2,1]$ & $\frac{\Lfun_F\bk{s+\frac{3}{2},\chi}\Lfun_K\bk{s+\frac{1}{2},\chi\circ\Nm}\Lfun_F\bk{s-\frac{1}{2},\chi}\Lfun_F\bk{2s,\chi^2}}{\Lfun_F\bk{s+\frac{5}{2},\chi}\Lfun_K\bk{s+\frac{3}{2},\chi\circ\Nm}\Lfun_F\bk{s+\frac{1}{2},\chi}\Lfun_F\bk{2s+1,\chi^2}}$ & 2 & 2 & 1 & 1 & 0 & 0 \\ \hline
		$\coset{2,1,3,2,1}$ & $\frac{\Lfun_F\bk{s-\frac{1}{2},\chi}\Lfun_K\bk{s+\frac{1}{2},\chi\circ\Nm}\Lfun_F\bk{2s,\chi^2}}{\Lfun_F\bk{s+\frac{5}{2},\chi}\Lfun_K\bk{s+\frac{3}{2},\chi\circ\Nm}\Lfun_F\bk{2s+1,\chi^2}}$ & 3 & 2 & 1 & 1 & 0 & 0 \\ \hline
		$[2,1,3,2,3]$ & $\frac{\Lfun_F\bk{s+\frac{1}{2},\chi}\Lfun_K\bk{s-\frac{1}{2},\chi\circ\Nm}\Lfun_F\bk{2s,\chi^2}}{\Lfun_F\bk{s+\frac{5}{2},\chi}\Lfun_K\bk{s+\frac{3}{2},\chi\circ\Nm}\Lfun_F\bk{2s+1,\chi^2}}$  & 3 & 2 & 1 & 1 & 1 & 0 \\ \hline
		$[2,1,3,2,1,3]$ & $\frac{\Lfun_F\bk{s-\frac{1}{2},\chi}\Lfun_K\bk{s-\frac{1}{2},\chi\circ\Nm}\Lfun_F\bk{2s,\chi^2}}{\Lfun_F\bk{s+\frac{5}{2},\chi}\Lfun_K\bk{s+\frac{3}{2},\chi\circ\Nm}\Lfun_F\bk{2s+1,\chi^2}}$ & 3 & 2 & 1 & 2 & 1 & 0 \\ \hline
		$\coset{2,1,3,2,1,3,2}$ & $\frac{\Lfun_F\bk{s-\frac{3}{2},\chi}\Lfun_K\bk{s-\frac{1}{2},\chi\circ\Nm}\Lfun_F\bk{2s,\chi^2}}{\Lfun_F\bk{s+\frac{5}{2},\chi}\Lfun_K\bk{s+\frac{3}{2},\chi\circ\Nm}\Lfun_F\bk{2s+1,\chi^2}}$ & 2 & 2 & 1 & 2 & 1 & 1 \\ \hline
		\caption{Poles of $J\bk{w,\chi_s}$ for $w\in W\bk{P_E,H_E}$, \\ $E=F\times K$}
		\label{Constant Term along the Borel, Quadratic}
	\end{longtable}

	In the following table we list the Gindikin-Karpelevich factor $J\bk{w,\chi,\lambda}$.
	Here $\lambda\bk{t}=\FNorm{t_1}_{F}^{s_1}\FNorm{t_2}_E^{s_2}\FNorm{t_3}_K^{s_3}$.
	\begin{longtable}{|c|c|}
		\hline
		$w\in W\bk{P_E,H_E}$ & $J\bk{w,\chi,\lambda}$ \\ \hline
		\endhead
		$\coset{}$ & $1$ \\ \hline
		$\coset{2}$ & $\frac{\Lfun_F\bk{s_2,\chi}}{\Lfun_F\bk{s_2+1,\chi}}$ \\ \hline
		$\coset{2,1}$ & $\frac{\Lfun_F\bk{s_2,\chi}\Lfun_F\bk{s_1+s_2,\chi}}{\Lfun_F\bk{s_2+1,\chi}\Lfun_F\bk{s_1+s_2+1,\chi}}$ \\ \hline
		$[2,3]$ & $\frac{\Lfun_F\bk{s_2,\chi}\Lfun_K\bk{s_2+s_3,\chi\circ\Nm}}{\Lfun_F\bk{s_2+1,\chi}\Lfun_K\bk{s_2+s_3+1,\chi\circ\Nm}}$ \\ \hline
		$[2,1,3]$ & $\frac{\Lfun_F\bk{s_2,\chi}\Lfun_K\bk{s_2+s_3,\chi\circ\Nm}\Lfun_F\bk{s_1+s_2,\chi}}{\Lfun_F\bk{s_2+1,\chi}\Lfun_K\bk{s_2+s_3+1,\chi\circ\Nm}\Lfun_F\bk{s_1+s_2+1,\chi}}$ \\ \hline
		$\coset{2,3,2}$ & $\frac{\Lfun_F\bk{s_2,\chi}\Lfun_K\bk{s_2+s_3,\chi\circ\Nm}\Lfun_F\bk{s_2+2s_3,\chi}}{\Lfun_F\bk{s_2+1,\chi}\Lfun_K\bk{s_2+s_3+1,\chi\circ\Nm}\Lfun_F\bk{s_2+2s_3+1,\chi}}$ \\ \hline
		$\coset{2,1,3,2}$ & $\frac{\Lfun_F\bk{s_2,\chi}\Lfun_K\bk{s_2+s_3,\chi\circ\Nm}\Lfun_F\bk{s_1+s_2,\chi}\Lfun_F\bk{s_1+2s_2+2s_3,\chi^2}}{\Lfun_F\bk{s_2+1,\chi}\Lfun_K\bk{s_2+s_3+1,\chi\circ\Nm}\Lfun_F\bk{s_1+s_2+1,\chi}\Lfun_F\bk{s_1+2s_2+2s_3+1,\chi^2}}$ \\ \hline
		$[2,3,2,1]$ & $\frac{\Lfun_F\bk{s_2,\chi}\Lfun_K\bk{s_2+s_3,\chi\circ\Nm}\Lfun_F\bk{s_2+2s_3,\chi}\Lfun_F\bk{s_1+2s_2+2s_3,\chi^2}}{\Lfun_F\bk{s_2+1,\chi}\Lfun_K\bk{s_2+s_3+1,\chi\circ\Nm}\Lfun_F\bk{s_2+2s_3+1,\chi}\Lfun_F\bk{s_1+2s_2+2s_3+1,\chi^2}}$ \\ \hline
		$\coset{2,1,3,2,1}$ & $\frac{\Lfun_F\bk{s_2,\chi}\Lfun_K\bk{s_2+s_3,\chi\circ\Nm}\Lfun_F\bk{s_1+s_2,\chi}\Lfun_F\bk{s_1+2s_2+2s_3,\chi^2}\Lfun_F\bk{s_2+2s_3,\chi}}{\Lfun_F\bk{s_2+1,\chi}\Lfun_K\bk{s_2+s_3+1,\chi\circ\Nm}\Lfun_F\bk{s_1+s_2+1,\chi}\Lfun_F\bk{s_1+2s_2+2s_3+1,\chi^2}\Lfun_F\bk{s_2+2s_3,\chi}}$ \\ \hline
		$[2,1,3,2,3]$ & $\frac{\Lfun_F\bk{s_2,\chi}\Lfun_K\bk{s_2+s_3,\chi\circ\Nm}\Lfun_F\bk{s_1+s_2,\chi}\Lfun_F\bk{s_1+2s_2+2s_3,\chi^2}\Lfun_K\bk{s_1+s_2+s_3,\chi\circ\Nm}}{\Lfun_F\bk{s_2+1,\chi}\Lfun_K\bk{s_2+s_3+1,\chi\circ\Nm}\Lfun_F\bk{s_1+s_2+1,\chi}\Lfun_F\bk{s_1+2s_2+2s_3+1,\chi^2}\Lfun_K\bk{s_1+s_2+s_3+1,\chi\circ\Nm}}$ \\ \hline
		$[2,1,3,2,1,3]$ & $\frac{\Lfun_F\bk{s_2,\chi}\Lfun_K\bk{s_2+s_3,\chi\circ\Nm}\Lfun_F\bk{s_1+s_2,\chi}\Lfun_F\bk{s_1+2s_2+2s_3,\chi^2}\Lfun_F\bk{s_2+2s_3,\chi}\Lfun_K\bk{s_1+s_2+s_3,\chi\circ\Nm}}{\Lfun_F\bk{s_2+1,\chi}\Lfun_K\bk{s_2+s_3+1,\chi\circ\Nm}\Lfun_F\bk{s_1+s_2+1,\chi}\Lfun_F\bk{s_1+2s_2+2s_3+1,\chi^2}\Lfun_F\bk{s_2+2s_3,\chi}\Lfun_K\bk{s_1+s_2+s_3+1,\chi\circ\Nm}}$ \\ \hline
		$\coset{2,1,3,2,1,3,2}$ & $\frac{\Lfun_F\bk{s_2,\chi}\Lfun_K\bk{s_2+s_3,\chi\circ\Nm}\Lfun_F\bk{s_1+s_2,\chi}\Lfun_F\bk{s_1+2s_2+2s_3,\chi^2}\Lfun_F\bk{s_2+2s_3,\chi}\Lfun_K\bk{s_1+s_2+s_3,\chi\circ\Nm}\Lfun_F\bk{s_1+s_2+2s_3,\chi}}{\Lfun_F\bk{s_2+1,\chi}\Lfun_K\bk{s_2+s_3+1,\chi\circ\Nm}\Lfun_F\bk{s_1+s_2+1,\chi}\Lfun_F\bk{s_1+2s_2+2s_3+1,\chi^2}\Lfun_F\bk{s_2+2s_3,\chi}\Lfun_K\bk{s_1+s_2+s_3+1,\chi\circ\Nm}\Lfun_F\bk{s_1+s_2+2s_3,\chi}}$ \\ \hline
		\caption{$J\bk{w,\chi,\lambda}$ for $w\in W\bk{P_E,H_E}$, $E=F\times K$}
		\label{Constant Term along the Borel, Quadratic, Gneral Character}
	\end{longtable}
	
	In the following table we list the exponents $\Real\bk{w^{-1}\cdot\chi_s}$ for all $w\in W\bk{P_E,H_E}$.
	\begin{longtable}{|c|c|c|}
		\hline
		$w\in W\bk{P_E,H_E}$ & $s=\frac{1}{2}$ & $s=\frac{3}{2}$ \\ \hline
		\endhead
		$[]$ & $\coset{0,1,0}$ & $\coset{1,0,0}+3\coset{0,1,0}+\Nm_{E\rmod F}\coset{0,0,1}$ \\ \hline
		$[2]$ & $-\coset{0,1,0}$ & $\coset{1,0,0}+\Nm_{E\rmod F}\coset{0,0,1}$ \\ \hline
		$\coset{2,1}$ & $-\coset{1,0,0}-\coset{0,1,0}$ & $-\coset{1,0,0}+\Nm_{E\rmod F}\coset{0,0,1}$ \\ \hline
		$[2,3]$ & $-\coset{0,1,0}-\Nm_{E\rmod F}\coset{0,0,1}$ & $\coset{1,0,0}-\Nm_{E\rmod F}\coset{0,0,1}$ \\ \hline
		$[2,1,3]$ & $-\coset{1,0,0}-\coset{0,1,0}-\Nm_{E\rmod F}\coset{0,0,1}$ & $-\coset{1,0,0}-\Nm_{E\rmod F}\coset{0,0,1}$ \\ \hline
		$[2,3,2]$ & $-\coset{0,1,0}-\Nm_{E\rmod F}\coset{0,0,1}$ & $\coset{1,0,0}-\coset{0,1,0}-\Nm_{E\rmod F}\coset{0,0,1}$ \\ \hline
		$\coset{2,1,3,2}$ & $-\coset{1,0,0}-2\coset{0,1,0}-\Nm_{E\rmod F}\coset{0,0,1}$ & $-\coset{1,0,0}-3\coset{0,1,0}-\Nm_{E\rmod F}\coset{0,0,1}$ \\ \hline
		$[2,3,2,1]$ & $-\coset{1,0,0}-2\coset{0,1,0}-\Nm_{E\rmod F}\coset{0,0,1}$ & $-2\coset{1,0,0}-\coset{0,1,0}-\Nm_{E\rmod F}\coset{0,0,1}$ \\ \hline
		$\coset{2,1,3,2,1}$ & $-\coset{1,0,0}-2\coset{0,1,0}-\Nm_{E\rmod F}\coset{0,0,1}$ & $-2\coset{1,0,0}-3\coset{0,1,0}-\Nm_{E\rmod F}\coset{0,0,1}$ \\ \hline
		$[2,1,3,2,3]$ & $-\coset{1,0,0}-2\coset{0,1,0}-\Nm_{E\rmod F}\coset{0,0,1}$ & $-\coset{1,0,0}-3\coset{0,1,0}-2\Nm_{E\rmod F}\coset{0,0,1}$ \\ \hline
		$[2,1,3,2,1,3]$ & $-\coset{1,0,0}-2\coset{0,1,0}-\Nm_{E\rmod F}\coset{0,0,1}$ & $-2\coset{1,0,0}-3\coset{0,1,0}-2\Nm_{E\rmod F}\coset{0,0,1}$ \\ \hline
		$\coset{2,1,3,2,1,3,2}$ & $-\coset{1,0,0}-\coset{0,1,0}-\Nm_{E\rmod F}\coset{0,0,1}$ & $-2\coset{1,0,0}-3\coset{0,1,0}-2\Nm_{E\rmod F}\coset{0,0,1}$ \\ \hline
		\caption{The exponents $\Real\bk{w^{-1}\cdot\chi_s}$ for $w\in W\bk{P_E,H_E}$, $E=F\times K$}
		\label{Table: Exponents, Quadratic}
	\end{longtable}

	\subsection{Split Case} Assume $E=F\times F\times F$.
	For this case, we denote $$t=h_{\alpha_1}\bk{t_1}h_{\alpha_2}\bk{t_2}h_{\alpha_3}\bk{t_3}h_{\alpha_4}\bk{t_4},$$ where $t_1,t_2,t_3,t_4\in F^\times$.

	In the following table we list $w^{-1}\cdot\chi_s\bk{t}$ for the various $w\in W\bk{P_E,H_E}$ and also the resulting characters $w^{-1}\cdot\chi_s$.
	\begin{longtable}{|c|c|}
		\hline
		$w \in W\bk{P_E,H_E}$ & $w^{-1}\cdot\chi_s\bk{t}$ \\ \hline
		\endhead
		$[]$ & $\chi\bk{t_2}\frac{\FNorm{t_2}^{s+\frac{3}{2}}}{\FNorm{t_1t_3t_4}}$ \\ \hline
		$[2]$ & $\chi\bk{\frac{t_1t_3t_4}{t_2}} \frac{\FNorm{t_1t_3t_4}^{s+\frac{1}{2}}}{\FNorm{t_2}^{s+\frac{3}{2}}}$ \\ \hline
		$[2,1]$ & $\chi\bk{\frac{t_3t_4}{t_1}}\frac{\FNorm{t_3t_4}^{s+\frac{1}{2}}}{\FNorm{t_2}\FNorm{t_1}^{s+\frac{1}{2}}}$ \\ \hline
		$[2,3]$ & $\chi\bk{\frac{t_1t_4}{t_3}}\frac{\FNorm{t_1t_4}^{s+\frac{1}{2}}}{\FNorm{t_2}\FNorm{t_3}^{s+\frac{1}{2}}}$ \\ \hline
		$[2,4]$ & $\chi\bk{\frac{t_1t_3}{t_4}}\frac{\FNorm{t_1t_3}^{s+\frac{1}{2}}}{\FNorm{t_2}\FNorm{t_4}^{s+\frac{1}{2}}}$ \\ \hline
		$[2,1,3]$ & $\chi\bk{\frac{t_2t_4}{t_1t_3}}\frac{\FNorm{t_2}^{s-\frac{1}{2}}\FNorm{t_4}^{s+\frac{1}{2}}}{\FNorm{t_1t_3}^{s+\frac{1}{2}}}$ \\ \hline
		$[2,1,4]$ & $\chi\bk{\frac{t_2t_3}{t_1t_4}}\frac{\FNorm{t_2}^{s-\frac{1}{2}}\FNorm{t_3}^{s+\frac{1}{2}}}{\FNorm{t_1t_4}^{s+\frac{1}{2}}}$ \\ \hline
		$[2,3,4]$ & $\chi\bk{\frac{t_2t_1}{t_3t_4}}\frac{\FNorm{t_2}^{s-\frac{1}{2}}\FNorm{t_1}^{s+\frac{1}{2}}}{\FNorm{t_3t_4}^{s+\frac{1}{2}}}$ \\ \hline
		$[2,1,3,2]$ & $\chi\bk{\frac{t_4^2}{t_2}}\frac{\FNorm{t_4}^{2s}}{\FNorm{t_1t_3} \FNorm{t_2}^{s-\frac{1}{2}}}$ \\ \hline
		$[2,1,4,2]$ & $\chi\bk{\frac{t_3^2}{t_2}}\frac{\FNorm{t_3}^{2s}}{\FNorm{t_1t_4} \FNorm{t_2}^{s-\frac{1}{2}}}$ \\ \hline
		$[2,3,4,2]$ & $\chi\bk{\frac{t_1^2}{t_2}}\frac{\FNorm{t_1}^{2s}}{\FNorm{t_3t_4} \FNorm{t_2}^{s-\frac{1}{2}}}$ \\ \hline
		$[2,1,3,4]$ & $\chi\bk{\frac{t_2^2}{t_1t_3t_4}} \frac{\FNorm{t_2}^{2s}}{\FNorm{t_1t_3t_4}^{s+\frac{1}{2}}}$ \\ \hline
		$[2,3,4,2,1]$ & $\chi\bk{\frac{t_2}{t_1^2}}\frac{\FNorm{t_2}^{s+\frac{1}{2}}}{\FNorm{t_3t_4} \FNorm{t_1}^{2s}}$ \\ \hline
		$[2,1,4,2,3]$ & $\chi\bk{\frac{t_2}{t_3^2}}\frac{\FNorm{t_2}^{s+\frac{1}{2}}}{\FNorm{t_1t_4} \FNorm{t_3}^{2s}}$ \\ \hline
		$[2,1,3,2,4]$ & $\chi\bk{\frac{t_2}{t_4^2}}\frac{\FNorm{t_2}^{s+\frac{1}{2}}}{\FNorm{t_1t_3} \FNorm{t_4}^{2s}}$ \\ \hline
		$[2,1,3,4,2]$ & $\chi\bk{\frac{t_1t_3t_4}{t_2^2}} \frac{\FNorm{t_1t_3t_4}^{s-\frac{1}{2}}}{\FNorm{t_2}^{2s}}$ \\ \hline
		$[2,1,3,4,2,1]$ & $\chi\bk{\frac{t_3t_4}{t_1t_2}}\frac{\FNorm{t_3t_4}^{s-\frac{1}{2}}}{\FNorm{t_2}^{s+\frac{1}{2}}\FNorm{t_1}^{s-\frac{1}{2}}}$ \\ \hline
		$[2,1,3,4,2,3]$ & $\chi\bk{\frac{t_1t_4}{t_2t_3}}\frac{\FNorm{t_1t_4}^{s-\frac{1}{2}}}{\FNorm{t_2}^{s+\frac{1}{2}}\FNorm{t_3}^{s-\frac{1}{2}}}$ \\ \hline
		$[2,1,3,4,2,4]$ & $\chi\bk{\frac{t_1t_3}{t_2t_4}}\frac{\FNorm{t_1t_3}^{s-\frac{1}{2}}}{\FNorm{t_2}^{s+\frac{1}{2}}\FNorm{t_4}^{s-\frac{1}{2}}}$ \\ \hline
		$[2,1,3,4,2,1,3]$ & $\chi\bk{\frac{t_4}{t_1t_3}}\frac{\FNorm{t_4}^{s-\frac{1}{2}}}{\FNorm{t_2}\FNorm{t_1t_3}^{s-\frac{1}{2}}}$ \\ \hline
		$[2,1,3,4,2,1,4]$ & $\chi\bk{\frac{t_3}{t_1t_4}}\frac{\FNorm{t_3}^{s-\frac{1}{2}}}{\FNorm{t_2}\FNorm{t_1t_4}^{s-\frac{1}{2}}}$ \\ \hline
		$[2,1,3,4,2,3,4]$ & $\chi\bk{\frac{t_1}{t_3t_4}}\frac{\FNorm{t_1}^{s-\frac{1}{2}}}{\FNorm{t_2}\FNorm{t_3t_4}^{s-\frac{1}{2}}}$ \\ \hline
		$[2,1,3,4,2,1,3,4]$ & $\chi\bk{\frac{t_2}{t_1t_3t_4}} \frac{\FNorm{t_2}^{s-\frac{3}{2}}}{\FNorm{t_1t_3t_4}^{s-\frac{1}{2}}}$ \\ \hline
		$[2,1,3,4,2,1,3,4,2]$ & $\chi\bk{\frac{1}{t_2}}\frac{1}{\FNorm{t_2}^{s-\frac{3}{2}}\FNorm{t_1t_3t_4}}$ \\ \hline
		\caption{$w^{-1}\cdot\chi_s$ for $w\in W\bk{P_E,H_E}$,  $E=F\times F\times F$}
		\label{Table: Action of W on characters, Split}
	\end{longtable}
	
	In the following table we list the Gindikin-Karpelevich factor $J\bk{w,\chi_s}$ and the poles of the global Gindikin-Karpelevich factor $J\bk{w,\chi_s}$ for $\Real\bk{s}>0$.
	\begin{longtable}{|c|c|c|c|c|c|c|c|c|c|c|}
		\hline
		\multicolumn{2}{|c|}{} & \multicolumn{2}{c|}{$s=\frac{1}{2}$} & $s=\frac{3}{2}$ & $s=\frac{5}{2}$ \\ \hline
		$w\in W\bk{P_E,H_E}$ & $J\bk{w,\chi_s}$ & $\chi=\Id$ & $\chi^2=\Id$ & $\chi=\Id$ & $1$ \\ \hline
		\endhead
		$[]$ & $1$ & 0 & 0 & 0 & 0 \\ \hline
		$[2]$ & $\frac{\Lfun\bk{s+\frac{3}{2},\chi}}{\Lfun\bk{s+\frac{5}{2},\chi}}$ & 0 & 0 & 0 & 0 \\ \hline
		$[2,1]$ & &&&& \\
		$[2,3]$ & $\frac{\Lfun\bk{s+\frac{1}{2},\chi}}{\Lfun\bk{s+\frac{5}{2},\chi}}$ & 1 & 0 & 0 & 0 \\
		$[2,4]$ & &&&& \\ \hline
		$[2,1,3]$ & &&&& \\
		$[2,1,4]$ & $\frac{\Lfun\bk{s+\frac{1}{2},\chi}^2}{\Lfun\bk{s+\frac{3}{2},\chi} \Lfun\bk{s+\frac{5}{2},\chi}}$ & 2 & 0 & 0 & 0 \\
		$[2,3,4]$ & &&&& \\ \hline
		$[2,1,3,2]$ & &&&& \\
		$[2,1,4,2]$ & $\frac{\Lfun\bk{s+\frac{1}{2},\chi}\Lfun\bk{s-\frac{1}{2},\chi}}{\Lfun\bk{s+\frac{5}{2},\chi}\Lfun\bk{s+\frac{3}{2},\chi}}$ & 2 & 0 & 1 & 0 \\
		$[2,3,4,2]$ & &&&& \\ \hline
		$[2,1,3,4]$ & $\frac{\Lfun\bk{s+\frac{1}{2},\chi}^3}{\Lfun\bk{s+\frac{3}{2},\chi}^2 \Lfun\bk{s+\frac{5}{2},\chi}}$ & 3 & 0 & 0 & 0 \\ \hline
		$[2,3,4,2,1]$ & &&&& \\
		$[2,1,4,2,3]$ & $\frac{\Lfun\bk{s-\frac{1}{2},\chi} \Lfun\bk{s+\frac{1}{2},\chi} \Lfun\bk{2s,\chi^2}}{\Lfun\bk{s+\frac{3}{2},\chi} \Lfun\bk{s+\frac{5}{2},\chi}\Lfun\bk{2s+1,\chi^2}}$ & 3 & 1 & 1 & 0 \\
		$[2,1,3,2,4]$ & &&&& \\ \hline
		$[2,1,3,4,2]$ & $\frac{\Lfun\bk{s+\frac{1}{2},\chi}^3\Lfun\bk{2s,\chi^2}}{\Lfun\bk{s+\frac{5}{2},\chi}\Lfun\bk{s+\frac{3}{2},\chi}^2\Lfun\bk{2s+1,\chi^2}}$ & 4 & 1 & 0 & 0 \\ \hline
		$[2,1,3,4,2,1]$ & &&&& \\
		$[2,1,3,4,2,3]$ & $\frac{\Lfun\bk{s-\frac{1}{2},\chi}\Lfun\bk{s+\frac{1}{2},\chi}^2\Lfun\bk{2s,\chi^2}}{\Lfun\bk{s+\frac{3}{2},\chi}^2 \Lfun\bk{s+\frac{5}{2},\chi}\Lfun\bk{2s+1,\chi^2}}$ & 4 & 1 & 1 & 0 \\
		$[2,1,3,4,2,4]$ & &&&& \\ \hline
		$[2,1,3,4,2,1,3]$ & &&&& \\
		$[2,1,3,4,2,1,4]$ & $\frac{\Lfun\bk{s-\frac{1}{2},\chi}^2\Lfun\bk{s+\frac{1}{2},\chi}\Lfun\bk{2s,\chi^2}}{\Lfun\bk{s+\frac{3}{2},\chi}^2 \Lfun\bk{s+\frac{5}{2},\chi}\Lfun\bk{2s+1,\chi^2}}$ & 4 & 1 & 2 & 0 \\
		$[2,1,3,4,2,3,4]$ & &&&& \\ \hline
		$[2,1,3,4,2,1,3,4]$ & $\frac{\Lfun\bk{s-\frac{1}{2},\chi}^3\Lfun\bk{2s,\chi^2}}{\Lfun\bk{s+\frac{3}{2},\chi}^2 \Lfun\bk{s+\frac{5}{2},\chi}\Lfun\bk{2s+1,\chi^2}}$ & 4 & 1 & 3 & 0 \\ \hline
		$[2,1,3,4,2,1,3,4,2]$ &  $\frac{\Lfun\bk{s-\frac{1}{2},\chi}^2\Lfun\bk{s-\frac{3}{2},\chi}\Lfun\bk{2s,\chi^2}}{\Lfun\bk{s+\frac{5}{2},\chi}\Lfun\bk{s+\frac{3}{2},\chi}^2\Lfun\bk{2s+1,\chi^2}}$ & 3 & 1 & 3 & 1 \\ \hline
		\caption{Poles of $J\bk{w,\chi_s}$ for $w\in W\bk{P_E,H_E}$, $E=F\times F\times F$}
		\label{Constant Term along the Borel, Split}
	\end{longtable}
	
	In the following table we list the Gindikin-Karpelevich factor $J\bk{w,\chi,\lambda}$.
	Here $\lambda\bk{t}=\FNorm{t_1}_{F}^{s_1}\FNorm{t_2}_F^{s_2}\FNorm{t_3}_F^{s_3}\FNorm{t_4}_F^{s_4}$.

	{
		\pagestyle{empty}
		\setlength\LTleft{-0.5in}
		\setlength\LTright{-0.5in}
		\begin{longtable}{@{\extracolsep{\fill}}*{2}{|c}|@{}}
			\hline
			$w\in W\bk{P_E,H_E}$ & $J\bk{w,\chi,\lambda}$ \\ \hline
			\endhead
			$[]$ & $1$ \\ \hline
			$[2]$ & $\frac{\Lfun\bk{s_2,\chi}}{\Lfun\bk{s_2+1,\chi}}$ \\ \hline
			$[2,1]$ & $\frac{\Lfun\bk{s_2,\chi}\Lfun\bk{s_1+s_2,\chi}}{\Lfun\bk{s_2+1,\chi}\Lfun\bk{s_1+s_2+1,\chi}}$ \\ \hline
			$[2,3]$ & $\frac{\Lfun\bk{s_2,\chi}\Lfun\bk{s_2+s_3,\chi}}{\Lfun\bk{s_2+1,\chi}\Lfun\bk{s_2+s_3+1,\chi}}$ \\ \hline
			$[2,4]$ & $\frac{\Lfun\bk{s_2,\chi}\Lfun\bk{s_2+s_4,\chi}}{\Lfun\bk{s_2+1,\chi}\Lfun\bk{s_2+s_4+1,\chi}}$ \\ \hline	
			$[2,1,3]$ & $\frac{\Lfun\bk{s_2,\chi}\Lfun\bk{s_1+s_2,\chi}\Lfun\bk{s_2+s_3,\chi}}{\Lfun\bk{s_2+1,\chi}\Lfun\bk{s_1+s_2+1,\chi}\Lfun\bk{s_2+s_3+1,\chi}}$ \\ \hline
			$[2,1,4]$ & $\frac{\Lfun\bk{s_2,\chi}\Lfun\bk{s_1+s_2,\chi}\Lfun\bk{s_2+s_4,\chi}}{\Lfun\bk{s_2+1,\chi}\Lfun\bk{s_1+s_2+1,\chi}\Lfun\bk{s_2+s_4+1,\chi}}$ \\ \hline
			$[2,3,4]$ & $\frac{\Lfun\bk{s_2,\chi}\Lfun\bk{s_2+s_3,\chi}\Lfun\bk{s_2+s_4,\chi}}{\Lfun\bk{s_2+1,\chi}\Lfun\bk{s_2+s_3+1,\chi}\Lfun\bk{s_2+s_4+1,\chi}}$ \\ \hline
			$[2,1,3,2]$ & $\frac{\Lfun\bk{s_2,\chi}\Lfun\bk{s_1+s_2,\chi}\Lfun\bk{s_2+s_3,\chi}\Lfun\bk{s_1+s_2+s_3,\chi}}{\Lfun\bk{s_2+1,\chi}\Lfun\bk{s_1+s_2+1,\chi}\Lfun\bk{s_2+s_3+1,\chi}\Lfun\bk{s_1+s_2+s_3+1,\chi}}$ \\ \hline
			$[2,1,4,2]$ & $\frac{\Lfun\bk{s_2,\chi}\Lfun\bk{s_1+s_2,\chi}\Lfun\bk{s_2+s_4,\chi}\Lfun\bk{s_1+s_2+s_4,\chi}}{\Lfun\bk{s_2+1,\chi}\Lfun\bk{s_1+s_2+1,\chi}\Lfun\bk{s_2+s_4+1,\chi}\Lfun\bk{s_1+s_2+s_4+1,\chi}}$ \\ \hline
			$[2,3,4,2]$ & $\frac{\Lfun\bk{s_2,\chi}\Lfun\bk{s_2+s_3,\chi}\Lfun\bk{s_2+s_4,\chi}\Lfun\bk{s_2+s_3+s_4,\chi}}{\Lfun\bk{s_2+1,\chi}\Lfun\bk{s_2+s_3+1,\chi}\Lfun\bk{s_2+s_4+1,\chi}\Lfun\bk{s_2+s_3+s_4+1,\chi}}$ \\ \hline
			$[2,1,3,4]$ & $\frac{\Lfun\bk{s_2,\chi}\Lfun\bk{s_1+s_2,\chi}\Lfun\bk{s_2+s_3,\chi}\Lfun\bk{s_2+s_4,\chi}}{\Lfun\bk{s_2+1,\chi}\Lfun\bk{s_1+s_2+1,\chi}\Lfun\bk{s_2+s_3+1,\chi}\Lfun\bk{s_2+s_4+1,\chi}}$ \\ \hline
			$[2,3,4,2,1]$ & $\frac{\Lfun\bk{s_2,\chi}\Lfun\bk{s_2+s_3,\chi}\Lfun\bk{s_2+s_4,\chi}\Lfun\bk{s_2+s_3+s_4,\chi}\Lfun\bk{s_1+2s_2+s_3+s_4,\chi^2}}{\Lfun\bk{s_2+1,\chi}\Lfun\bk{s_2+s_3+1,\chi}\Lfun\bk{s_2+s_4+1,\chi}\Lfun\bk{s_2+s_3+s_4+1,\chi}\Lfun\bk{s_1+2s_2+s_3+s_4+1,\chi^2}}$ \\ \hline
			$[2,1,4,2,3]$ & $\frac{\Lfun\bk{s_2,\chi}\Lfun\bk{s_1+s_2,\chi}\Lfun\bk{s_2+s_4,\chi}\Lfun\bk{s_1+s_2+s_4,\chi}\Lfun\bk{s_1+2s_2+s_3+s_4,\chi^2}}{\Lfun\bk{s_2+1,\chi}\Lfun\bk{s_1+s_2+1,\chi}\Lfun\bk{s_2+s_4+1,\chi}\Lfun\bk{s_1+s_2+s_4+1,\chi}\Lfun\bk{s_1+2s_2+s_3+s_4+1,\chi^2}}$ \\ \hline
			$[2,1,3,2,4]$ & $\frac{\Lfun\bk{s_2,\chi}\Lfun\bk{s_1+s_2,\chi}\Lfun\bk{s_2+s_3,\chi}\Lfun\bk{s_1+s_2+s_3,\chi}\Lfun\bk{s_1+2s_2+s_3+s_4,\chi^2}}{\Lfun\bk{s_2+1,\chi}\Lfun\bk{s_1+s_2+1,\chi}\Lfun\bk{s_2+s_3+1,\chi}\Lfun\bk{s_1+s_2+s_3+1,\chi}\Lfun\bk{s_1+2s_2+s_3+s_4+1,\chi^2}}$ \\ \hline
			$[2,1,3,4,2]$ & $\frac{\Lfun\bk{s_2,\chi}\Lfun\bk{s_1+s_2,\chi}\Lfun\bk{s_2+s_3,\chi}\Lfun\bk{s_2+s_4,\chi}\Lfun\bk{s_1+2s_2+s_3+s_4,\chi^2}}{\Lfun\bk{s_2+1,\chi}\Lfun\bk{s_1+s_2+1,\chi}\Lfun\bk{s_2+s_3+1,\chi}\Lfun\bk{s_2+s_4+1,\chi}\Lfun\bk{s_1+2s_2+s_3+s_4+1,\chi^2}}$ \\ \hline
			$[2,1,3,4,2,1]$ & $\frac{\Lfun\bk{s_2,\chi}\Lfun\bk{s_1+s_2,\chi}\Lfun\bk{s_2+s_3,\chi}\Lfun\bk{s_2+s_4,\chi}\Lfun\bk{s_1+2s_2+s_3+s_4,\chi^2}\Lfun\bk{s_2+s_3+s_4,\chi}}{\Lfun\bk{s_2+1,\chi}\Lfun\bk{s_1+s_2+1,\chi}\Lfun\bk{s_2+s_3+1,\chi}\Lfun\bk{s_2+s_4+1,\chi}\Lfun\bk{s_1+2s_2+s_3+s_4+1,\chi^2}\Lfun\bk{s_2+s_3+s_4+1,\chi}}$ \\ \hline
			$[2,1,3,4,2,3]$ & $\frac{\Lfun\bk{s_2,\chi}\Lfun\bk{s_1+s_2,\chi}\Lfun\bk{s_2+s_3,\chi}\Lfun\bk{s_2+s_4,\chi}\Lfun\bk{s_1+2s_2+s_3+s_4,\chi^2}\Lfun\bk{s_1+s_2+s_4,\chi}}{\Lfun\bk{s_2+1,\chi}\Lfun\bk{s_1+s_2+1,\chi}\Lfun\bk{s_2+s_3+1,\chi}\Lfun\bk{s_2+s_4+1,\chi}\Lfun\bk{s_1+2s_2+s_3+s_4+1,\chi^2}\Lfun\bk{s_1+s_2+s_4+1,\chi}}$ \\ \hline
			$[2,1,3,4,2,4]$ & $\frac{\Lfun\bk{s_2,\chi}\Lfun\bk{s_1+s_2,\chi}\Lfun\bk{s_2+s_3,\chi}\Lfun\bk{s_2+s_4,\chi}\Lfun\bk{s_1+2s_2+s_3+s_4,\chi^2}\Lfun\bk{s_1+s_2+s_3,\chi}}{\Lfun\bk{s_2+1,\chi}\Lfun\bk{s_1+s_2+1,\chi}\Lfun\bk{s_2+s_3+1,\chi}\Lfun\bk{s_2+s_4+1,\chi}\Lfun\bk{s_1+2s_2+s_3+s_4+1,\chi^2}\Lfun\bk{s_1+s_2+s_3+1,\chi}}$ \\ \hline
			$[2,1,3,4,2,1,3]$ & $\frac{\Lfun\bk{s_2,\chi}\Lfun\bk{s_1+s_2,\chi}\Lfun\bk{s_2+s_3,\chi}\Lfun\bk{s_2+s_4,\chi}\Lfun\bk{s_1+2s_2+s_3+s_4,\chi^2}\Lfun\bk{s_2+s_3+s_4,\chi}\Lfun\bk{s_1+s_2+s_4,\chi}}{\Lfun\bk{s_2+1,\chi}\Lfun\bk{s_1+s_2+1,\chi}\Lfun\bk{s_2+s_3+1,\chi}\Lfun\bk{s_2+s_4+1,\chi}\Lfun\bk{s_1+2s_2+s_3+s_4+1,\chi^2}\Lfun\bk{s_2+s_3+s_4+1,\chi}\Lfun\bk{s_1+s_2+s_4+1,\chi}}$ \\ \hline
			$[2,1,3,4,2,1,4]$ & $\frac{\Lfun\bk{s_2,\chi}\Lfun\bk{s_1+s_2,\chi}\Lfun\bk{s_2+s_3,\chi}\Lfun\bk{s_2+s_4,\chi}\Lfun\bk{s_1+2s_2+s_3+s_4,\chi^2}\Lfun\bk{s_1+s_2+s_3,\chi}\Lfun\bk{s_2+s_3+s_4,\chi}}{\Lfun\bk{s_2+1,\chi}\Lfun\bk{s_1+s_2+1,\chi}\Lfun\bk{s_2+s_3+1,\chi}\Lfun\bk{s_2+s_4+1,\chi}\Lfun\bk{s_1+2s_2+s_3+s_4+1,\chi^2}\Lfun\bk{s_1+s_2+s_3+1,\chi}\Lfun\bk{s_2+s_3+s_4+1,\chi}}$ \\ \hline
			$[2,1,3,4,2,3,4]$ & $\frac{\Lfun\bk{s_2,\chi}\Lfun\bk{s_1+s_2,\chi}\Lfun\bk{s_2+s_3,\chi}\Lfun\bk{s_2+s_4,\chi}\Lfun\bk{s_1+2s_2+s_3+s_4,\chi^2}\Lfun\bk{s_1+s_2+s_3,\chi}\Lfun\bk{s_1+s_2+s_4,\chi}}{\Lfun\bk{s_2+1,\chi}\Lfun\bk{s_1+s_2+1,\chi}\Lfun\bk{s_2+s_3+1,\chi}\Lfun\bk{s_2+s_4+1,\chi}\Lfun\bk{s_1+2s_2+s_3+s_4+1,\chi^2}\Lfun\bk{s_1+s_2+s_4+1,\chi}\Lfun\bk{s_1+s_2+s_4+1,\chi}}$ \\ \hline
			$[2,1,3,4,2,1,3,4]$ & $\frac{\Lfun\bk{s_2,\chi}\Lfun\bk{s_1+s_2,\chi}\Lfun\bk{s_2+s_3,\chi}\Lfun\bk{s_2+s_4,\chi}\Lfun\bk{s_1+2s_2+s_3+s_4,\chi^2}\Lfun\bk{s_1+s_2+s_3,\chi}\Lfun\bk{s_2+s_3+s_4,\chi}\Lfun\bk{s_1+s_2+s_4,\chi}}{\Lfun\bk{s_2+1,\chi}\Lfun\bk{s_1+s_2+1,\chi}\Lfun\bk{s_2+s_3+1,\chi}\Lfun\bk{s_2+s_4+1,\chi}\Lfun\bk{s_1+2s_2+s_3+s_4+1,\chi^2}\Lfun\bk{s_1+s_2+s_3+1,\chi}\Lfun\bk{s_2+s_3+s_4+1,\chi}\Lfun\bk{s_1+s_2+s_4+1,\chi}}$ \\ \hline
			$[2,1,3,4,2,1,3,4,2]$ & $\frac{\Lfun\bk{s_2,\chi}\Lfun\bk{s_1+s_2,\chi}\Lfun\bk{s_2+s_3,\chi}\Lfun\bk{s_2+s_4,\chi}\Lfun\bk{s_1+2s_2+s_3+s_4,\chi^2}\Lfun\bk{s_1+s_2+s_3,\chi}\Lfun\bk{s_2+s_3+s_4,\chi}\Lfun\bk{s_1+s_2+s_4,\chi}\Lfun\bk{s_1+s_2+s_3+s_4,\chi}}{\Lfun\bk{s_2+1,\chi}\Lfun\bk{s_1+s_2+1,\chi}\Lfun\bk{s_2+s_3+1,\chi}\Lfun\bk{s_2+s_4+1,\chi}\Lfun\bk{s_1+2s_2+s_3+s_4+1,\chi^2}\Lfun\bk{s_1+s_2+s_3+1,\chi}\Lfun\bk{s_2+s_3+s_4+1,\chi}\Lfun\bk{s_1+s_2+s_4+1,\chi}\Lfun\bk{s_1+s_2+s_3+s_4+1,\chi}}$ \\ \hline
			\caption{$J\bk{w,\chi,\lambda}$ for $w\in W\bk{P_E,H_E}$, \\ $E=F\times F\times F$}
			\label{Constant Term along the Borel, Split, Gneral Character}
		\end{longtable}
		\thispagestyle{empty}
	}
	
	In the following table we list the exponents $\Real\bk{w^{-1}\cdot\chi_s}$ for all $w\in W\bk{P_E,H_E}$.
	\begin{longtable}{|c|c|c|c|c|c|c|c|c|c|c|}
		\hline
		$w\in W\bk{P_E,H_E}$ & $s=\frac{1}{2}$ & $s=\frac{3}{2}$ \\ \hline
		\endhead
		$[]$ & $\coset{0,1,0,0}$ & $\coset{1,3,1,1}$ \\ \hline
		$[2]$ & $-\coset{0,1,0,0}$ & $\coset{1,0,1,1}$ \\ \hline
		$[2,1]$ & $-\coset{1,1,0,0}$ & $\coset{-1,0,1,1}$ \\ \hline
		$[2,3]$ & $-\coset{0,1,1,0}$ & $\coset{1,0,-1,1}$ \\ \hline
		$[2,4]$ & $-\coset{0,1,0,1}$ & $\coset{1,0,1,-1}$ \\ \hline
		$[2,1,3]$ & $-\coset{1,1,1,0}$ & $\coset{-1,0,-1,1}$ \\ \hline
		$[2,1,4]$ & $-\coset{1,1,0,1}$ & $\coset{-1,0,1,-1}$ \\ \hline
		$[2,3,4]$ & $-\coset{0,1,1,1}$ & $\coset{1,0,-1,-1}$ \\ \hline
		$[2,1,3,2]$ & $-\coset{1,1,1,0}$ & $\coset{-1,-1,-1,1}$ \\ \hline
		$[2,1,4,2]$ & $-\coset{1,1,0,1}$ & $\coset{-1,-1,1,-1}$ \\ \hline
		$[2,3,4,2]$ & $-\coset{0,1,1,1}$ & $\coset{1,-1,-1,-1}$ \\ \hline
		$[2,1,3,4]$ & $-\coset{1,1,1,1}$ & $\coset{-1,0,-1,-1}$ \\ \hline
		$[2,3,4,2,1]$ & $-\coset{1,1,1,1}$ & $\coset{-2,-1,-1,-1}$ \\ \hline
		$[2,1,4,2,3]$ & $-\coset{1,1,1,1}$ & $\coset{-1,-1,-2,-1}$ \\ \hline
		$[2,1,3,2,4]$ & $-\coset{1,1,1,1}$ & $\coset{-1,-1,-1,-2}$ \\ \hline
		$[2,1,3,4,2]$ & $-\coset{1,2,1,1}$ & $\coset{-1,-3,-1,-1}$ \\ \hline
		$[2,1,3,4,2,1]$ & $-\coset{1,2,1,1}$ & $\coset{-2,-3,-1,-1}$ \\ \hline
		$[2,1,3,4,2,3]$ & $-\coset{1,2,1,1}$ & $\coset{-1,-3,-2,-1}$ \\ \hline
		$[2,1,3,4,2,4]$ & $-\coset{1,2,1,1}$ & $\coset{-1,-3,-1,-2}$ \\ \hline
		$[2,1,3,4,2,1,3]$ & $-\coset{1,2,1,1}$ & $\coset{-2,-3,-2,-1}$ \\ \hline
		$[2,1,3,4,2,1,4]$ & $-\coset{1,2,1,1}$ & $\coset{-2,-3,-1,-2}$ \\ \hline
		$[2,1,3,4,2,3,4]$ & $-\coset{1,2,1,1}$ & $\coset{-1,-3,-2,-2}$ \\ \hline
		$[2,1,3,4,2,1,3,4]$ & $-\coset{1,2,1,1}$ & $\coset{-2,-3,-2,-2}$ \\ \hline
		$[2,1,3,4,2,1,3,4,2]$ & $-\coset{1,1,1,1}$ & $\coset{-2,-3,-2,-2}$ \\ \hline
		\caption{The exponents $\Real\bk{w^{-1}\cdot\chi_s}$ for $w\in W\bk{P_E,H_E}$, $E=F\times F\times F$}
		\label{Table: Exponents, Split}
	\end{longtable}
\end{landscape}

\end{document}